\documentclass[oneside,english,reqno]{amsart}
\usepackage[T1]{fontenc}
\usepackage[latin9]{inputenc}
\usepackage{color}
\usepackage{babel}
\usepackage{array}
\usepackage{float}
\usepackage{mathrsfs}
\usepackage{amsthm}
\usepackage{amsbsy}
\usepackage{amstext}
\usepackage{amssymb}
\usepackage{graphicx}
\usepackage{esint}
\PassOptionsToPackage{normalem}{ulem}
\usepackage{ulem}
\usepackage[unicode=true,pdfusetitle,
 bookmarks=true,bookmarksnumbered=false,bookmarksopen=false,
 breaklinks=false,pdfborder={0 0 1},backref=false,colorlinks=false]
 {hyperref}
\hypersetup{
 colorlinks=true,citecolor=blue,linkcolor=blue,linktocpage=true}

\makeatletter

\providecommand{\tabularnewline}{\\}

\numberwithin{equation}{section}
\numberwithin{figure}{section}
\theoremstyle{plain}
\newtheorem{thm}{\protect\theoremname}[section]
  \theoremstyle{plain}
  \newtheorem{lem}[thm]{\protect\lemmaname}
  \theoremstyle{plain}
  \newtheorem{cor}[thm]{\protect\corollaryname}
  \theoremstyle{definition}
  \newtheorem{defn}[thm]{\protect\definitionname}
  \theoremstyle{remark}
  \newtheorem{rem}[thm]{\protect\remarkname}
  \theoremstyle{definition}
  \newtheorem{example}[thm]{\protect\examplename}
  \theoremstyle{plain}
  \newtheorem{prop}[thm]{\protect\propositionname}

\providecommand{\MR}[1]{}

\theoremstyle{remark}

\renewcommand{\section}{%
\@startsection{section}{1}%
  \z@{.7\linespacing\@plus\linespacing}{.5\linespacing}%
  {\normalfont\scshape\centering\bfseries}}

\renewcommand{\subsection}{%
\@startsection{subsection}{2}%
  \z@{.5\linespacing\@plus.7\linespacing}{.5\linespacing}%
  {\normalfont\bfseries}}

\makeatother

  \providecommand{\corollaryname}{Corollary}
  \providecommand{\definitionname}{Definition}
  \providecommand{\examplename}{Example}
  \providecommand{\lemmaname}{Lemma}
  \providecommand{\propositionname}{Proposition}
  \providecommand{\remarkname}{Remark}
\providecommand{\theoremname}{Theorem}

\begin{document}

\title{Restrictions and Extensions of Semibounded Operators}

\author{Palle Jorgensen, Steen Pedersen, and Feng Tian}

\address{(Palle E.T. Jorgensen) Department of Mathematics, The University
of Iowa, Iowa City, IA 52242-1419, U.S.A. }

\email{jorgen@math.uiowa.edu }

\urladdr{http://www.math.uiowa.edu/\textasciitilde{}jorgen/}

\address{(Steen Pedersen) Department of Mathematics, Wright State University,
Dayton, OH 45435, U.S.A. }

\email{steen@math.wright.edu }

\urladdr{http://www.wright.edu/\textasciitilde{}steen.pedersen/}

\address{(Feng Tian) Department of Mathematics, Wright State University, Dayton,
OH 45435, U.S.A. }

\email{feng.tian@wright.edu }

\urladdr{http://www.wright.edu/\textasciitilde{}feng.tian/}


\maketitle
\begin{center}
To the memory of William B. Arveson.
\par\end{center}
\begin{abstract}
We study restriction and extension theory for semibounded Hermitian
operators in the Hardy space $\mathscr{H}_{2}$ of analytic functions
on the disk $\mathbb{D}$. Starting with the operator ${\displaystyle z\frac{d}{dz}}$,
we show that, for every choice of a closed subset $F\subset\mathbb{T}=\partial\mathbb{D}$
of measure zero, there is a densely defined Hermitian restriction
of ${\displaystyle z\frac{d}{dz}}$ corresponding to boundary functions
vanishing on $F$. For every such restriction operator, we classify
all its selfadjoint extension, and for each we present a complete
spectral picture. 

We prove that different sets $F$ with the same cardinality can lead
to quite different boundary-value problems, inequivalent selfadjoint
extension operators, and quite different spectral configurations.
As a tool in our analysis, we prove that the von Neumann deficiency
spaces, for a fixed set $F$, have a natural presentation as reproducing
kernel Hilbert spaces, with a Hurwitz zeta-function, restricted to
$F\times F$, as reproducing kernel. 
\end{abstract}
\tableofcontents{}

\section{Introduction\label{sec:Intro}}

In this paper, we study a model for families of semibounded but unbounded
selfadjoint operators in Hilbert space. It is of interest in our understanding
of the spectral theory of selfadjoint operators arising from extension
of a single semibounded Hermitian operator with dense domain. Up to
unitary equivalence, the model represents a variety of spectral configurations
of interest in the study of special functions, in the theory of Toeplitz
operators, and in applications to quantum mechanics, and to signal
processing.

We study a duality between restrictions and extensions of semibounded
(unbounded) operators. Since the spectrum of a selfadjoint semibounded
operator is contained in a half-line, it is useful to work with spaces
of analytic functions. To narrow the field down to a manageable scope,
we pick as Hilbert space the familiar Hardy space $\mathscr{H}_{2}$
of analytic functions on the complex disk $\mathbb{D}$ with square-summable
coefficients. (As a Hilbert space, of course $\mathscr{H}_{2}$ is
a copy of $l^{2}(\mathbb{N}_{0})$, but $l^{2}$ does not capture
all the harmonic analysis of the Hardy space $\mathscr{H}_{2}$.) 

Traditionally, the study of semibounded operators is viewed merely
as a special case of the wider context of Hermitian, or selfadjoint
operators. But \emph{semibounded} does suggest that some measure (in
this case, a spectral resolution) is supported in a halfline, say
$[0,\infty)$, and thus, this in turn suggests analytic continuation,
and Hilbert spaces of analytic functions (see, e.g., \cite{ABK02}).
In this paper, we take seriously this idea.

Now the operator ${\displaystyle H:=z\frac{d}{dz}}$ is selfadjoint
on its natural domain in $\mathscr{H}_{2}$ with spectrum $\mathbb{N}_{0}$
(the natural numbers including $0$). We show (section \ref{sec:Hardy})
that, for every measure-zero closed subset $F$ of the circle $\mathbb{T}=\partial\mathbb{D}$,
$H\:(={\displaystyle z\frac{d}{dz}})$ in $\mathscr{H}_{2}$ has a
well defined and densely defined Hermitian restriction operator $L_{F}$,
and we find its selfadjoint extensions. They are indexed by the unitary
operators in a \emph{reproducing kernel Hilbert space} (RKHS) of functions
on $F$, where the reproducing kernel in turn is a restriction of
Hurwitz\textquoteright{}s zeta-function.

This RKHS-feature in the boundary analysis is one way that the boundary-value
problems in Hilbert spaces of analytic functions are different from
more traditional two sided boundary-value problems; see e.g., \cite{JPT11-1,JPT11-2}.
Two-sided boundary-value problems may be attacked with an integration
by parts, or, in higher real dimensions, with Greens-Gauss-Stokes.
By contrast, boundary-value theory is Hilbert spaces of analytic functions
must be studied with the use of different tools.

While it is possible, for every subset $F$ of $\mathbb{T}$ (closed,
measure zero), to write down parameters for all the selfadjoint extensions
of $L_{F}$, to find more explicit formulas, and to compute spectra,
it is helpful to first analyze the special case when the set $F$
is finite. Even the special case when $F$ is a singleton is of interest.
We begin with a study of this in sections \ref{sec:sbdd} and \ref{sec:spH}
below. 

More generally, we show in section \ref{sec:Hardy} that, when $F$
is finite, that then the corresponding Hermitian restriction operator
$L_{F}$ has deficiency indices $(m,m)$ where $m$ is the cardinality
of $F$. Moreover we prove that the variety of all selfadjoint extensions
of $L_{F}$ will then be indexed (bijectively) by a compact Lie group
$G(F)$ depending on $F$.

There is a number of differences between boundary value problems in
$L^{2}$-spaces of functions in real domains, and in Hilbert spaces
of analytic functions. 

Our present study is confined here to one complex variable, and to
the Hardy-space $\mathscr{H}_{2}$ on the disk $\mathbb{D}$ in the
complex plane. The occurrence of the above mentioned family of compact
Lie groups is one way our analysis of extensions and restrictions
of operators is different in the complex domain.

There are others (see section \ref{sec:Hardy} for details). Here
we outline one such striking difference:

Recall that for a Hermitian partial differential operators $P$ acting
on test functions in bounded open domains $\Omega\subset\mathbb{R}^{n}$,
the adjoint operators will again act as PDOs. 

Specifically, in studying boundary value problems in the Hilbert space
$L^{2}(\Omega)$, initially one takes a given Hermitian $P$ to be
defined on a Schwartz space of functions on $\Omega$, vanishing with
all their derivatives on $\partial\Omega$. This realization of $P$
is called the minimal operator, $P_{min}$ for specificity, and it
is $L^{2}$-Hermitian with dense domain in $L^{2}(\Omega)$. 

The adjoint of $P_{min}$ is denoted $P^{*}$ and it is defined relative
to the inner product in the Hilbert space $L^{2}(\Omega)$. And this
adjoint is the maximal operator $P_{max}$. The crucial fact is that
$P_{max}$ is again a (partial) differential operator. Indeed it is
the operator $P$ acting on the domain of functions $f\in L^{2}(\Omega)$
such that $Pf$ is again in $L^{2}(\Omega)$, where the meaning of
$Pf$ in \emph{the weak sense} of distributions. This follows conventions
of L. Schwartz, and K. O. Friedrichs; see e.g., \cite{DS88b,Gru09,AB09}.

Turning now to the complex case, we study here the operator ${\displaystyle Q:=z\frac{d}{dz}}$
in the Hardy-space $\mathscr{H}_{2}$ on the disk $\mathbb{D}$, and
its realization as a Hermitian operator with domain equal to one in
a family of suitable dense linear subspaces in $\mathscr{H}_{2}$.
These Hermitian operators will have equal deficiency spaces (in the
sense of von Neumann), and they will depend on conditions assigned
on chosen closed subsets $F\subset\partial\mathbb{D}$), of (angular)
measure 0.

When such a subset $F$ is given, let $Q_{F}$ be the corresponding
restriction operator. But now the adjoint operators $Q_{F}^{*}$,
defined relative to the inner product in $\mathscr{H}_{2}$, turn
out no longer to be differential operators; they are different operators.
The nature of these operators $Q_{F}^{*}$ is studied in sect \ref{sub:LadH}.

In section \ref{sec:Hardy}, we study the operator ${\displaystyle z\frac{d}{dz}}$
in the Hardy space $\mathscr{H}_{2}$ of the disk $\mathbb{D}$, one
such Hermitian operator $L_{F}$ for each closed subset $F\subset\partial\mathbb{D}$
of zero angular measure. For functions $f\in\mathscr{D}(L_{F}^{*})$
we have boundary values $\tilde{f}$, i.e., extensions from $\mathbb{D}$
to $\overline{\mathbb{D}}$ (= closure). But then the function ${\displaystyle z\frac{d}{dz}f}$
may have simple poles at the points $\zeta\in F$. We show (Corollary
\ref{cor:BadHF}) that the contribution to $\boldsymbol{B}_{F}(f,f)$
(see (\ref{eq:BoundaryForm})) is
\begin{equation}
\boldsymbol{B}_{F}(f,f)_{\zeta}=\Im\left(C_{\zeta}(f)\overline{\tilde{f}(\zeta)}\right)\label{eq:BFLad}
\end{equation}
where $\Im$ stands for imaginary part, and $C_{\zeta}(f)$ is the
residue of the meromorphic function ${\displaystyle z\mapsto z\left(\frac{d}{dz}f\right)\left(z\right)}$
at the pole $z=\zeta$ ($\in F\subset\partial\mathbb{D}$.)

Note that (\ref{eq:BFLad}), for the boundary form in the study of
deficiency indices in the Hardy space $\mathscr{H}_{2}$, contrasts
sharply with the more familiar analogous formula for the boundary
value problems in the real case, i.e., on an interval.

In section \ref{sec:hd}, we consider a higher dimensional version
of the Hardy-Hilbert space $\mathscr{H}_{2}$, the Arveson-Drury space
$\mathscr{H}_{d}^{(AD)}$, $d>1$. While the case $d>1$ does have
a number of striking parallels with $d=1$, there are some key differences.
The reason for the parallels is that the reproducing kernel, the Szegö
kernel for $\mathscr{H}_{2}$ extends from one complex dimension to
$d>1$ almost verbatim, \cite{Arv98,Dru78}.

A central theme in our paper ($d=1$) is showing that the study of
von Neumann boundary theory for Hermitian operators translates into
a geometric analysis on the boundary of the disk $\mathbb{D}$ in
one complex dimension, so on the circle $\partial\mathbb{D}$. 

We point out in sect \ref{sec:hd} that multivariable operator theory
is more subtle. Indeed, Arveson proved \cite[Coroll 2]{Arv98} that
the Hilbert norm in $\mathscr{H}_{d}^{(AD)}$, $d>1$, cannot be represented
by a Borel measure on $\mathbb{C}^{d}$. So, in higher dimension,
the question of \textquotedbl{}geometric boundary\textquotedbl{} is
much more subtle.

\subsection{Unbounded Operators}

Before passing to the main theorems in the paper, we recall a classification
theorem of von Neumann which will be used. Starting with a fixed  Hermitian
operator with dense domain in Hilbert space, and equal indices, von
Neumann\textquoteright{}s classification (\cite{vN32a,vN32b,AG93,DS88b,Kat95})
shows that the variety of all selfadjoint extensions, and their spectral
theory, can be understood and classified in the language of an associated
boundary form, and the set of all partial isometries between a pair
of deficiency-spaces. It has numerous applications. One significance
of the result lies in the fact that the two deficiency-spaces typically
have a small dimension, or allow for a reduction, reflect an underlying
geometry, and they are computable.
\begin{lem}[see e.g. \cite{DS88b}]
\label{lem:vN def-space}Let $L$ be a \textcolor{black}{closed}
Hermitian operator with dense domain $\mathscr{D}_{0}$ in a Hilbert
space. Set 
\begin{alignat}{1}
\mathscr{D}_{\pm} & =\{\psi_{\pm}\in dom(L^{*})\left.\right|L^{*}\psi_{\pm}=\pm i\psi_{\pm}\}\nonumber \\
\mathscr{C}(L) & =\{U:\mathscr{D}_{+}\rightarrow\mathscr{D}_{-}\left.\right|U^{*}U=P_{\mathscr{D}_{+}},UU^{*}=P_{\mathscr{D}_{-}}\}\label{eq:vN1}
\end{alignat}
where $P_{\mathscr{D}_{\pm}}$ denote the respective projections.
Set 
\[
\mathscr{E}(L)=\{S\left.\right|L\subseteq S,S^{*}=S\}.
\]
Then there is a bijective correspondence between $\mathscr{C}(L)$
and $\mathscr{E}(L)$, given as follows: 

If $U\in\mathscr{C}(L)$, and let $L_{U}$ be the restriction of $L^{*}$
to 
\begin{equation}
\{\varphi_{0}+f_{+}+Uf_{+}\left.\right|\varphi_{0}\in\mathscr{D}_{0},f_{+}\in\mathscr{D}_{+}\}.\label{eq:vN2}
\end{equation}
Then $L_{U}\in\mathscr{E}(L)$, and conversely every $S\in\mathscr{E}(L)$
has the form $L_{U}$ for some $U\in\mathscr{C}(L)$. With $S\in\mathscr{E}(L)$,
take 
\begin{equation}
U:=(S-iI)(S+iI)^{-1}\left.\right|_{\mathscr{D}_{+}}\label{eq:vN3}
\end{equation}
and note that 
\begin{enumerate}
\item $U\in\mathscr{C}(L)$, and
\item $S=L_{U}$.
\end{enumerate}

Vectors $f\in dom(L^{*})$ admit a unique decomposition 
\begin{equation}
f=\varphi_{0}+f_{+}+f_{-}\label{eq:vNdep}
\end{equation}
where $\varphi_{0}\in dom(L)$, and $f_{\pm}\in\mathscr{D}_{\pm}$.
For the boundary-form $\mathbf{B}(\cdot,\cdot)$, we have
\begin{alignat}{1}
\mathbf{B}(f,f) & =\frac{1}{2i}\left(\left\langle L^{*}f,f\right\rangle -\left\langle f,L^{*}f\right\rangle \right)\nonumber \\
 & =\left\Vert f_{+}\right\Vert ^{2}-\left\Vert f_{-}\right\Vert ^{2}.\label{eq:BoundaryForm}
\end{alignat}
\uline{Note}, the sesquilinear form $\boldsymbol{B}$ in (\ref{eq:BoundaryForm})
is the boundary form referenced in (\ref{eq:BFLad}) above.

\end{lem}
\begin{proof}[Proof sketch (Lemma \ref{lem:vN def-space})]
 We refer to the cited references for details. The key step in the
verification of formula (\ref{eq:BoundaryForm}) for the boundary
form $\boldsymbol{B}(f,f)$, $f\in dom(L^{*})$, is as follows: Let
$f=\varphi_{0}+f_{+}+f_{-}$. After each of the two terms $\left\langle L^{*}f,f\right\rangle $
and $\left\langle f,L^{*}f\right\rangle $ are computed, we find cancellation
upon subtraction, and only the two $\left\Vert f_{\pm}\right\Vert ^{2}$
terms survive; specifically: 
\begin{eqnarray*}
 &  & \left\langle L^{*}f,f\right\rangle -\left\langle f,L^{*}f\right\rangle \\
 & = & i\left(\left\Vert f_{+}\right\Vert ^{2}-\left\Vert f_{-}\right\Vert ^{2}\right)-(-i)\left(\left\Vert f_{+}\right\Vert ^{2}-\left\Vert f_{-}\right\Vert ^{2}\right)\\
 & = & 2i\left(\left\Vert f_{+}\right\Vert ^{2}-\left\Vert f_{-}\right\Vert ^{2}\right).
\end{eqnarray*}

\end{proof}
While there are earlier studies of boundary forms (in the sense of
(\ref{eq:BoundaryForm})) in the context of Sturm-Liouville operators,
and Hermitian PDOs in bounded domains in $\mathbb{R}^{n}$, e.g.,
\cite{BMT11,BL10}, there appear not to be prior analogues of this
in Hilbert spaces of analytic functions in bounded complex domains.

\textbf{Terminology. }We shall refer to eq. (\ref{eq:vNdep}) as the
von \emph{Neumann decomposition}; and to the classification of the
family of all selfadjoint extensions (see (\ref{eq:vN2}) \& (\ref{eq:vN3}))
as the \emph{von Neumann classification}. 
\begin{lem}
\label{lem:vNext}~
\begin{enumerate}
\item \label{enu:vNa}Consider the von Neumann decomposition 
\begin{equation}
\mathscr{D}(L^{*})=\mathscr{D}(L)+\mathscr{D}_{+}+\mathscr{D}_{-}\label{eq:DLs}
\end{equation}
in (\ref{eq:vNdep}); then the boundary form $\boldsymbol{B}(f,g)$
vanishes if one of the vectors $f$ or $g$ from $\mathscr{D}(L^{*})$
is in $\mathscr{D}(L)$. 
\item \label{enu:vNb}Every subspace $S\subset\mathscr{D}(L^{*})$ such
that $\boldsymbol{B}(f,f)=0$, $\forall f\in S$, is the graph of
a partial isometry from $\mathscr{D}_{+}$ into $\mathscr{D}_{-}$.
\item \label{enu:vNc}Introducing the inner product
\begin{equation}
\left\langle f,g\right\rangle _{*}=\left\langle f,g\right\rangle +\left\langle L^{*}f,L^{*}g\right\rangle \label{eq:gInner}
\end{equation}
$f,g\in\mathscr{D}(L^{*})$, we note that the three terms in the decomposition
(\ref{eq:DLs}) are mutually orthogonal w.r.t. the inner product $\left\langle \cdot,\cdot\right\rangle _{*}$
in (\ref{eq:gInner}).
\end{enumerate}
\end{lem}
\begin{proof}
All assertions (\ref{enu:vNa})-(\ref{enu:vNc}) follow from a direct
computation, and use of the definitions. 

It follows from (\ref{enu:vNb}) in Lemma \ref{lem:vNext}, that the
boundary form $\boldsymbol{B}(\cdot,\cdot)$ passes to the quotient
$\left(\mathscr{D}(L^{*})/\mathscr{D}(L)\right)\times\left(\mathscr{D}(L^{*})/\mathscr{D}(L)\right)$. 
\end{proof}

\subsection{Graphs of Partial Isometries $\mathscr{D}_{+}\longrightarrow\mathscr{D}_{-}$}

In section \ref{sec:Hardy}, we will compute the partial isometries
between defect spaces in a Hardy space $\mathscr{H}_{2}$ in terms
of boundary values and residues. But we begin with axioms of the underlying
geometry in Hilbert space.
\begin{lem}
\label{lem:Y23}Let $L$ be a Hermitian symmetric operator with dense
domain in a Hilbert space $\mathscr{H}$, and let $x_{\pm}\in\mathscr{D}_{\pm}$
be a pair of vectors in the respective deficiency-spaces. Suppose
\begin{equation}
\left\Vert x_{+}\right\Vert =\left\Vert x_{-}\right\Vert >0.\label{eq:defnorm}
\end{equation}
Then the system
\begin{equation}
\begin{cases}
{\displaystyle y_{2}:=\frac{1}{2i}\left(x_{+}-x_{-}\right)}\\
\\
{\displaystyle y_{3}:=\frac{1}{2}\left(x_{+}+x_{-}\right)}
\end{cases}\label{eq:defsys}
\end{equation}
satisfy $y_{i}\in\mathscr{D}(L^{*})$, $i=2,3,$ and
\begin{equation}
L^{*}y_{2}=y_{3},\mbox{ and }L^{*}y_{3}=-y_{2}.\label{eq:YY23}
\end{equation}
Conversely, if $\left\{ y_{2},y_{3}\right\} $ is a pair of non-zero
vectors in $\mathscr{D}(L^{*})$ such that (\ref{eq:YY23}) holds;
then
\begin{equation}
\begin{cases}
x_{+}=y_{3}+iy_{2}\in\mathscr{D}_{+} & \mbox{and}\\
x_{-}=y_{3}-iy_{2}\in\mathscr{D}_{-},
\end{cases}\label{eq:xpm}
\end{equation}
and (\ref{eq:defnorm}) holds.\end{lem}
\begin{proof}
The result follows from Lemmas \ref{lem:vN def-space} and \ref{lem:vNext}
together with a simple computation.\end{proof}
\begin{cor}
Let $L$ and $\mathscr{D}_{\pm}$ (deficiency-spaces) be as in the
lemma. Let $U:\mathscr{D}_{+}\longrightarrow\mathscr{D}_{-}$ be a
partial isometry, and let $x_{+}$ be in the initial space of $U$.
Then the Hermitian extension $H$ of $L$ given by 
\begin{equation}
H\left(x_{+}+Ux_{+}\right)=i\left(x_{+}-Ux_{+}\right)\label{eq:pi}
\end{equation}
is specified equivalently by the two vectors $y_{2}$ and $y_{3}$
as follows:
\begin{equation}
y_{2}\in\mathscr{D}(H)\;\mbox{and}\ Hy_{2}=y_{3};\;\mbox{or }y_{3}\in\mathscr{D}(H)\:\mbox{and }Hy_{3}=-y_{2}\label{eq:HY23}
\end{equation}
where
\begin{equation}
\begin{cases}
y_{2}={\displaystyle \frac{1}{2i}\left(x_{+}-Ux_{+}\right)}, & \mbox{and}\\
\\
y_{3}={\displaystyle \frac{1}{2}\left(x_{+}+Ux_{+}\right)}.
\end{cases}\label{eq:y232}
\end{equation}
In particular, for the boundary form $\boldsymbol{B}(\cdot,\cdot)$
on $\mathscr{D}(L^{*})$ from (\ref{eq:BoundaryForm}), we have $\boldsymbol{B}(y_{i},y_{i})=0$
for $i=2,3$.\end{cor}
\begin{proof}
This follows from the lemma since a pair of vectors $x_{\pm}\in\mathscr{D}_{\pm}$
is in the graph of a partial isometry $\mathscr{D}_{+}\longrightarrow\mathscr{D}_{-}$
if and only if (\ref{eq:defnorm}) holds.
\end{proof}

\subsection{Prior Literature}

There are related investigations in the literature on spectrum and
deficiency indices. For the case of indices $(1,1)$, see for example
\cite{ST10,Ma11}. For a study of odd-order operators, see \cite{BH08}.
Operators of even order in a single interval are studied in \cite{Oro05}.
The paper \cite{BV05} studies matching interface conditions in connection
with deficiency indices $(m,m)$. Dirac operators are studied in \cite{Sak97}.
For the theory of selfadjoint extensions operators, and their spectra,
see \cite{Smu74,Gil72}, for the theory; and \cite{Naz08,VGT08,Vas07,Sad06,Mik04,Min04}
for recent papers with applications. For applications to other problems
in physics, see e.g., \cite{AH11,PoRa76,Ba49,MK08}. And \cite{Ch11}
on the double-slit experiment. For related problems regarding spectral
resolutions, but for fractal measures, see e.g., \cite{DJ07,DJ09,DJ11}.

The study of deficiency indices $(n,n)$ has a number of additional
ramifications in analysis: Included in this framework is Krein's analysis
of Voltera operators and strings; and the determination of the spectrum
of inhomogenous strings; see e.g., \cite{DS01,KN89,Kr70,Kr55}.

Also included is their use in the study of de Branges spaces, see
e.g., \cite{Ma11}, where it is shown that any regular simple symmetric
operator with deficiency indices $(1,1)$ is unitarily equivalent
to the operator of multiplication in a reproducing kernel Hilbert
space of functions on the real line with a sampling property). Further
applications include signal processing, and de Branges-Rovnyak spaces:
Characteristic functions of Hermitian symmetric operators apply to
the cases unitarily equivalent to multiplication by the independent
variable in a de Branges space of entire functions.

\subsection{Organization of the Paper}

The central themes in our paper are presented, in their most general
form, in sections \ref{sec:Hardy} and \ref{sec:Fried}. However,
in sections \ref{sec:model} through \ref{sec:qua}, we are preparing
the ground leading up to section \ref{sec:Hardy}, beginning with
some lemmas on unbounded operators in section \ref{sec:model}.

Further in sections \ref{sec:sbdd} and \ref{sec:spH} (before introducing
harmonic analysis in the Hardy space $\mathscr{H}_{2}$), we begin
with an analysis of model operators in the Hilbert space $l^{2}(\mathbb{N}_{0})$.
In sections \ref{sec:sbdd} and \ref{sec:spH}, it will be helpful
to restrict our analysis to the case of deficiency indices $(1,1)$.

The reason for beginning with the Hilbert space $l^{2}(\mathbb{N}_{0})$
is that some computations are presented more clearly there. But they
will then be used in section \ref{sec:Hardy} where we introduce operators
in the Hardy space $\mathscr{H}_{2}$ (of analytic functions on the
disk $\mathbb{D}$ with $l^{2}$-coefficients.)

It is only with the use of kernel theory for $\mathscr{H}_{2}$, and
its subspaces, that we are able to make precise our results for the
case of deficiency indices $(m,m)$ where $m$ can be any number in
$\mathbb{N}\cup\{\infty\}$. For the case when our operators have
indices $(m,m)$, $m>1$, the possibilities encompass a rich variety.
Indeed, there is a boundary value problem for every choice of a closed
subset $F\subset\mathbb{T}=\partial\mathbb{D}$ of measure zero. In
fact we prove that even different sets $F$ with the same cardinality
can lead to quite different boundary-value problems, inequivalent
extension operators, and quite different spectral configurations for
the selfadjoint extensions.

\section{\label{sec:model}Restrictions of Selfadjoint Operators}

The general setting here is as sketched above; see especially Lemma
\ref{lem:vN def-space}, a statement of von Neumann\textquoteright{}s
theorem yielding a classification of the selfadjoint extensions of
a fixed Hermitian operator $L$ with dense domain in a given Hilbert
space $\mathscr{H}$, and $L$ having equal deficiency indices. 

Below we will be concerned with the converse question: Given an unbounded
selfadjoint operator in $\mathscr{H}$; what are the parameters for
the variety of \emph{all} closed Hermitian restrictions having dense
domain in $\mathscr{H}$. The answer is given below, where we further
introduce the restriction on the possibilities by the added requirement
of semi-boundedness.

Our main reference regarding unbounded operators in Hilbert space
will be \cite{DS88b}, but we will be relying too on results from
\cite{AG93} on spectral theory in the case of indices $(m,m)$ with
$m$ finite, \cite{Kat95} on closed quadratic forms, and \cite{Gru09}
for distribution theory and semibounded operators.

\subsection{Conventions and Notation.}
\begin{itemize}
\item $\mathscr{H}$ - a complex Hilbert space;
\item $H$ - selfadjoint operator in $\mathscr{H}$;
\item $\mathscr{D}(H)$ - domain of $H$, dense in $\mathscr{H}$; 
\item For $\xi\in\mathbb{C}$, $\Im(\xi)\neq0$, the resolvent operator
\[
R(\xi)=(\xi I-H)^{-1}
\]
is well defined; it is bounded, i.e.,
\begin{equation}
\left\Vert R(\xi)\right\Vert _{\mathscr{H}\rightarrow\mathscr{H}}\leq\left|\Im(\xi)\right|^{-1},\;\mbox{and}\label{eq:res}
\end{equation}
\begin{equation}
R(\xi)^{*}=R(\overline{\xi})\label{eq:res-1}
\end{equation}

\item $\perp$ - orthogonal complement.
\end{itemize}
\textbf{Question:} What are the closed restriction operators $L$
for $H$, such that $\mathscr{D}(L)$ is dense, where $\mathscr{D}(L)$
is the domain of $L$?

The answer is given in the following lemma:
\begin{lem}
\label{lem:subsp}Let $\mathscr{H}$, $H$, and $\xi$, be as above,
in particular, $\Im(\xi)\neq0$ is fixed. Then there is a bijective
correspondence between \emph{(\ref{enu:a})} and \emph{(\ref{enu:b})}
as follows:\end{lem}
\begin{enumerate}
\item \label{enu:a}$L$ is a closed restriction of $H$ with dense domain
$\mathscr{D}(L)$ in $\mathscr{H}$; and 
\item \label{enu:b}$\mathfrak{M}$ is a closed subspace in $\mathscr{H}$
such that 
\begin{equation}
\mathfrak{M}\cap\mathscr{D}(H)=0;\label{eq:csub}
\end{equation}
where the correspondence $L$ to $\mathfrak{M}$ is
\begin{equation}
\mathfrak{M}=\{\psi\in\mathscr{H}\:\big|\:\psi\in\mathscr{D}(L^{*}),L^{*}\psi=\xi\psi\};\label{eq:csub-1}
\end{equation}
while, from $\mathfrak{M}$ to $L$, it is:
\begin{equation}
\mathscr{D}(L)=\left(R(\overline{\xi})\:\mathfrak{M}\right)^{\perp}.\label{eq:csub-2}
\end{equation}
\end{enumerate}
\begin{proof}
\textbf{From}\textbf{\emph{ }}\textbf{(\ref{enu:a})}\textbf{\emph{
}}\textbf{to (\ref{enu:b}).} Let $L$ be a restriction operator as
in (\ref{enu:a}), i.e., with $\mathscr{D}(L)$ dense in $\mathscr{H}$.
Then $L$ is Hermitian, and therefore, 
\begin{equation}
L\subset L^{*}.\label{eq:resL}
\end{equation}
It follows in particular that $\mathscr{D}(L^{*})$ is also dense. 

By von Neumann's theory, \cite{DS88b}, we conclude that, for every
$\xi\in\mathbb{C}$ s.t. $\Im(\xi)\neq0$, the subspaces $\mathfrak{M}_{\xi}$
in (\ref{eq:csub-1}) have the same dimension. In particular, $L$
has deficiency indices $(n,n)$ where $n=\mbox{dim }\mathfrak{M}$.
We pick a fixed $\xi\in\mathbb{C}$, $\Im(\xi)\neq0$. It further
follows from \cite{DS88a} that $\mathfrak{M}$ is closed. 

We now prove that $\mathfrak{M}$ satisfies (\ref{eq:csub}). If $\mathfrak{M}=0$,
there is nothing to prove. Now suppose $\psi\in\mathfrak{M}$ and
$\psi\neq0$. Since $L\subseteq H$, from (\ref{eq:resL}) we get
\begin{equation}
L\subset H\subset L^{*},\label{eq:LH}
\end{equation}
and therefore if $\psi\in\mathscr{D}(H)$, it follows that
\[
\left\langle \psi,L^{*}\psi\right\rangle =\left\langle \psi,H\psi\right\rangle \in\mathbb{R}.
\]
But (\ref{eq:csub-1}) implies:

\begin{equation}
\left\langle \psi,L^{*}\psi\right\rangle =\xi\left\langle \psi,\psi\right\rangle =\xi\left\Vert \psi\right\Vert ^{2}.\label{eq:psi}
\end{equation}
Since $\Im(\xi)\neq0$, we have a contradiction. Hence (\ref{eq:csub})
must hold.

\textbf{From (\ref{enu:b}) to (\ref{enu:a}).} Let $\mathfrak{M}$
be a closed subspace satisfying (\ref{eq:csub}). Then define the
subspace $\mathscr{D}(L)$ as in (\ref{eq:csub-2}). Let $L:=H\big|_{\mathscr{D}(L)}$,
i.e., defined to be the restriction of $H$ to this subspace. 

The key step in the argument is the assertion that $\mathscr{D}(L)$,
so defined, is \emph{dense }in $\mathscr{H}$. We will prove the implication:
\begin{equation}
\psi\in\mathscr{D}(L)^{\perp}\Longrightarrow\psi=0\label{eq:perp}
\end{equation}
By (\ref{eq:csub-2}), we have
\[
\mathscr{D}(L)=\left(R(\overline{\xi})\:\mathfrak{M}\right)^{\perp}=R(\xi)\left(\mathfrak{M}^{\perp}\right).
\]
Hence, $\psi\in\mathscr{D}(L)^{\perp}$ $\Longleftrightarrow$
\[
\left\langle R(\overline{\xi})\psi,m^{\perp}\right\rangle =\left\langle \psi,\underset{\in\mathscr{D}(L)}{\underbrace{R(\xi)m^{\perp}}}\right\rangle =0,\:\forall m^{\perp}\in\mathfrak{M}^{\perp}.
\]
But $\mathfrak{M}=\mathfrak{M}^{\perp\perp}$, since $\mathfrak{M}$
is closed. Hence $R(\overline{\xi})\psi\in\mathfrak{M}\cap\mathscr{D}(H)=0$;
and therefore $\psi=0$; which proves (\ref{eq:perp}).
\end{proof}

\subsection{The Case When $H$ is Semibounded}

There is an extensive general theory of semibounded Hermitian operators
with dense domain in Hilbert space, \cite{AG93,DS88b,Kat95}. One
starts with a fixed semibounded Hermitian operator $L$, and then
passes to a corresponding quadratic form $q_{L}$ \cite{Kat95}. The
lower bound for $L$ is defined from $q_{L}$. 

Now, the initial operator $L$ will automatically have equal deficiency
indices, and, in the general case (Lemma \ref{lem:subsp}), there
is therefore a rich variety of possibilities for the selfadjoint extensions
of $L$ . In this paper, we will be concerned with particular model
examples of semibounded operators, typically have much more restricted
parameters for their selfadjoint extensions than what is possible
for more general semibounded operators. Nonetheless, there are many
instances of operators arising in applications which are unitarily
equivalent to the \textquotedblleft{}simple\textquotedblright{} model.
Some will be discussed in detail in sections \ref{sec:qua} and \ref{sec:Hardy}
below.
\begin{defn}
We say that a Hermitian operator $L$ with dense domain $\mathscr{D}(L)$
in a fixed Hilbert space $\mathscr{H}$ is semibounded if there is
a number $b>-\infty$ such that
\begin{equation}
\left\langle x,Lx\right\rangle \geq b\left\Vert x\right\Vert ^{2},\:\mbox{for all }x\in\mathscr{D}(L).\label{eq:semibdd}
\end{equation}
Then the best constant $b$, valid for all $x$ in (\ref{eq:semibdd}),
will be called the greatest lower bound (GLB.)\end{defn}
\begin{lem}
\label{lem:semibdd}Suppose a selfadjoint operator $H$ has a lower
bound $b$. If $c\in\mathbb{R}$ satisfies $-\infty<c<b$, then, in
the parameterization from Lemma \ref{lem:subsp}, we may take
\begin{equation}
\mathfrak{M}=\{\psi\in\mathscr{H}\:\big|\:\psi\in\mathscr{D}(L^{*}),\: L^{*}\psi=c\,\psi\},\label{eq:M}
\end{equation}
and, in the reverse direction, 
\begin{equation}
\mathscr{D}(L)=\left((H-cI)^{-1}\mathfrak{M}\right)^{\perp}.\label{eq:L-semi}
\end{equation}
This will again be a bijective correspondence between: 
\begin{enumerate}
\item all the closed and densely defined restrictions $L$ of $H$, and 
\item all the closed subspaces $\mathfrak{M}$ in $\mathscr{H}$ satisfying
\begin{equation}
\mathfrak{M}\cap\mathscr{D}(H)=\{0\}.\label{eq:int}
\end{equation}

\end{enumerate}
\end{lem}
\begin{proof}
The argument is the same as that used in the proof of Lemma \ref{lem:subsp},
\emph{mutatis mutandis}. 
\end{proof}

\section{\label{sec:sbdd}Semibounded Operators}

Below we consider the particular semibounded Hermitian operator $L$
with its dense domain in the Hilbert space $l^{2}(\mathbb{N}_{0})$
of square-summable one-sided sequences. (Our justification for beginning
with $l^{2}$ is the natural and known isometric isomorphism $l^{2}\simeq\mathscr{H}_{2}$
(with the Hardy space); see \cite{Rud87} and section \ref{sec:Hardy}
below for details.) It is specified by a single linear condition,
see (\ref{eq:domL}) below, and is obtained as a restriction of a
selfadjoint operator $H$ in $l^{2}(\mathbb{N}_{0})$ having spectrum
$\mathbb{N}_{0}$. While $0$ is in the bottom of the spectrum of
$H$, it is not \emph{a priori} clear that the greatest lower bound
for its restriction $L$ will also be $0$, (see sect. \ref{sec:qua}
for details.) After all, finding the lower bound for $L$ is a quadratic
optimization problem with constraints. Nonetheless we prove (Lemma
\ref{lem:lbdd}) that $L$ also has $0$ as its lower bound.

For $p>0$, let $\ell^{p}$ be the set of complex sequence ${\displaystyle x=(x_{k})_{k=0}^{\infty}}$
such that ${\displaystyle \sum\left|x_{k}\right|^{p}}<\infty$. The
operator ${\displaystyle H\left(x_{k}\right)=\left(kx_{k}\right)}$
with domain 
\begin{equation}
{\displaystyle \mathscr{D}(H)=\{x\in\ell^{2}:Hx\in\ell^{2}\}}\label{eq:Hsa}
\end{equation}
is selfadjoint in the Hilbert space $l^{2}$. 
\begin{lem}
\label{lem:Hdom}$\mathscr{D}\left(H\right)$ is a subspace of $\ell^{1}.$
In particular, $\sum x_{k}$ is absolutely convergent for all $x$
in $\mathscr{D}\left(H\right)$. More generally, for $m\in\mathbb{N}_{0}$,
we have:
\begin{equation}
\mathscr{D}(H^{m})\subset l^{p}\quad\mbox{if}\quad p>\frac{2}{2m+1}.\label{eq:lp}
\end{equation}
\end{lem}
\begin{proof}
By Cauchy-Schwarz, we have 
\[
\sum\left|x_{k}\right|=\sum\frac{1}{k}\left|kx_{k}\right|\leq\left(\sum\frac{1}{k^{2}}\right)^{1/2}\left(\sum\left|kx_{k}\right|^{2}\right)^{1/2}.
\]
The second part (\ref{eq:lp}) follows from an application of Hölder's
inequality.\end{proof}
\begin{rem}
The proof shows that $\phi:\left(x_{k}\right)\to\sum x_{k}$ is a
continuous functional $\mathscr{D}(H)\to\mathbb{C},$ when $\mathscr{D}(H)$
is equipped with the graph norm. 
\end{rem}
Consider the operator $L$ on $\ell^{2}$ determined by $\left(Lx\right)_{k}=k\, x_{k}$
with domain 
\begin{equation}
\mathscr{D}(L)=\mathbb{D}_{0}=\left\{ x\in\mathscr{D}\left(H\right):\sum x_{k}=0\right\} .\label{eq:domL}
\end{equation}
Note $\mathscr{D}(L)$ has co-dimension one as a subspace of $\mathscr{D}(H).$

\begin{lem}
\label{lem:densedom}$\mathbb{D}_{0}$ is dense in $\ell^{2}.$ \end{lem}
\begin{proof}
The assertion follows from Lemma \ref{lem:semibdd} above, but we
include a direct proof as well, as this argument will be used later.

The set of sequences $\ell_{\mathrm{fin}}$ with only a finite number
of non-zero terms is dense in $\ell^{2}.$ Suppose $x=\left(x_{k}\right),$
$x_{k}=0$ for $k>n,$ and $A=\sum x_{k}.$ Let $y_{m}=\left(y_{m,k}\right)_{k=0}^{\infty}$
in $\ell_{\mathrm{fin}}$ be determined by 
\[
y_{m,k}=\begin{cases}
x_{k} & \text{if }k<n\\
-\frac{A}{m} & \text{if }m\, n\leq k<m(n+1)
\end{cases}.
\]
Then 
\[
\left\Vert x-y_{m}\right\Vert ^{2}=\sum_{k=0}^{\infty}\left|x_{k}-y_{m,k}\right|^{2}=\sum_{k=m\, n}^{m(n+1)-1}\frac{\left|A\right|^{2}}{m^{2}}=\frac{\left|A\right|^{2}}{m}\to0
\]
as $m\to\infty.$
\end{proof}
Since $\phi$ is continuous $L$ is a closed operator. Furthermore,
\begin{equation}
L\subset H\subset L^{*}\label{eq:LHL}
\end{equation}
since $L\subset H$ and $H$ is selfadjoint. 

Here, containment in (\ref{eq:LHL}) for pairs of operators means
containment of the respective graphs.
\begin{lem}
\label{lem:L*-eigenvalues}Every complex number is an eigenvalue of
$L^{*}$ of multiplicity one. \end{lem}
\begin{proof}
The case when $\xi$ has non-zero imaginary part is covered by Lemma
\ref{lem:subsp}. Since $\mathscr{D}(L)$ is dense in $\ell^{2}$
we have 
\begin{align*}
L^{*}y=\xi y & \iff\left\langle \left(L^{*}-\xi\right)y,x\right\rangle =0,\forall x\in\mathscr{D}(L)\\
 & \iff\left\langle y,\left(L-\overline{\xi}\right)x\right\rangle =0,\forall x\in\mathscr{D}(L)\\
 & \iff\sum_{0}^{\infty}\overline{y_{k}}(k-\overline{\xi})x_{k}=0,\forall x\in\mathscr{D}(L).
\end{align*}
Considering $x_{k}=-x_{k+1}=1$ and $x_{j}=0$ for all $j\neq k,k+1$
we conclude 
\[
\overline{y_{0}}\overline{\xi}=\overline{y_{1}}\left(1-\overline{\xi}\right)=y_{2}\left(2-\overline{\xi}\right)=\cdots=\overline{y_{k}}\left(k-\overline{\xi}\right)=\cdots
\]
Hence, if $k-\overline{\xi}\neq0$ for all $k,$ then 
\[
y_{k}=\frac{y_{0}\,\overline{\xi}}{k-\overline{\xi}},\forall k\geq1.
\]
And, if $k_{0}-\overline{\xi}=0,$ then $y_{k}=0$ for all $k\neq k_{0}.$ 
\end{proof}

\subsection{Selfadjoint Extensions}

As a consequence of the lemma, $L$ has deficiency indices $(1,1)$
and the corresponding defect spaces are 
\[
\mathscr{D}_{\pm}=\mathbb{C}x_{\pm}
\]
where 
\begin{equation}
\left(x_{\pm}\right)_{k}=\frac{1}{k\mp i},\quad k\geq0.\label{eq:defv}
\end{equation}
In particular, we have the von Neumann formula (eq. (\ref{eq:DLs})
in Lemma \ref{lem:vNext}; and also see \cite[pg 1227, Lemma 10]{DS88b})
\begin{equation}
\mathscr{D}(L^{*})=\mathscr{D}(L)\oplus\mathbb{C}x_{+}\oplus\mathbb{C}x_{-}.\label{eq:vNext}
\end{equation}
By von Neumann (Lemma \ref{lem:vN def-space}), any selfadjoint extension
of $L$ is of the form 
\begin{align*}
\mathscr{D}(L_{\theta}) & =\mathscr{D}(L)+\mathbb{C}\left(x_{+}+e(\theta)x_{-}\right)\\
L_{\theta}\left(x+ax_{+}+a\, e(\theta)x_{-}\right) & =Lx+aix_{+}-a\, e(\theta)ix_{-},\quad x\in\mathscr{D}(L),a\in\mathbb{C}.
\end{align*}

Alternatively, if $\zeta=e(\theta)=e^{i2\pi\theta}$, let 
\begin{equation}
\left(y_{2}\right)_{k}=\frac{1}{1+k^{2}}\label{eq:Y2}
\end{equation}
and 
\begin{equation}
\left(y_{3}\right)_{k}=\frac{k}{1+k^{2}}\label{eq:Y3}
\end{equation}
where $L^{*}y_{2}=y_{3}$, and $L^{*}y_{3}=-y_{2}$. This is an application
of Lemma \ref{lem:Y23}. 

Then 
\[
\frac{1}{k-i}+\frac{\zeta}{k+i}=(1+\zeta)\frac{k}{1+k^{2}}+i(1-\zeta)\frac{1}{1+k^{2}}.
\]
Hence 
\[
x_{+}+\zeta x_{-}=(1+\zeta)y_{3}+i(1-\zeta)y_{2}.
\]
Similarly, the selfadjoint extension operators $L_{\theta}=H_{e(\theta)}$
satisfies
\[
L_{\theta}(x_{+}+\zeta x_{-})=-(1+\zeta)y_{2}+i(1-\zeta)y_{3}.
\]
Consequently, if $H_{\zeta}=L_{\theta},$ then
\begin{equation}
\mathscr{D}(H_{\zeta})=\mathscr{D}(L)+\mathbb{C}(1+\zeta)y_{3}+i(1-\zeta)y_{2}\label{eq:ext}
\end{equation}
and 
\begin{equation}
H_{\zeta}(x+a((1+\zeta)y_{3}+i(1-\zeta)y_{2}))=Lx+a\left(-(1+\zeta)y_{2}+i(1-\zeta)y_{3}\right).\label{eq:ext-1}
\end{equation}

\begin{rem}
\label{rem:H}$y_{2}$ is in $\mathscr{D}(H)$ and not in $\mathscr{D}(L)$,
hence $\mathscr{D}(H)=\mathscr{D}\left(H_{-1}\right).$ Since both
$H$ and $H_{-1}$ are restrictions of $L^{*}$ we conclude that $H=H_{-1}=L_{1/2}.$ 
\end{rem}
We say an operator has \emph{discrete spectrum} if it has empty essential
spectrum. 
\begin{thm}
\label{thm:Spectrum-L_t-discreet}Every selfadjoint extention of $L$
has discrete spectrum of uniform multiplicity one. \end{thm}
\begin{proof}
Since $L$ has finite deficiency indices and one of its selfadjoint
extentions has discrete spectrum, so does every selfadjoint extension
of $L,$ see e.g. \cite[Section 11.6]{dO09}. 

Consider some selfadjoint extension $H_{\zeta}$ of $L.$ If $\lambda$
is an eigenvalue for $H_{\zeta}$ and $x$ is a corresponding eigenvector,
then 
\[
L^{*}x=H_{\zeta}x=\lambda x
\]
since $H_{\zeta}$ is a restriction of $L^{*}.$ Hence the multiplicity
claim follows from Lemma \ref{lem:L*-eigenvalues}. 
\end{proof}
In conclusion, we add that an application of Lemma \ref{lem:subsp}
to $\mathscr{H}=l^{2}(\mathbb{N}_{0})$ and the above results (sect
\ref{sec:sbdd}) yield the following: 
\begin{cor}
\label{cor:bddsq}Any bounded sequence $\alpha=(a_{n})\notin l^{2}$
induces a densely defined restriction operator $L_{\alpha}$ with
domain 
\begin{equation}
\left\{ x=(x_{n})\in\mathscr{D}(H)\; s.t.\;\mbox{the boundary condition }\sum_{n\in\mathbb{N}_{0}}a_{n}x_{n}=0\;\mbox{holds}\right\} .\label{eq:bddsq}
\end{equation}
 \end{cor}
\begin{proof}
One uses the arguments from Lemmas \ref{lem:L*-eigenvalues} and \ref{lem:subsp},
\emph{mutatis mutandis}.\end{proof}
\begin{rem}
There are densely defined $(1,1)$ restrictions for $H$ not accounted
for in Corollary \ref{cor:bddsq}. 

To see this, recall (Lemma \ref{lem:subsp}) that all the $(1,1)$
restrictions $L_{\alpha}$ of $H$ are defined from a boundary condition
\begin{equation}
\sum_{k\in\mathbb{N}_{0}}(1+k)y_{k}x_{k}=0,\label{eq:bdres}
\end{equation}
where $y\in l^{2}\backslash\mathscr{D}(H)$. So that $\alpha=\left(\left(1+k\right)y_{k}\right)_{k}$
as in (\ref{eq:bddsq}). If every one of these $\left(\left(1+k\right)y_{k}\right)_{k}$
were in $l^{\infty}$, (by the uniform boundedness principle) we would
get $\left\{ y\in l^{2}\:\big|\:\left\Vert y\right\Vert _{2}\leq1\right\} $
contained in the Hilbert cube, which is a contradiction. (Life outside
the Hilbert cube.)
\end{rem}

\section{\label{sec:spH}The Spectrum of $H_{\zeta}.$}

In this section we analyze the spectrum of each of the selfadjoint
extensions of the basic Hermitian operator $L$ from section \ref{sec:sbdd}.
Since $L$ has deficiency indices $(1,1)$, it follows from Lemma
2.1 that the selfadjoint extensions are parameterized bijectively
by the circle group $\mathbb{T}=\{z\in\mathbb{C}\:|\:\left|z\right|=1\}$.

While the case of indices $(1,1)$ may seem overly special, we show
in section \ref{sec:Hardy} below that our detailed analysis of the
$(1,1)$ case has direct implication for the general configuration
of deficiency indices $(m,m)$, even including $m=\infty$. 

For the $(1,1)$ case, we have a one-parameter family of selfadjoint
extensions of the initial Hermitian operator $L$. These are indexed
by $\zeta\in\mathbb{T}$, or equivalently by $\mathbb{R}/\mathbb{Z}$
via the rule $\zeta=e(t)=e^{i2\pi t}$, $t\in\mathbb{R}/\mathbb{Z}$. 

Now fix $t\in\mathbb{R}/\mathbb{Z}$, say $-\frac{1}{2}<t\leq\frac{1}{2}$;
let $K:=(1+\pi\coth(\pi))/2$; let $\gamma$ be the Euler's constant;
set 
\begin{equation}
\psi(z):=-\gamma+(z-1)\sum_{k=0}^{\infty}\frac{1}{\left(k+1\right)\left(k+z\right)};\label{eq:psipsi}
\end{equation}
and set
\begin{equation}
G(z):=\Re\{\psi(i)\}-\psi(-z),\quad z\in\mathbb{C}.\label{eq:Gz}
\end{equation}
We will use $\Re$ and $\Im$ to denote real part, and imaginary part,
respectively. The function $\psi$ in (\ref{eq:psipsi}) is the digamma
function; see e.g., \cite{AS92,CSL11}.

For fixed $t$, we then show that the spectrum of the selfadjoint
extension $L_{t}:=H_{\zeta}$, $\zeta=e(t)$, is the set of solutions
$\lambda=\lambda_{n}(t)\in\mathbb{R}$ to the equation 
\begin{equation}
G(\lambda)=K\tan(\pi t).\label{eq:eeqn}
\end{equation}
See Figure \ref{fig:eigen}. We prove in Lemma \ref{lem:psireim}
that $K=\Im(\psi(i))$.
\begin{thm}
\label{thm:eigEqn}Let $H_{\zeta}$ be the selfadjoint extension in
(\ref{eq:ext-1}) with domain $\mathscr{D}(H_{\zeta})$ in (\ref{eq:ext}).
Let $\lambda\in\mathbb{R}$ be an eigenvalue of $H_{\zeta}$ with
the corresponding eigenfunction $x\in\mathscr{D}(H_{\zeta})$, i.e.,
\[
H_{\zeta}x=\lambda x,\: x\in\mathscr{D}(H_{\zeta}).
\]
Then $\lambda$ is a zero of the function 
\begin{equation}
F(\lambda):=\sum_{k=0}^{\infty}\frac{\lambda(1+\zeta)k+i\lambda(1-\zeta)+(1+\zeta)-i(1-\zeta)k}{\left(k-\lambda\right)\left(1+k^{2}\right)}.\label{eq:F-1}
\end{equation}
Conversely, every eigenvalue of $H_{\zeta}$ arises this way. \end{thm}
\begin{proof}
Suppose $H_{\zeta}x=\lambda x$, i.e., 
\[
Lx+a\left(-(1+\zeta)y_{2}+i(1-\zeta)y_{3}\right)=\lambda x+\lambda a((1+\zeta)y_{3}+i(1-\zeta)y_{2})
\]
in term of coordinates 
\[
kx_{k}+a\frac{-(1+\zeta)+i(1-\zeta)k}{1+k^{2}}=\lambda x_{k}+\lambda a\frac{(1+\zeta)k+i(1-\zeta)}{1+k^{2}}.
\]
Solving for $x_{k}$ we get 
\[
(k-\lambda)x_{k}=a\frac{\lambda\left((1+\zeta)k+i(1-\zeta)\right)+(1+\zeta)-i(1-\zeta)k}{1+k^{2}}
\]
hence 
\[
x_{k}=a\frac{\lambda(1+\zeta)k+i\lambda(1-\zeta)+(1+\zeta)-i(1-\zeta)k}{\left(k-\lambda\right)\left(1+k^{2}\right)}.
\]
Hence $\lambda$ is an eigenvalue for $H_{\zeta}$ iff $F(\lambda)=0,$
where $F(\lambda)$ is given in (\ref{eq:F-1}). \end{proof}
\begin{rem}
Considering now a $2\pi$-periodic interval, and setting $\zeta=\cos(t)+i\sin(t)$
we see that 
\begin{align}
F(\lambda) & =(1+\cos(t))\sum_{k=0}^{\infty}\frac{1+\lambda k}{\left(k-\lambda\right)\left(1+k^{2}\right)}-\sin(t)\sum_{k=0}^{\infty}\frac{1}{1+k^{2}}\nonumber \\
 & +i\left(\left(\cos(t)-1\right)\sum_{k=0}^{\infty}\frac{1}{\left(1+k^{2}\right)}+\sin(t)\sum_{k=0}^{\infty}\frac{1+\lambda k}{\left(k-\lambda\right)\left(1+k^{2}\right)}\right)\label{eq:F}
\end{align}
Note if $\cos(t)=-1$, then $L_{t}=H$ is selfadjoint with $spec(H)=\mathbb{N}_{0}$;
see eq. (\ref{eq:Hsa}).

If $\cos(t)\neq-1$, then $\lambda$ is an eigenvalue iff
\begin{equation}
\sum_{k=0}^{\infty}\frac{1+\lambda k}{\left(k-\lambda\right)\left(1+k^{2}\right)}=\frac{\sin(t)}{1+\cos(t)}\sum_{k=0}^{\infty}\frac{1}{1+k^{2}}\label{eq:eigen}
\end{equation}

\end{rem}

\begin{rem}
\label{rem:K}In fact, 
\[
\sum_{k=1}^{\infty}\frac{\cos(k\, x)}{k^{2}+a^{2}}=\frac{\pi}{2a}\frac{\cosh(a(\pi-x))}{\sinh(a\pi)}-\frac{1}{2a^{2}},\qquad0\leq x\leq2\pi.
\]
Hence 
\begin{equation}
K:=\sum_{k=0}^{\infty}\frac{1}{1+k^{2}}=\frac{\pi}{2}\frac{\cosh(\pi)}{\sinh(\pi)}+\frac{1}{2}=\frac{\pi}{2}\coth(\pi)+\frac{1}{2}\simeq18.708598.\label{eq:conK}
\end{equation}
\end{rem}
\begin{thm}
\label{thm:Spectrum-L_t} If $\cos(t)\neq-1$, then the selfadjoint
operator $L_{t}=H_{e(t)}$ corresponding to $\zeta=e(t)$ via von
Neumann's formula (\ref{eq:vNext}) has pure point spectrum of the
following form 
\begin{equation}
\mbox{spectrum}(L_{t})=\{\lambda_{n}(t)\}_{n\in\mathbb{N}_{0}}\quad\mbox{where}\label{eq:spHt}
\end{equation}
\begin{equation}
\lambda_{0}(t)<0,\;\mbox{and }n-1<\lambda_{n}(t)<n,\:\mbox{ for }n\in\mathbb{N};\label{eq:spHt-1}
\end{equation}
see Figure \ref{fig:eigen}.\end{thm}
\begin{proof}
To begin with, we take a closer look at formula (\ref{eq:eigen}).
Set 
\[
G(\lambda)=\sum_{k=0}^{\infty}\frac{1+\lambda\, k}{\left(k-\lambda\right)\left(1+k^{2}\right)}
\]
and note that ${\displaystyle \frac{\sin\left(t\right)}{1+\cos\left(t\right)}=\tan\left(t/2\right)}$.
We see that (\ref{eq:eigen}) is equivalent to the following equation
\begin{equation}
G(\lambda)=\tan\left(\frac{t}{2}\right)K\label{eq:G}
\end{equation}
(using now a $2\pi$-periodic-interval), where ${\displaystyle K=\sum_{k=0}^{\infty}\frac{1}{1+k^{2}}}$,
see Remark \ref{rem:K} and (\ref{eq:conK}).

To solve (\ref{eq:G}), note that $\frac{d}{d\lambda}G(\lambda)$
may be computed via a differentiation under the $\sum_{k\in\mathbb{N}_{0}}$-
summation. We then get
\begin{equation}
G'(\lambda)=\sum_{k=0}^{\infty}\frac{1}{(k-\lambda)^{2}},\quad\text{ and }G''(\lambda)=\sum_{k=0}^{\infty}\frac{2}{\left(k-\lambda\right)^{3}}.\label{eq:dG}
\end{equation}
In particular, $G'(\lambda)>0$ when $\lambda\notin\mathbb{N}_{0}$
and $G''(\lambda)>0$ when $\lambda<0.$ 

Let $k_{0}\geq0$ be an integer. Write 
\[
G(\lambda)=\frac{1+\lambda\, k_{0}}{\left(k_{0}-\lambda\right)\left(1+k_{0}^{2}\right)}+\sum_{\begin{array}{c}
k=0\\
k\neq k_{0}
\end{array}}^{\infty}\frac{1+\lambda\, k}{\left(k-\lambda\right)\left(1+k^{2}\right)}.
\]
 By the Weierstrass \emph{M}-test, the sum
\[
h(\lambda)=\sum_{\begin{array}{c}
k=0\\
k\neq k_{0}
\end{array}}^{\infty}\frac{1+\lambda\, k}{\left(k-\lambda\right)\left(1+k^{2}\right)}
\]
is absolutely convergent on any compact set not containing $k_{0}.$
In particular, $h(\lambda)$ is bounded on any compact set not containing
$k_{0}.$ On the other hand 
\[
\lim_{\lambda\searrow k_{0}}\frac{1+\lambda\, k_{0}}{\left(k_{0}-\lambda\right)\left(1+k_{0}^{2}\right)}=-\infty\text{ and }\lim_{\lambda\nearrow k_{0}}\frac{1+\lambda\, k_{0}}{\left(k_{0}-\lambda\right)\left(1+k_{0}^{2}\right)}=\infty.
\]

We have verified that the graph of $g$ roughly looks like Figure
\ref{fig:G}.

\begin{figure}
\includegraphics{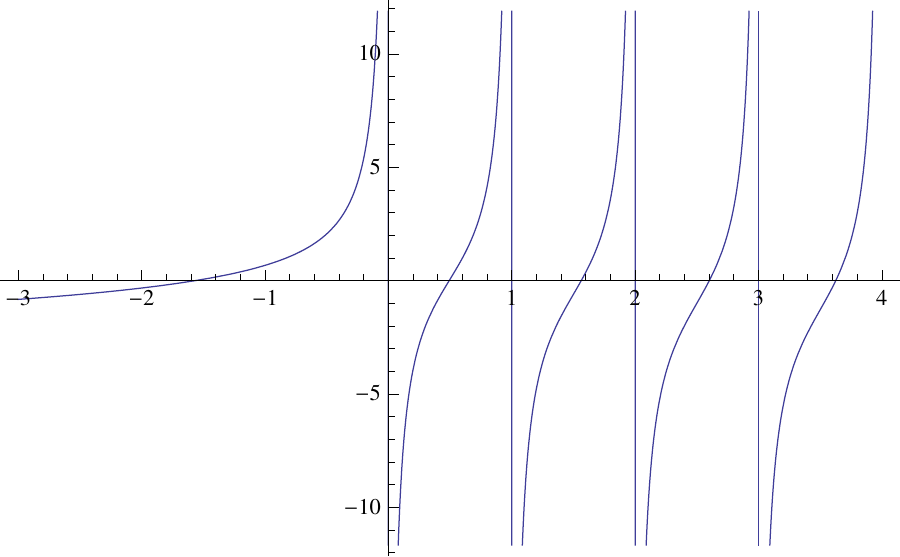}

\caption{\label{fig:G}Graph of $G(\lambda)$, and a root between $-2$ and
$-1$.}
\end{figure}

To finish verifying rigorously that this figure illustrates the behavior
of $G$, it remains to show that there is a negative root; and to
show $G(\lambda)\to-\infty$ as $\lambda\to-\infty.$

To see that there is root between $-1$ and $-2$ we calculate as
follows: 
\begin{align*}
G(-1) & =\sum_{k=0}^{\infty}\frac{1-k}{\left(1+k\right)\left(1+k^{2}\right)}=1-\sum_{k=2}^{\infty}\frac{k-1}{\left(1+k\right)\left(1+k^{2}\right)}\\
G(-2) & =\sum_{k=0}^{\infty}\frac{1-2k}{\left(2+k\right)\left(1+k^{2}\right)}=\frac{1}{2}-\sum_{k=1}^{\infty}\frac{2k-1}{\left(2+k\right)\left(1+k^{2}\right)}
\end{align*}
hence 
\begin{align*}
G(-1) & >1-\sum_{k=2}^{\infty}\frac{k-1}{\left(1+k\right)\left(1+k^{2}\right)}=2-\frac{\pi\coth(\pi)}{2}\approx0.423326\\
G(-2) & <\frac{1}{2}-\left(\frac{1}{3\cdot2}+\frac{3}{4\cdot5}+\frac{5}{5\cdot10}+\frac{7}{6\cdot17}+\frac{9}{7\cdot26}\right)=-\frac{215}{6188}.
\end{align*}
Hence there is a root between $-2$ and $-1$, by the intermediate
value theorem.

Finally, it remains to show that $G(\lambda)\to-\infty$ as $\lambda\to-\infty.$ 

Write 
\begin{align*}
G(-\lambda) & =\sum_{k=0}^{\infty}\frac{1-\lambda\, k}{\left(k+\lambda\right)\left(1+k^{2}\right)}
\end{align*}
 for $\lambda>0.$ Since 
\[
0\leq\sum_{k=0}^{\infty}\frac{1}{\left(k+\lambda\right)\left(1+k^{2}\right)}\leq\sum_{k=0}^{\infty}\frac{1}{k\left(1+k^{2}\right)}<\infty,
\]
 we need to show 
\[
\sum_{k=0}^{\infty}\frac{\lambda\, k}{\left(k+\lambda\right)\left(1+k^{2}\right)}\to\infty
\]
as $\lambda\to\infty.$ Observe that 
\[
\frac{d}{d\lambda}\frac{\lambda\, k}{(k+\lambda)\left(1+k^{2}\right)}=\frac{k^{2}}{(k+\lambda)^{2}\left(1+k^{2}\right)}>0.
\]
Hence by Monotone Convergence Theorem 
\begin{align*}
\lim_{\lambda\to\infty}\sum_{k=0}^{\infty}\frac{\lambda\, k}{\left(k+\lambda\right)\left(1+k^{2}\right)} & =\sum_{k=0}^{\infty}\lim_{\lambda\to\infty}\frac{\lambda\, k}{\left(k+\lambda\right)\left(1+k^{2}\right)}\\
 & =\sum_{k=0}^{\infty}\frac{k}{1+k^{2}}=\infty.
\end{align*}
The convergence is logarithmic as illustrated by Figure \ref{fig:Glog}.

Now fix $t$ such that $\cos(t)\neq-1.$ The intersection points in
(\ref{eq:G}) can be constructed as follows:

The first coordinates of the intersection points of the horizontal
line $y=\tan\left(t/2\right)$ and the curve $y=G(\lambda)$ are points
\begin{equation}
S(t):=\{\lambda_{n}(t)\}_{n\in\mathbb{N}_{0}}.\label{eq:St}
\end{equation}
Each of the intervals $(-\infty,0)$, and $(n,n+1)$, $n=0,1,2,\ldots$
contains precisely one point from $S(t)$. Hence, the numbers in $S(t)$
from (\ref{eq:St}) can be assigned an indexing such that the monotone
order relations in (\ref{eq:spHt-1}) are satisfied.
\end{proof}
\begin{figure}
\includegraphics{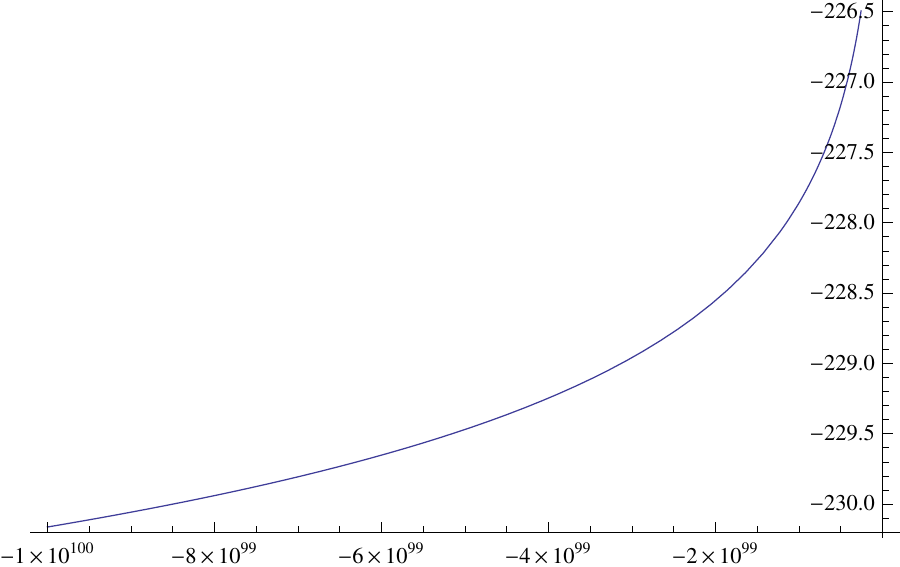}

\caption{\label{fig:Glog}${\displaystyle \lim_{\lambda\rightarrow-\infty}G(\lambda)=-\infty}$}
\end{figure}

\begin{figure}
\includegraphics[scale=0.8]{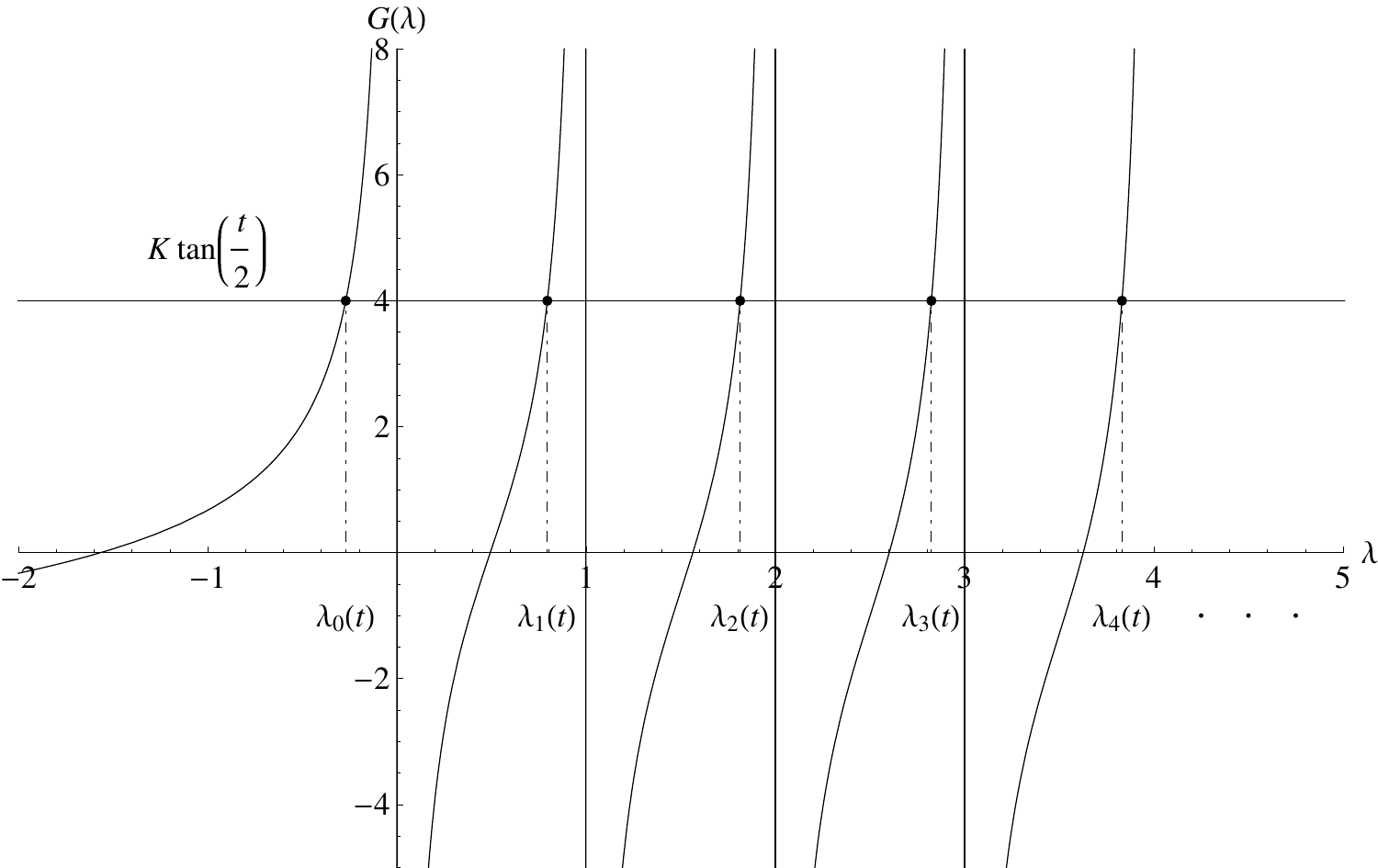}

\caption{\label{fig:eigen}The set $S(t)$ of intersection points.}
\end{figure}

\begin{cor}
\label{cor:bottom}For every $b<0$, there is a unique selfadjoint
extension $H_{b}$ of $L$ such that $b$ is the bottom of the spectrum
of $H_{b}$.
\end{cor}

\begin{cor}
If $\zeta=e(t)\neq-1,$ $\lambda_{n}(t),$ $n=0,1,2,\ldots$ are the
eigenvalues of $L_{t},$ and 
\begin{equation}
y_{n,k}(t)=\frac{\lambda_{n}(t)}{k-\lambda_{n}(t)},\forall k\geq0,\label{eq:ev}
\end{equation}
then $y_{n}(t)=\left(y_{n,k}(t)\right)_{k=0}^{\infty},$ $n=0,1,2,\ldots$
is an orthogonal basis for $\ell^{2}$; in particular,
\[
\sum_{k\in\mathbb{N}_{0}}\frac{1}{\left(k-\lambda_{n}(t)\right)\left(k-\lambda_{m}(t)\right)}=0,\quad n\neq m.
\]
\end{cor}
\begin{proof}
By Theorem \ref{thm:Spectrum-L_t-discreet} the spectrum of $L_{t}$
is discrete and has multiplicity one. By Theorem \ref{thm:Spectrum-L_t}
no eigenvalue $\lambda_{n}(t)$ is an integer. By, the proof of, Lemma
\ref{lem:L*-eigenvalues} the $y_{n}(t)$ are the eigenvectors for
$L_{t}.$ \end{proof}
\begin{rem}
The norm of $y_{n}(t)$ can be evaluated in terms of gamma functions
\[
\sum_{k=0}^{\infty}\left(\frac{\lambda_{n}(t)}{k-\lambda_{n}(t)}\right)^{2}=\lambda_{n}(t)^{2}\psi'\left(-\lambda_{n}(t)\right)
\]
where $\psi,$ known as the $\psi$ function, is the logarithmic derivative
\[
\psi(t)=\frac{\Gamma'(t)}{\Gamma(t)}
\]
of the gamma function $\Gamma(t)=\int_{0}^{\infty}e^{-x}x^{t-1}dx.$
The derivatives of the $\psi$ function are called polygamma functions
\cite{SB11}.\end{rem}
\begin{thm}
\label{thm:n-asymptotic}Let $t$ be fixed such that $\cos(t)\neq-1.$
Let $\lambda_{n}(t),$ $n\in\mathbb{N}_{0},$ be the eigenvalues of
$L_{t}$ enumerated as in Theorem \ref{thm:Spectrum-L_t}, then 
\[
n-\lambda_{n}(t)\to0
\]
as $n\to\infty.$ \end{thm}
\begin{proof}
The $\psi$ function has the series expansion \cite{AS92}, 
\begin{equation}
\psi(z)=-\gamma+\sum_{k=0}^{\infty}\left(\frac{1}{k+1}-\frac{1}{k+z}\right),\label{eq:psi-series}
\end{equation}
where $\gamma$ is the Euler constant. Consequently, a computation
shows that 
\begin{eqnarray}
G(\lambda) & = & \sum_{k\in\mathbb{N}_{0}}\frac{1+\lambda k}{\left(k-\lambda\right)\left(1+k^{2}\right)}\nonumber \\
 & = & \frac{1}{2}\left(\psi(i)+\psi(-i)-2\psi(-\lambda)\right)\nonumber \\
 & = & \Re\left(\psi(i)\right)-\psi(-\lambda).\label{eq:psipsip}
\end{eqnarray}
Using this identity and (\ref{eq:psi-series}) we find 
\begin{eqnarray}
G(\lambda)-G(\lambda-1) & = & G(\frac{1}{2})-G(-\frac{1}{2})+\int_{\frac{1}{2}}^{\lambda}\frac{d}{dx}\left(G(x)-G(x-1)\right)dx\nonumber \\
 & = & \sum_{k\in\mathbb{N}_{0}}\frac{1+\frac{1}{2}k}{\left(k-\frac{1}{2}\right)\left(1+k^{2}\right)}-\sum_{k\in\mathbb{N}_{0}}\frac{1-\frac{1}{2}k}{\left(k+\frac{1}{2}\right)\left(1+k^{2}\right)}\nonumber \\
 &  & +\int_{\frac{1}{2}}^{\lambda}\left(\sum_{k\in\mathbb{N}_{0}}\frac{1}{\left(k-x\right)^{2}}-\sum_{k\in\mathbb{N}_{0}}\frac{1}{\left(k+1-x\right)^{2}}\right)dx\nonumber \\
 & = & \sum_{k\in\mathbb{N}_{0}}\frac{4}{4k^{2}-1}+\int_{\frac{1}{2}}^{\lambda}\frac{1}{x^{2}}\, dx\nonumber \\
 & = & -2+\left(2-\lambda^{-1}\right)=-\lambda^{-1}.\label{eq:psidiff}
\end{eqnarray}
\textbf{Caution.} Note that in equating the two sides, it is understood
that the common poles on the LHS have been cancelled, so only one
contribution from a pole remains, from the pole at $\lambda=0$. While
cancellation of poles is $\infty-\infty$, nonetheless, our assertion
is precise because we verify agreement of the residues at the poles
that are subtracted.

With this caution, we note that (\ref{eq:psipsip}) and (\ref{eq:psidiff})
extend to $\mathbb{C}$, i.e., that both equations now will be valid
for all complex values of $\lambda$.

Hence, if $\lambda=n+\varepsilon$ with $n\geq0$ and $0<\varepsilon<1,$
then
\begin{align*}
G(\lambda)=G(n+\varepsilon) & =G(\varepsilon)-\frac{1}{1+\varepsilon}-\frac{1}{2+\varepsilon}-\cdots-\frac{1}{n+\varepsilon}\\
 & <G(\varepsilon)-\frac{1}{2}-\frac{1}{3}-\cdots-\frac{1}{n+1}.
\end{align*}
 Since $\sum_{k=2}^{\infty}1/k$ is divergent the result follows. \end{proof}
\begin{rem}
The proof of Theorem \ref{thm:n-asymptotic} shows, the graph of $G(\lambda),$
$n<\lambda<n+1$, is obtained from the graph of $G(\lambda),$ $n-1<\lambda<n,$
by shifting the point $(\lambda-1,G(\lambda-1))$ to the point $\left(\lambda,G(\lambda-1)-\tfrac{1}{\lambda}\right)$. \end{rem}
\begin{lem}
\label{lem:psireim}The function $\mathbb{C}\ni z\mapsto\psi(z)$
in (\ref{eq:psi-series}) in Theorem \ref{thm:n-asymptotic} is meromorphic
with simple poles at $z\in-\mathbb{N}_{0}$, all residues are $-1$;
and $\psi(\cdot)$ is reflection symmetric, i.e., $\psi(\overline{z})=\overline{\psi(z)}$,
$z\in\mathbb{C}$. 

Moreover, for the two functions $\Re\{\psi\}$ and $\Im\{\psi\}$,
we have the following formulas 
\begin{eqnarray}
\Re\{\psi(x+iy)\} & = & -\gamma+\sum_{k\in\mathbb{N}_{0}}\frac{(x-1)(x+k)+y^{2}}{\left(k+1\right)(\left(x+k\right)^{2}+y^{2})}\label{eq:psire}\\
\Im\{\psi(x+iy)\} & = & \sum_{k\in\mathbb{N}_{0}}\frac{y}{\left(x+k\right)^{2}+y^{2}}\label{eq:psiim}
\end{eqnarray}
for all $x+iy\in\mathbb{C}$; where $\gamma$ is the Euler\textendash{}Mascheroni
constant. (Note the occurrence of the kernels of Poisson and of the
Hilbert transform. Indeed, the formula (\ref{eq:psiim}) for $\Im\{\psi(z)\}$
is a sampling of the Poisson kernel for the upper half plane, sampled
on the $x$-axis with points from $\mathbb{N}_{0}$ as sample points.)

In particular the numbers 
\begin{eqnarray}
\Re\{\psi(i)\} & = & -\gamma+\sum_{k\in\mathbb{N}_{0}}\frac{1-k}{\left(k+1\right)\left(k^{2}+1\right)}\thickapprox0.0946503\label{eq:repsii}\\
\Im\{\psi(i)\} & = & \sum_{k\in\mathbb{N}_{0}}\frac{1}{k^{2}+1}=\frac{1}{2}\left(1+\pi\coth(\pi)\right)=K;\label{eq:impsii}
\end{eqnarray}
see (\ref{eq:eeqn}).\end{lem}
\begin{proof}
The first assertion in the lemma follows from (\ref{eq:psi-series})
and the following reduction:
\begin{equation}
\psi(z)=-\gamma+(z-1)\sum_{k\in\mathbb{N}_{0}}\frac{1}{\left(k+z\right)\left(k+1\right)}.\label{eq:psired}
\end{equation}
Fix $n\in\mathbb{N}_{0}$; then
\[
\lim_{z\rightarrow-n}(z+n)\psi(z)=-1.
\]

From (\ref{eq:psi-series}), we see that
\begin{eqnarray*}
\psi(x+iy) & = & -\gamma+\sum_{k\in\mathbb{N}_{0}}\left(\frac{1}{k+1}-\frac{1}{\left(x+k\right)+iy}\right)\\
 & = & -\gamma+\sum_{k\in\mathbb{N}_{0}}\left(\frac{1}{k+1}-\frac{x+k}{\left(x+k\right)^{2}+y^{2}}\right)+i\sum_{k\in\mathbb{N}_{0}}\frac{y}{\left(x+k\right)^{2}+y^{2}}\\
 & = & -\gamma+\sum_{k\in\mathbb{N}_{0}}\frac{(x-1)(x+k)+y^{2}}{\left(k+1\right)(\left(x+k\right)^{2}+y^{2})}+i\sum_{k\in\mathbb{N}_{0}}\frac{y}{\left(x+k\right)^{2}+y^{2}}.
\end{eqnarray*}
The desired results follow from this.\end{proof}
\begin{rem}
As a consequence of the Lemma \ref{lem:psireim}, we get the following: 

For every 
\begin{equation}
v:=K\tan(\pi t)\in\mathbb{R},\;(\mbox{see }(\ref{eq:eeqn}))\label{eq:vKtan}
\end{equation}
and every $n\in\mathbb{N}$, let 
\[
\mathbb{D}_{v,n}:=\left\{ z\in\mathbb{C}\::\:\left|z-(n-\frac{1}{2})\right|<r_{n}\right\} 
\]
where $r_{n}\in(\frac{1}{2},1)$, and such that $G(z)-v$ has only
one zero in $\mathbb{D}_{v,n}$. Then, by the argument principle,
\begin{eqnarray}
\frac{1}{2\pi i}\oint_{\partial\mathbb{D}_{v,n}}\frac{zG'(z)}{G(z)-v}dz & = & \sum(\mbox{zeros})-\sum(\mbox{poles})\label{eq:cint}\\
 & = & \lambda_{n}(t)-\left(2n-1\right).\label{eq:cintt}
\end{eqnarray}
See Figure \ref{fig:eigen}. Therefore, 
\begin{equation}
\lambda_{n}(t)=(2n-1)+\frac{1}{2\pi i}\oint_{\partial\mathbb{D}_{v,n}}\frac{zG'(z)}{G(z)-v}dz,\quad n\in\mathbb{N}.\label{eq:inttt}
\end{equation}
\bigskip{}

\end{rem}
The proof of Theorem \ref{thm:Spectrum-L_t} shows that for each $n,$
$t\to\lambda_{n}(t)$ is a continuous function $\lambda_{n}:(-\pi,\pi)\to\mathbb{R},$
when $n\geq1$ and $\lambda_{0}:(-\infty,0)\to\mathbb{R},$ such that
\begin{alignat}{1}
\lambda_{n}(t)\nearrow n & \text{ as }t\nearrow\pi\,\,\,\,\,\,\,\,\,\,\text{for all }n\geq0\nonumber \\
\lambda_{n}(t)\searrow n-1 & \text{ as }t\searrow-\pi\,\,\,\,\,\text{for all }n\geq1,\text{ and}\label{Sec-4-eq:Spectral-Limits}\\
\lambda_{0}(t)\searrow-\infty & \text{ as }t\searrow-\pi.\nonumber 
\end{alignat}
The case $t=\pm\pi$ has spectrum $\mathbb{N}_{0}.$ In particular,
\[
\bigcup_{\zeta\in\mathbb{T}}\mathrm{spectrum}\left(H_{\zeta}\right)=\bigcup_{t\in(-\pi,\pi]}\mathrm{spectrum}\left(L_{t}\right)=\mathbb{R}
\]
and the unions are disjoint. 

In the proof of Theorem \ref{thm:n-asymptotic} we saw that the $\lambda_{0}(t)\searrow-\infty\text{ as }t\searrow-\pi$
is logarithmic. The purpose of the next result is to establish rates
of convergence for the remaining limits in (\ref{Sec-4-eq:Spectral-Limits}). 
\begin{thm}
Let $\lambda_{n}(t),$ $n\in\mathbb{N}_{0},$ be the eigenvalues of
$L_{t}$ enumerated as in Theorem \ref{thm:Spectrum-L_t} and let
$\gamma_{0}=2\gamma+\psi(i)+\psi(-i),$ where $\gamma$ is the Euler
constant and $\psi$ is the digamma function, then $\gamma_{0}\approx-1.34373$
and for $-\pi<t<\pi$ we have 
\begin{align*}
\lambda_{0}(t) & \approx-\frac{\pi-t}{2K}-\frac{\gamma_{0}}{2}\left(\frac{\pi-t}{2K}\right)^{2},\text{ when }t\approx\pi\\
\lambda_{n}(t) & \approx n-\frac{\pi-t}{2K}+\left(-\gamma_{0}+2\sum_{k=1}^{n}\frac{1}{k}\right)\left(\frac{\pi-t}{2K}\right)^{2},\text{ when }t\approx\pi\\
\lambda_{n+1}(t) & \approx n+\frac{\pi+t}{2K}+\left(-\gamma_{0}+2\sum_{k=1}^{n}\frac{1}{k}\right)\left(\frac{\pi+t}{2K}\right)^{2},\text{ when \ensuremath{t\approx-\pi}}
\end{align*}
for all $n\geq1.$ Here $K=\tfrac{1}{2}\left(1+\pi\coth(\pi)\right)\approx2.07667$. \end{thm}
\begin{proof}
Recall, see Figure \ref{App-fig:G}, when $-\pi<t<\pi,$ the eigenvalues
of $L_{t}$ are the solutions $-\infty<\lambda_{0}(t)<0,$ and for
$n\geq1,$ $n-1<\lambda_{n}(t)<n$ to 
\begin{equation}
G\left(\lambda_{n}(t)\right)=K\tan(t/2)\label{eq:eigenvalue}
\end{equation}
and 

\begin{eqnarray}
G(\lambda)=\sum_{k=0}^{\infty}\frac{1+\lambda k}{\left(k-\lambda\right)\left(1+k^{2}\right)} & = & \frac{1}{2}\left(\psi(i)+\psi(-i)-2\psi(-\lambda)\right)\nonumber \\
 & = & \Re\{\psi(i)\}-\psi(-\lambda).\label{eq:G-psi}
\end{eqnarray}
Here 
\begin{equation}
\psi(z)=-\gamma+\sum_{k=0}^{\infty}\frac{z-1}{(k+1)(k+z)}\label{eq:tmpsi}
\end{equation}
is the digamma function and $\gamma$ is the Euler constant. 

\begin{figure}
\label{App-fig:G}

\includegraphics[width=4in]{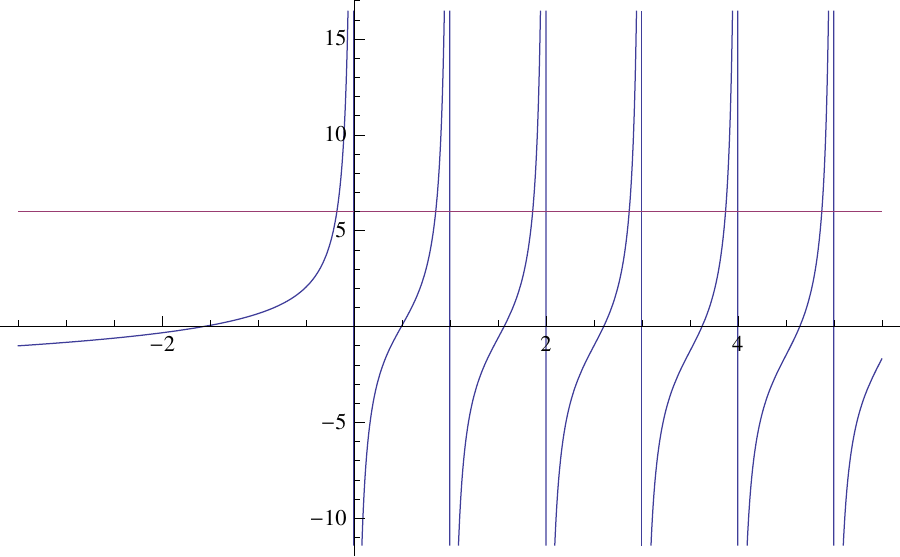}

\caption{$G(\lambda),K\tan(t/2)=6$}
\end{figure}

Since we would like to find the asymptotics of $t\to\lambda_{n}(t)$
at the asymptotes, it is convenient to rewrite (\ref{eq:eigenvalue})
as 
\begin{equation}
g(\lambda)=\cot(t/2)/K,\label{eq:eigenvalue2}
\end{equation}
where $g(\lambda)=1/G(\lambda)$ when $\lambda\notin\mathbb{N}_{0}$
and $g(\lambda)=0$ when $\lambda\in\mathbb{N}_{0}.$ Then $g$ is
analytic in a neighborhood of $n$ for all $n\in\mathbb{N}_{0}.$
Note $\cot(t/2)\to0$ as $t\to\pm\pi.$ Let $I_{n}$ be the open interval
containing $n$ whose endpoints are roots of $G(\lambda)$ and let
\[
h_{n}(x)=\left(g\mid_{I_{n}}\right)^{-1}(x).
\]
For each $n\geq0,$ $h_{n}$ is the inverse of a restriction of $g$
satisfying $h_{n}(0)=n.$ 

By construction of $h_{n}$ we can rewrite (\ref{eq:eigenvalue2})
as 
\[
h_{n}\left(\cot(t/2)/K\right)=\begin{cases}
\lambda_{n+1}(t), & \text{when }0<t<\pi\\
\lambda_{n}(t), & \text{when }-\pi<t<0
\end{cases},
\]
for $n=0,1,2,\ldots$Hence, to investigate the asymptotics of $\lambda_{n}(t)$
as $t\nearrow\pi$ and as $t\searrow-\pi,$ we need to investigate
the asymptotics of $h_{n}\left(\cot(t/2)/K\right)$ as $t\nearrow\pi$
and as $t\searrow-\pi.$ We will do this by writing down a few terms
of the Taylor series at $0$ of $h_{n}$ for each $n.$ 

\begin{figure}
\includegraphics[width=4in]{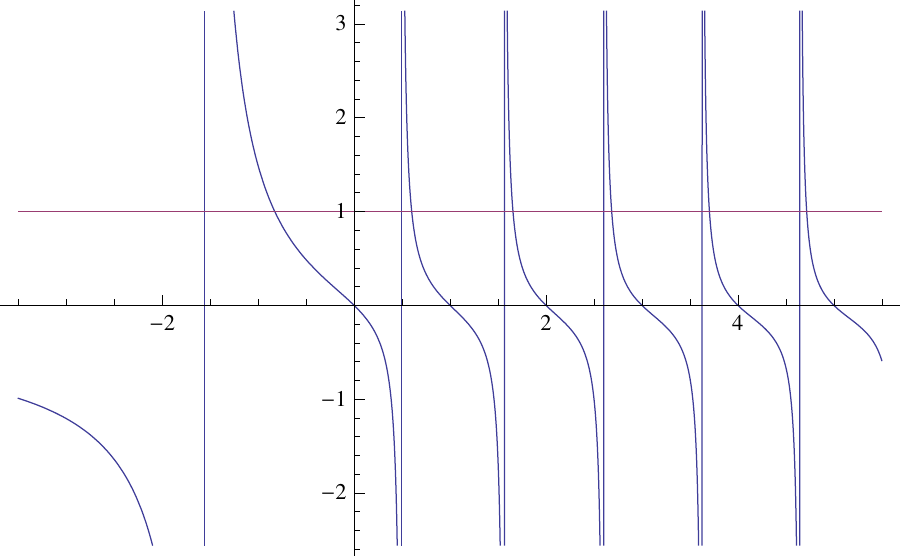}

\caption{$g(\lambda)=1/G(\lambda),\cot(t/2)=K$}
\end{figure}
As noted above $h_{n}(0)=n.$ It is easy to see that 
\begin{align*}
h_{n}'(x) & =\frac{1}{g'\left(h_{n}(x)\right)}\text{ and}\\
h_{n}''(x) & =-\frac{g''\left(h_{n}(x)\right)}{\left(g'\left(h_{n}(x)\right)\right)^{3}}.
\end{align*}
Hence, using that $h_{n}(0)=n,$ we see that to calculate the derivatives
of $h_{n}(x),$ at $x=0,$ it is sufficient to calculate the derivatives
\[
g'(n),g''(n),\ldots.
\]
Using (\ref{eq:G-psi}) and $g(x)=1/G(x),$ we get $g(x)=2/\left(\psi(i)+\psi(-i)-2\psi(-x)\right),$
so
\begin{align}
g'(x) & =\frac{-4\psi'(-x)}{\left(\psi(i)+\psi(-i)-2\psi(-x)\right)^{2}}\nonumber \\
 & =\frac{-\psi'(x)}{\left(\Re\left\{ \psi(i)\right\} -\psi(-x)\right)^{2}},\label{eq:gp}\\
\nonumber 
\end{align}
and
\begin{eqnarray}
g''(x) & = & \frac{4\psi''(-x)\left(\psi(i)+\psi(-i)-2\psi(-x)\right)+16\left(\psi'(-x)\right)^{2}}{\left(\psi(i)+\psi(-i)-2\psi(-x)\right)^{3}}.\nonumber \\
 & = & \frac{\psi''(-x)\left(\Re\{\psi(i)\}-\psi(-x)\right)+2\left(\psi'(-x)\right)^{2}}{\left(\Re\left\{ \psi(i)\right\} -\psi(-x)\right)^{3}}.\label{eq:gpp}
\end{eqnarray}

To calculate $\psi''(n)$ we need some more information of $\psi(x).$
By (\ref{eq:psi-series}) 
\begin{equation}
\psi(-x)=-\gamma+\frac{1}{x}+\sum_{k=1}^{\infty}\left(\frac{1}{k}-\frac{1}{k-x}\right).\label{eq:psi-series-2}
\end{equation}
 Hence
\begin{equation}
\psi'(-x)=\sum_{k=0}^{\infty}\frac{1}{(k-x)^{2}}\text{ and }\psi''(-x)=-2\sum_{k=0}^{\infty}\frac{1}{(k-x)^{3}}.\label{eq:psi-p-pp}
\end{equation}
Plugging (\ref{eq:psi-series-2}) and (\ref{eq:psi-p-pp}) into (\ref{eq:gp})
we see that the numerator and denominator are meromorphic functions
with poles of the same order at $x=n.$ It follows that
\[
g'(n)=-1,\text{ for all }n\geq0.
\]
Similarly, it follows from (\ref{eq:psi-series-2}), (\ref{eq:psi-p-pp}),
and (\ref{eq:gpp}) that 
\begin{align*}
g''(0) & =-2\gamma-\psi(i)-\psi(-i)=-\gamma_{0}\\
g''(n) & =-\gamma_{0}+2\sum_{k=1}^{n}\frac{1}{k},\text{ for all }n\geq1.
\end{align*}
So that $h_{n}'(0)=-1$ for all $n\geq0,$ $h_{0}''(0)=-\gamma_{0},$
and $h''_{n}(0)=-\gamma_{0}+2\sum_{k=1}^{n}\frac{1}{k}.$ Hence, 
\[
\lambda_{0}(t)\approx-\frac{\cot(t/2)}{K}-\frac{\gamma_{0}}{2}\left(\frac{\cot(t/2)}{K}\right)^{2},\text{ when }t\approx\pi
\]
and for all $n\geq1,$
\begin{align*}
\lambda_{n}(t)\approx n-\frac{\cot(t/2)}{K}+\left(-\gamma_{0}+2\sum_{k=1}^{n}\frac{1}{k}\right)\left(\frac{\cot(t/2)}{K}\right)^{2}, & \text{ when \ensuremath{t\approx\pi}}\\
\lambda_{n+1}(t)\approx n-\frac{\cot(t/2)}{K}+\left(-\gamma_{0}+2\sum_{k=1}^{n}\frac{1}{k}\right)\left(\frac{\cot(t/2)}{K}\right)^{2} & \text{ when \ensuremath{t\approx-\pi.}}
\end{align*}
Using 
\begin{align*}
\cot(t/2) & =\frac{\pi-t}{2}+\frac{\left(\pi-t\right)^{3}}{24}+\frac{\left(\pi-t\right)^{7}}{240}+O\left((\pi-t)^{7}\right)\\
 & =-\frac{\pi+t}{2}+\frac{\left(\pi+t\right)^{3}}{24}+\frac{\left(\pi+t\right)^{7}}{240}+O\left((\pi+t)^{7}\right),
\end{align*}
completes the proof.
\end{proof}

\section{\label{sec:qua}Quadratic Forms}

In this section we show that the basic Hermitian operator $L$ from
section \ref{sec:sbdd} has zero as its lower bound, i.e., that no
positive number is a lower bound for $L$.

Let 
\[
Q_{L}(x)=\left\langle x,Lx\right\rangle =\sum_{k\in\mathbb{N}_{0}}k\left|x_{k}\right|^{2},\quad x\in\mathbb{D}_{0}.
\]

\begin{lem}
\label{lem:lbdd}The quadratic form $Q_{L}$ has greatest lower bound
zero.\end{lem}
\begin{proof}
Clearly, $Q_{L}(x)\geq0$ for all $x\in\mathscr{D}(L)$. What is the
largest constant $c$ such that 
\[
Q_{L}(x)\geq c\left\langle x,x\right\rangle ,\quad x\in\mathscr{D}(L)\:?
\]
Clearly, $c\geq0.$ 

Next we will minimize 
\[
\sum_{k\in\mathbb{N}_{0}}k\left|x_{k}\right|^{2}
\]
subject to the constraints
\begin{equation}
\sum_{k\in\mathbb{N}_{0}}\left|x_{k}\right|^{2}=1\text{ and }\sum_{k\in\mathbb{N}_{0}}x_{k}=0.\label{eq:Constraints}
\end{equation}
Indeed, we prove $\mbox{GLB}(L)=0$. 

To do this, set ${\displaystyle s_{n}=\sum_{i=1}^{n}\frac{1}{i}}$,
$n=1,2,.\dots$, and 
\[
x_{j}^{(n)}=\begin{cases}
-s_{n} & \mbox{if }j=0\\
\frac{1}{j} & \mbox{if }1\leq j\leq n\\
0 & \mbox{if }j>n.
\end{cases}
\]
Note $(x^{(n)})\in\mathbb{D}_{0}=\mathscr{D}(L)$, for all $n$. Then
\begin{alignat*}{1}
\frac{Q_{L}(x^{(n)})}{\left\Vert x^{(n)}\right\Vert _{2}^{2}} & =\frac{\sum_{i=1}^{n}i\left(x_{i}^{(n)}\right)^{2}}{\left(\sum_{i=1}^{n}x_{i}^{(n)}\right)^{2}+\sum_{i=1}^{n}\left(x_{i}^{(n)}\right)^{2}}\\
 & =\frac{\sum_{i=1}^{n}\frac{1}{i}}{\left(\sum_{i=1}^{n}\frac{1}{i}\right)^{2}+\sum_{i=1}^{n}\frac{1}{i^{2}}}\\
 & =\frac{s_{n}}{\left(s_{n}\right)^{2}+\sum_{i=1}^{n}\frac{1}{i^{2}}}\\
 & \sim\frac{1}{s_{n}}\rightarrow0.
\end{alignat*}
Recall that 
\begin{alignat*}{1}
\lim_{n\rightarrow\infty}\sum_{i=1}^{n}\frac{1}{i} & =\lim_{n\rightarrow\infty}s_{n}=\infty\;\mbox{and}\\
\lim_{n\rightarrow\infty}\sum_{i=1}^{n}\frac{1}{i^{2}} & =\frac{\pi^{2}}{6}.
\end{alignat*}

\end{proof}

\subsection{Lower Bounds for Restrictions}

The result from Lemma \ref{lem:lbdd} in the present section, while
dealing with an example, illustrates a more general question: Consider
a selfadjoint operator $H$ in a Hilbert space $\mathscr{H}$ having
its spectrum $spec(H)$ contained in the halfline $[0,\infty)$, and
with $0\in spec(H)$. Then every densely defined Hermitian restriction
$L$ of $H$ will define a quadratic form $Q_{L}$ also having $0$
as a lower bound. But, in general, the greatest lower bound (GLB)
for such a restriction may well be strictly positive. 

The particular restriction $L$ in Lemma \ref{lem:lbdd} does have
$\mbox{GLB}(Q_{L})=0$; and this coincidence of lower bounds will
persist for a general family of cases to be considered in section
\ref{sec:Hardy}. 

Nonetheless, as we show below, there are other related semibounded
selfadjoint operators $H$ with $0$ in the bottom of $spec(H)$,
and having densely defined restrictions $L$ such that $\mbox{GLB}(Q_{L})$
is strictly positive. We now outline such a class of examples. The
deficiency indices will be $(2,2)$.

We begin by specifying the selfadjoint operator $H$ and then identifying
its restriction $L$. We do this by identifying $\mathscr{D}(L)$
as a subspace in $\mathscr{D}(H)$, still with $\mathscr{D}(L)$ dense
in the ambient Hilbert space $\mathscr{H}$. But to analyze the operators
we will have occasion to switch between two different ONBs. To do
this, it will be convenient to realize vectors in $\mathscr{H}$ in
cosine and sine-Fourier bases for $L^{2}(0,\pi)$. In this form, our
operators may be specified as $-(d/dx)^{2}$ with suitable boundary
conditions. The reason for the minus-sign is to make the operators
semibounded.

But establishing the stated bounds is subtle, and inside the arguments,
we will need to alternate between the two Fourier bases in $L^{2}(0,\pi)$.
Indeed, inside the proof we switch between the two ONBs. This allows
us to prove $\mbox{GLB}(Q_{L})=1$, i.e., establishing the best lower
bound for the restriction $L$; strictly larger than the bound for
$H$.

Hence vectors in our Hilbert space $\mathscr{H}$ will have equivalent
presentations both the form of $l^{2}$ (square-summable sequences,
one of each of the two orthogonal bases) and of $L^{2}(0,\pi)$. To
get a cosine-Fourier representation for a function $f\in L^{2}(0,\pi)$,
make an even extension $F_{ev}$ of $f$, i.e., extending $f$ to
$(-\pi,\pi)$, then make a cosine-Fourier series for $F_{ev}$, and
restrict it back to $(0,\pi)$. To get a sine-Fourier series for $f$,
do the same but now using instead an odd extension $F_{odd}$ of $f$
to $(-\pi,\pi)$.

Let $\mathscr{H}^{(s)}:=l^{2}(\mathbb{N}_{0})$, with the standard
basis
\begin{equation}
e_{k}=(0,\ldots,0,\underset{k^{th}}{1},0,\ldots),\quad k\in\mathbb{N}_{0}.\label{eq:stdonb}
\end{equation}
Let $H$ be the selfadjoint operator in $\mathscr{H}^{(s)}$ specified
by 
\begin{eqnarray}
\mathscr{D}^{(s)}(H) & = & \left\{ \alpha=(a_{n})\in l^{2}\:\big|\:\left(n^{2}a_{n}\right)\in l^{2},\; n\in\mathbb{N}_{0}\right\} \label{eq:DomH}\\
\left(H\alpha\right)_{n} & = & n^{2}a_{n},\;\forall\alpha\in\mathscr{D}(H).\label{eq:defH}
\end{eqnarray}
In particular, $He_{0}=0$, and so $\inf\{spec(H)\}=0$. 

Consider the restriction 
\begin{equation}
L\subset H\label{eq:LsubH}
\end{equation}
with domain

\begin{equation}
\mathscr{D}_{1}^{(s)}:=\left\{ \left(a_{n}\right)\in\mathscr{D}(H)\:\Big|\:\sum_{n=0}^{\infty}a_{2n}=\sum_{n=0}^{\infty}a_{2n+1}=0\right\} \label{eq:domLc}
\end{equation}

Let $\mathscr{H}^{(c)}:=L^{2}(0,\pi)$. Using Fourier series, there
is a natural isometric isomorphism 
\begin{equation}
\mathscr{H}^{(s)}\simeq\mathscr{H}^{(c)},\label{eq:Fisoiso}
\end{equation}
corresponding to the two orthogonal bases in $L^{2}(0,\pi)$. Recall
that, for all $f\in L^{2}(0,\pi)$, 
\begin{eqnarray}
f(x) & = & \sum_{n=0}^{\infty}a_{n}\cos(nx)\label{eq:cext}\\
 & = & \sum_{n=1}^{\infty}b_{n}\sin(nx).\label{eq:sext}
\end{eqnarray}
Note that the two right-hand sides in (\ref{eq:cext}) and (\ref{eq:sext})
correspond to even and odd $2\pi$-periodic extensions of $f(x)$. 
\begin{rem}
In (\ref{eq:domLc}), $e_{0}\in\mathscr{D}(H)\backslash\mathscr{D}_{1}^{(s)}$.
To see this, note that while the constant function $f_{0}\equiv1/\sqrt{\pi}$
on $(0,\pi)$ has $e_{0}$ as its cos-representation via (\ref{eq:cext}),
its sin-representation $(b_{n})$ via (\ref{eq:sext}) is as follows:
\[
b_{n}=\begin{cases}
\sqrt{\frac{3}{\pi}}\frac{1}{1+2k} & \mbox{ if }n=1+2k,\:\mbox{odd, and}\\
0 & \mbox{ if }n=2k,\:\mbox{even}.
\end{cases}
\]
\end{rem}
\begin{lem}
\label{lem:Lapbdd}Let $f\in L^{2}(0,\pi)$, and assume that $f'$
and $f''\in L^{2}(0,\pi)$; then the combined boundary conditions
\begin{equation}
f=f'=0\quad\mbox{at the two endpoints }x=0\mbox{ and }x=\pi\label{eq:cbdend}
\end{equation}
take the form 
\begin{equation}
\sum_{n=0}^{\infty}a_{2n}=\sum_{n=0}^{\infty}a_{1+2n}=0\label{eq:cbd}
\end{equation}
using the cos-representation (\ref{eq:cext}), while in the sin-representation
(\ref{eq:sext}) for $f$, the same conditions (\ref{eq:cbdend})
take the equivalent form:
\begin{equation}
\sum_{n=1}^{\infty}n\, b_{2n}=\sum_{n=0}^{\infty}(1+2n)b_{1+2n}=0.\label{eq:cbds}
\end{equation}
\end{lem}
\begin{proof}
Direct substitution of (\ref{eq:cext}) into (\ref{eq:cbdend}) yields
\[
\sum_{n\in\mathbb{N}_{0}}a_{n}=\sum_{n\in\mathbb{N}_{0}}(-1)^{n}a_{n}=0
\]
which simplifies to (\ref{eq:cbd}). If instead we substitute (\ref{eq:sext})
into (\ref{eq:cbdend}), we get
\[
\sum_{n\in\mathbb{N}}n\, b_{n}=\sum_{n\in\mathbb{N}}(-1)^{n}n\, b_{n}=0
\]
which simplifies into (\ref{eq:cbds}).\end{proof}
\begin{defn}
\label{def:domL}We set $\mathscr{D}\subset L^{2}(0,\pi)$ to be the
dense subspace given by any one of the three equivalent systems of
conditions (\ref{eq:cbdend}), (\ref{eq:cbd}), and (\ref{eq:cbds}).\end{defn}
\begin{lem}
\label{lem:Lapres}Let $\mathscr{H}=L^{2}(0,\pi)$ be as above, and
let $\mathscr{D}$ be the dense subspace specified in Definition \ref{def:domL}.
We set
\begin{equation}
H=-\left(\frac{d}{dx}\right)^{2}\bigg|_{\left\{ f\:\big|\: f=\sum_{n\in\mathbb{N}_{0}}a_{n}\cos(nx),\:(n^{2}a_{n})\in l^{2}\right\} }\label{eq:lap}
\end{equation}
i.e., restriction. (This is the Neumann-operator.) 

It follows that $H$ is selfadjoint with spectrum $spec(H)=\left\{ n^{2}\:\big|\: n\in\mathbb{N}_{0}\right\} $,
and we set
\begin{equation}
L:=H\big|_{\mathscr{D}}.\label{eq:sublap}
\end{equation}
Then $L$ is a densely defined restriction of $H$ and its deficiency
indices are $(2,2)$. \end{lem}
\begin{proof}
See Lemma \ref{lem:semibdd} and the discussion above. Note that $\mathscr{D}$
arises from $\mathscr{D}(H)$ by imposition of the two additional
linear conditions (\ref{eq:cbd}). Hence the deficiency indices are
$(2,2)$.

Under (\ref{eq:Fisoiso}), we have the following correspondence:
\begin{eqnarray*}
\mathscr{D}^{(s)} & \longleftrightarrow & \mathscr{D}^{(c)}:=\left\{ f,f''\in\mathscr{H}^{(c)}\right\} ;\\
\mathscr{D}_{1}^{(s)} & \longleftrightarrow & \mathscr{D}_{1}^{(c)}:=\left\{ f\in\mathscr{H}^{(c)}\:\big|\: f(0)=f(\pi)=0\right\} \,\,\,\,\,\,\,\,\,(\mbox{Dirichlet})\\
\mathscr{D}_{2}^{(s)} & \longleftrightarrow & \mathscr{D}_{2}^{(c)}:=\left\{ f\in\mathscr{H}^{(c)}\:\big|\: f'(0)=f'(\pi)=0\right\} \,\,\,\,\,\,(\mbox{Neumann})
\end{eqnarray*}
Moreover, 
\[
(H,\mathscr{D}^{(s)})\longleftrightarrow\left(-\left(d/dx\right)^{2},\mathscr{D}^{(c)}\right)\quad\mbox{selfadjoint}
\]
and for the restriction operator, 
\[
(L,\mathscr{D})\longleftrightarrow(-\left(d/dx\right)^{2},\mathscr{D})
\]
where $\mathscr{D}$ is the dense subspace in Definition \ref{def:domL}. \end{proof}
\begin{lem}
Let $(L,\mathscr{D})$ be the restriction of $H$ to $\mathscr{D}$.
Then 
\begin{equation}
\inf\left\{ \frac{\left\langle \alpha,L\alpha\right\rangle _{2}}{\left\Vert \alpha\right\Vert _{2}^{2}}\:\Big|\:\alpha\in\mathscr{D}\right\} =1.\label{eq:GLBLsubH}
\end{equation}
\end{lem}
\begin{proof}
For all $\alpha=(a_{n})\in\mathscr{D}^{(s)}$, let 
\[
f_{\alpha}(x):=\sum_{n=0}^{\infty}a_{n}\cos(nx)=\sum_{n=1}^{\infty}b_{n}\sin(nx)
\]
where we have switched in $\mathscr{H}^{(c)}=L^{2}(0,\pi)$ from the
cosine basis to sine basis, $\left(a_{n}\right)\mapsto\left(b_{n}\right)$.
Then, for all $\alpha\in\mathscr{D}$, we have: 
\begin{eqnarray}
\left\langle \alpha,L\alpha\right\rangle _{l^{2}} & = & \left\langle f_{\alpha},-\left(d/dx\right)^{2}f_{\alpha}\right\rangle _{\mathscr{H}^{(c)}}\nonumber \\
 & = & \sum_{n=1}^{\infty}n^{2}\left|b_{n}\right|^{2}\geq\sum_{n=1}^{\infty}\left|b_{n}\right|^{2}=\left\Vert \alpha\right\Vert _{l^{2}}^{2}.\label{eq:lbdd}
\end{eqnarray}
The desired result follows.
\end{proof}

\begin{rem}
Note that (\ref{eq:lbdd}) may be restated as a Poincaré inequality
\cite{AB09,YL11} as follows:

Let $f\in L^{2}(0,\pi)$ be such that $f'$ and $f''$ are in $L^{2}$,
and $f=f'=0$ at the endpoints $x=0$, and $x=\pi$; then
\begin{equation}
\int_{0}^{\pi}\left|f'(x)\right|^{2}dx\geq\int_{0}^{\pi}\left|f(x)\right|^{2}dx.\label{eq:pineq}
\end{equation}

\end{rem}

\section{\label{sec:Hardy}The Hardy space $\mathscr{H}_{2}$. }

It is well known that there are two mirror-image versions of the basic
$l^{2}$-Hilbert space $l^{2}(\mathbb{N}_{0})$ of one-sided square-summable
sequences: On one side of the mirror we have plain $l^{2}(\mathbb{N}_{0})$,
the discrete version; and on the other, there is the Hardy space $\mathscr{H}_{2}$
of functions $f(z)$, analytic in the open disk $\mathbb{D}$ in the
complex plane, and represented with coefficients from $l^{2}(\mathbb{N}_{0})$.
(See (\ref{eq:f(z)})-(\ref{eq:H2norm}) below.) By \textquotedblleft{}mirror-image\textquotedblright{}
we are here referring to a familiar unitary equivalence between the
two sides, see \cite{Rud87}. This other viewpoint, involving complex
power series, further makes useful connections to special function
theory. In this connection, the following monographs \cite{AS92,EMOT81}
are especially relevant to our discussion below. 

Introducing the analytic version $\mathscr{H}_{2}$ further allows
us to bring to bear on our problem powerful tools from reproducing
kernel theory, from harmonic analysis and analytic function theory
(see e.g., \cite{Rud87}). The reproducing kernel for $\mathscr{H}_{2}$
is the familiar Szegö kernel. This then further allows us to assign
a geometric meaning to our boundary value problems, formulated initially
in the language of von Neumann deficiency spaces. In the context of
geometric measure theory, the reproducing kernel approach was used
in \cite{DJ11} in a related but different context.

Under the natural isometric isomorphism of $l^{2}$ onto $\mathscr{H}_{2}$,
the domain $\mathscr{D}(H)$ in $l^{2}(\mathbb{N}_{0})$ is mapped
into a subalgebra $\mathscr{A}(\mathscr{H}_{2})$ in $\mathscr{H}_{2}$,
a Banach algebra, consisting of functions on the complex disk having
continuous extensions to the closure of the disk $\overline{\mathbb{D}}$,
and with absolutely convergent power series. 
\begin{lem}
\label{lem:isoH2}Under the isomorphism $l^{2}\simeq\mathscr{H}_{2}$,
the selfadjoint operator $H$ becomes ${\displaystyle z\frac{d}{dz}}$,
and the domain of its restriction $L$ consists of continuous functions
$f$ on $\overline{\mathbb{D}}$, analytic in $\mathbb{D}$, $f\in\mathscr{A}(\mathscr{H}_{2})$,
such that $f(1)=0$.
\end{lem}

\subsection{\label{sub:HdomH}Domain  Analysis}

There is a natural isometric isomorphism $l^{2}(\mathbb{N}_{0})\simeq\mathscr{H}_{2}$
where $\mathscr{H}_{2}$ is the Hardy space of all analytic functions
$f$ on $\mathbb{D}:=\{z\in\mathbb{C}\:\big|\:\left|z\right|<1\}$,
with coefficients $(x_{k})\in l^{2}(\mathbb{N}_{0})$, 
\begin{equation}
f(z)=\sum_{k=\mathbb{N}_{0}}x_{k}z^{k},\label{eq:f(z)}
\end{equation}
and where
\begin{equation}
\left\Vert f\right\Vert _{\mathscr{H}_{2}}^{2}=\sup_{r<1}\left\{ \int_{0}^{1}\left|f(r\, e(t))\right|^{2}dt\right\} ;\label{eq:H2norm}
\end{equation}
see \cite{Rud87}. (In (\ref{eq:H2norm}), we used $e(t):=e^{i2\pi t}$,
$t\in\mathbb{R}$.)
\begin{lem}
\label{lem:H2}Under the unitary isomorphism $l^{2}(\mathbb{N}_{0})\simeq\mathscr{H}_{2}:x=(x_{k})\mapsto f(z)=\sum_{k\in\mathbb{N}_{0}}x_{k}z^{k}$,
the selfadjoint operator $H:x\mapsto(k\, x_{k})$ becomes $f\mapsto z\frac{d}{dz}$,
and
\begin{equation}
\mathscr{D}(H)\rightarrow\mathscr{A}(\mathscr{H}_{2})\label{eq:H2}
\end{equation}
where $\mathscr{A}(\mathscr{H}_{2})$ is the Banach algebra of functions
$f\in\mathscr{H}_{2}$ with continuous extension $\tilde{f}$ to $\overline{\mathbb{D}}$. 

The unitary one-parameter group $\{U(t)\}_{t\in\mathbb{R}}$ generated
by $H=z\frac{d}{dz}$ in $\mathscr{H}_{2}$ is 
\begin{equation}
\left(U(t)f\right)(z)=f\left(e(t)z\right),\label{eq:UH2sg}
\end{equation}
for all $f\in\mathscr{H}_{2}$, $t\in\mathbb{R}$, and all $z\in\mathbb{D}$
(= the disk) where $e(t)=e^{i2\pi t}$.\end{lem}
\begin{proof}
Since $\mathscr{D}(H)\subset l^{1}(\mathbb{N})$, the power series
${\displaystyle \tilde{f}(z)=\sum_{k\in\mathbb{N}}x_{k}z^{k}}$ is
absolutely convergent for all $z\in\overline{\mathbb{D}}$ when $x\in\mathscr{D}(H)\subset l^{1}(\mathbb{N}_{0})$,
but $f\mapsto\tilde{f}$ does not map onto $\mathscr{A}(\mathscr{H}_{2})$. \end{proof}
\begin{cor}
The operators $\{U(t)\}_{t\in\mathbb{R}}$ in (\ref{eq:UH2sg}) extend
to a contraction semigroup $\{U(t)\,;\, t\in\mathbb{C},\: t=s+i\sigma,\:\sigma>0\}$,
i.e., analytic continuation in $t$ to the upper half-place $\mathbb{C}_{+}$,
and 
\[
\left\Vert U(s+i\sigma)f\right\Vert _{\mathscr{H}_{2}}\leq\left\Vert f\right\Vert _{\mathscr{H}_{2}},\quad f\in\mathscr{H}_{2},
\]
holds for all $s+i\sigma\in\mathbb{C}_{+}$.\end{cor}
\begin{proof}
Follows from Lemma \ref{lem:isoH2}, and a substitution of $e(s+i\sigma)z$
into eq. (\ref{eq:f(z)}) and (\ref{eq:UH2sg}).\end{proof}
\begin{lem}
\label{lem:Hdef}Under the isomorphism ${\displaystyle x\rightarrow f(z)=\sum_{k\in\mathbb{N}_{0}}x_{k}z^{k}}$,
the defect vector $y=\left(\frac{1}{1+k}\right)_{k\in\mathbb{N}_{0}}$
is mapped into
\begin{equation}
y(z)=-\frac{1}{z}\log\left(1-z\right),\; z\in\mathbb{D}\backslash\{0\}.\label{eq:defect}
\end{equation}
Under the isomorphism ${\displaystyle x\rightarrow f_{x}(z)=\sum_{k\in\mathbb{N}_{0}}x_{k}z^{k}}$,
the domain $\mathscr{D}(L)$ of the restriction $L$ is 
\begin{equation}
\left\{ f\in\mathscr{H}_{2}\:\big|\:\left\langle y,\left(I+z\frac{d}{dz}\right)f\right\rangle _{\mathscr{H}_{2}}=0\right\} =\left\{ \tilde{f}\in\mathscr{D}\widetilde{(H)}\:\big|\:\tilde{f}(1)=0\right\} .\label{eq:DH}
\end{equation}
 \end{lem}
\begin{proof}
Immediate.\end{proof}
\begin{thm}
\label{rem:Lerch}Let $\Phi$ be the Lerch\textquoteright{}s transcendent
\cite{EMOT81,AS92}, 
\begin{equation}
\Phi(z,s,v)=\sum_{k\in\mathbb{N}_{0}}\frac{z^{k}}{\left(k+v\right)^{s}}.\label{eq:lerch}
\end{equation}
(Recall that $\Phi$ converges absolutely for all $v\neq0,-1,-2,\ldots$,
with either $z\in\mathbb{D}$, or $z\in\partial\mathbb{D}=\mathbb{T}$
and $\Re(s)>1$. )

Under the isomorphism $l^{2}(\mathbb{N}_{0})\simeq\mathscr{H}_{2}$,
the defect vectors ${\displaystyle \left(\frac{1}{k\mp i}\right)_{k\in\mathbb{N}_{0}}}$
in (\ref{eq:defv}) are mapped to
\begin{equation}
\tilde{x}_{\pm}(z)=\Phi(z,1,\mp i);\label{eq:deflerch}
\end{equation}
and eq. (\ref{eq:defect}) can be written as
\begin{equation}
y(z)=\Phi(z,1,1).\label{eq:deflerch2}
\end{equation}

Moreover, the eigenvectors $\left(\frac{\lambda_{n}(t)}{k-\lambda_{n}(t)}\right)_{k\in\mathbb{N}_{0}}$in
(\ref{eq:ev}) are mapped to
\[
\tilde{y}_{n,t}(z)=\lambda_{n}(t)\,\Phi(z,1,-\lambda_{n}(t)).
\]
Hence, for all $t\in\mathbb{R}$, the set $\{\tilde{y}_{n,t}\}_{n\in\mathbb{N}_{0}}$
forms an orthogonal basis for $\mathscr{H}_{2}$.\end{thm}
\begin{proof}
Follows from Lemma \ref{lem:isoH2}, Theorems \ref{thm:eigEqn} and
\ref{thm:Spectrum-L_t}.
\end{proof}

\subsection{The Szegö Kernel }

In understanding the defect spaces $\mathscr{D}_{\pm}$, and the partial
isometries between them, we will be making use of the Szegö kernel.
It is the reproducing kernel for the Hardy space $\mathscr{H}_{2}$,
i.e., a function $K$ on $\mathbb{D}\times\mathbb{D}$ such $K(\cdot,z)\in\mathscr{H}_{2}$,
$z\in\mathbb{D}$; and 
\begin{equation}
f(z)=\left\langle K(\cdot,z),f\right\rangle _{\mathscr{H}_{2}},\quad f\in\mathscr{H}_{2}.\label{eq:SZK}
\end{equation}
(Note our inner product is linear in the second variable.)

The following formula is known
\begin{equation}
K(w,z)=\frac{1}{1-\overline{z}w}.\label{eq:ZKK}
\end{equation}
Our main concern is properties of functions $f\in\mathscr{H}_{2}$
on the boundary $\partial\mathbb{D}$.
\begin{lem}
\label{lem:Hardy}The following offers a correspondence between the
two representations $l^{2}(\mathbb{N}_{0})$ and $\mathscr{H}_{2}$,
see (\ref{eq:f(z)})-(\ref{eq:H2norm}). Consider the two operators
$L$ and $H$ from sections \ref{sec:sbdd}-\ref{sec:spH}. In the
$\mathscr{H}_{2}$ model, we have $H=z\frac{d}{dz}$ , $L=H\big|_{\mathscr{D}(L)}$,
and 
\[
\mathscr{D}(L)=\{f\in\mathscr{H}_{2}\:\big|\:\tilde{f}(1)=0\}
\]
where $\tilde{f}$ is the boundary function (see \cite[ch11]{Rud87}),
i.e., for $e(x)\in\mathbb{T}=\partial\mathbb{D}$, a.a. $x\in\mathbb{R}/\mathbb{Z}\simeq[0,1)$,
setting $e(x):=e^{i2\pi x}$, Fatou's theorem states existence a.e.
of the function $\tilde{f}$ as follows:
\begin{equation}
\tilde{f}(x)=\lim_{z\rightarrow e(x)}f(z).\quad(\mbox{non-tangentially})\label{eq:bdf}
\end{equation}
We make the identification $\tilde{f}(e(x))\simeq\tilde{f}(x)$ with
$\tilde{f}$ $\mathbb{Z}$-periodic. Under the correspondence: $f\longleftrightarrow\tilde{f}$,
we have
\begin{equation}
f(z)=\int_{0}^{1}\frac{\tilde{f}(x)}{1-\overline{e(x)}z}dx\label{eq:fext}
\end{equation}
where the RHS in (\ref{eq:fext}) is a Szegö -kernel integral. Moreover,
\begin{equation}
z\frac{d}{dz}f\longleftrightarrow\frac{1}{2\pi i}\frac{d}{dx}\tilde{f}.\label{eq:ddz}
\end{equation}

As a result, for the vectors in the two defect-spaces $\mathscr{D}_{\pm}$
in Lemma \ref{lem:vN def-space}, we have
\begin{eqnarray*}
f_{+}(z) & = & \int_{0}^{1}\frac{e^{-2\pi x}}{1-\overline{e(x)}z}dx\\
 & = & \frac{1-e^{-2\pi}}{2\pi i}\sum_{n=0}^{\infty}\frac{z^{n}}{n-i},\;\mbox{ and}\\
f_{-}(z) & = & \int_{0}^{1}\frac{e^{2\pi x}}{1-\overline{e(x)}z}dx\\
 & = & -\frac{e^{2\pi}-1}{2\pi i}\sum_{n=0}^{\infty}\frac{z^{n}}{n+i}.
\end{eqnarray*}
\end{lem}
\begin{proof}
The key ingredient in the proof is the representation of $f(z)\in\mathscr{H}_{2}$
as kernel integrals arising from the corresponding boundary versions
$\tilde{f}(x)\simeq\tilde{f}(e(x))$ where $e(x)=e^{i2\pi x}$, and
$\tilde{f}$ is as in (\ref{eq:bdf}). The kernel integral in (\ref{eq:fext})
is a reproducing property for the Szegö-kernel, see \cite{Rud87}.

To show that the transform $f(z)\longleftrightarrow\tilde{f}(x)$
(Fatou a.e. extension to $\partial\mathbb{D}$) is a norm-preserving
isomorphism of $\mathscr{H}_{2}$ onto a closed subspace in $L^{2}(\mbox{of a periodic interval})$,
we check that
\begin{equation}
f(z)=\sum_{n=0}^{\infty}c_{n}z^{n}\label{eq:series}
\end{equation}
where the coefficients $(c_{n})$ in (\ref{eq:se}) are also the Fourier-coefficients
of the function $\tilde{f}$ in (\ref{eq:bdf}). But for $z\in\mathbb{D}$,
i.e., $\left|z\right|<1$, we may expand the RHS in (\ref{eq:fext})
as follows
\[
f(z)=\sum_{n=0}^{\infty}z^{n}\int_{0}^{1}e(-nx)\tilde{f}(x)dx.
\]
But 
\begin{equation}
c_{n}=\int_{0}^{1}\overline{e(nx)}\tilde{f}(x)dx\label{eq:cn}
\end{equation}
are the $\tilde{f}$-Fourier coefficients over the period interval
$[0,1)$. The result follows from this as follows
\begin{eqnarray*}
\left\Vert f\right\Vert _{\mathscr{H}_{2}}^{2} & = & \sum_{n=0}^{\infty}\left|c_{n}\right|^{2}\,\,\,\,\,\,\,\,\,\,\,\,\,\,\,\,\,\,\,(\mbox{by Lemma \ref{lem:H2}})\\
 & = & \int_{0}^{1}\left|\tilde{f}(x)\right|^{2}dx\,\,\,\,\,\,\,(\mbox{by (}\ref{eq:cn}\mbox{) and Parseval applied to }\tilde{f})\\
 & = & \left\Vert \tilde{f}\right\Vert _{L^{2}(0,1)}^{2}.
\end{eqnarray*}

\end{proof}
Let ${\displaystyle H=z\frac{d}{dz}}$ be the selfadjoint operator
in Lemma \ref{lem:isoH2}.
\begin{cor}
Under the boundary correspondence (\ref{eq:bdf})-(\ref{eq:fext})
in Lemma \ref{lem:Hardy}, if $f\in\mathscr{D}(H)\subset\mathscr{H}_{2}$;
then the boundary function $\tilde{f}$ is a well-defined and continuous
function on $\mathbb{T}\simeq[0,1)\simeq\mathbb{R}/\mathbb{Z}$. If
further $f\in\mathscr{D}(H^{2})$, then $\tilde{f}$\textup{ ($=$
the boundary function)} is Lipschitz, i.e., 
\begin{equation}
\left|\tilde{f}(x)-\tilde{f}(y)\right|\leq\mbox{Const}\cdot\left|x-y\right|.\label{eq:Lip}
\end{equation}
\end{cor}
\begin{proof}
By Lemma \ref{lem:Hardy} and Conclusions \ref{conclusion}, the operator
\begin{equation}
\mathscr{H}_{2}\ni f\mapsto\tilde{f}\in\mathscr{L}_{+}\subset L^{2}(\mathbb{T})\label{eq:T}
\end{equation}
is an isometric isomorphism of $\mathscr{H}_{2}$ onto $\mathscr{L}_{+}$.
Moreover, see the proof of Lemma \ref{lem:Hardy}, if $f\in\mathscr{D}(H)$,
then
\begin{equation}
\tilde{f}(e(x))=\sum_{n\in\mathbb{N}_{0}}c_{n}e(nx),\quad x\in\mathbb{R}/\mathbb{Z}\label{eq:se}
\end{equation}
with $(c_{n})\in l^{2}$ and $(nc_{n})\in l^{2}$, see eq. (\ref{eq:series})-(\ref{eq:cn}).
Hence $(c_{n})$ in (\ref{eq:se}) is in $l^{1}$, by Cauchy-Schwarz,
see Lemma \ref{lem:Hdom}. Therefore, by domination, the boundary
function $\tilde{f}$ in (\ref{eq:se}) is uniformly continuous on
$\mathbb{T}$.

Assume now that $f\in\mathscr{D}(H^{2})$; then $(n^{2}c_{n})\in l^{2}$,
and therefore:
\begin{eqnarray*}
\left|\tilde{f}(x)-\tilde{f}(y)\right| & \leq & \left|x-y\right|2\pi\sum_{n=1}^{\infty}\left|n\, c_{n}\right|\\
 & \leq & \left|x-y\right|2\pi\left(\frac{\pi^{2}}{6}\right)^{1/2}\left(\sum_{n=1}^{\infty}n^{4}\left|c_{n}\right|^{2}\right)^{1/2}\\
 & \leq & \left|x-y\right|\frac{2\pi^{2}}{\sqrt{6}}\left\Vert H^{2}f\right\Vert _{\mathscr{H}_{2}},
\end{eqnarray*}
which is the desired Lipschitz estimate (\ref{eq:Lip}).\end{proof}
\begin{cor}
Consider the selfadjoint operator $H=z\,\frac{d}{dz}$ in $\mathscr{H}_{2}$
(as in Lemma \ref{lem:H2}). For $f\in\mathscr{D}(H^{2})$, consider
the boundary function $\tilde{f}$ as in Lemma \ref{lem:Hardy}, and
set
\begin{equation}
\mathscr{D}_{\{\pm1\}}:=\left\{ f\in\mathscr{D}(H^{2})\:\big|\:\tilde{f}(1)=\tilde{f}(-1)=0\right\} \label{eq:D11}
\end{equation}
and set
\begin{equation}
L_{2}:=H^{2}\bigg|_{\mathscr{D}_{\{\pm1\}}},\label{eq:L2subH2}
\end{equation}
then $L_{2}$ is a densely defined restriction of $H^{2}$, and
\begin{equation}
\mbox{GLB}(Q_{L_{2}})=1.\label{eq:glbL2}
\end{equation}
\end{cor}
\begin{proof}
An inspection shows that this example is unitarily equivalent to the
one considered in Lemmas \ref{lem:Lapbdd} and \ref{lem:Lapres}.
Indeed, if $f\in\mathscr{D}(H^{2})$ has the representation $f(z)=\sum_{n\in\mathbb{N}_{0}}a_{n}z^{n}$,
then the two conditions in (\ref{eq:D11}) translate into
\[
\sum_{n\in\mathbb{N}_{0}}a_{n}=\sum_{n\in\mathbb{N}_{0}}(-1)^{n}a_{n}=0;
\]
or equivalently, 
\[
\sum_{n\in\mathbb{N}_{0}}a_{2n}=\sum_{n\in\mathbb{N}_{0}}a_{1+2n}=0;
\]
compare with (\ref{eq:cbd}). Hence the result follows from Lemma
\ref{lem:Lapres}.\end{proof}
\begin{rem}
In using the unit-interval $I=[0,1]$ as range of the independent
variable $x$ in the representation $\mathscr{H}_{2}\ni f(z)\longleftrightarrow\tilde{f}(x)$
via (\ref{eq:bdf}) we use the parameterization
\begin{equation}
I\ni x\mapsto e(x)=e^{i2\pi x}\in\mathbb{T}=\partial\mathbb{D}.\label{eq:e}
\end{equation}
To justify the correspondence $\tilde{f}(x)\longleftrightarrow\tilde{f}(e(x))$
for a function $\tilde{f}$ on $\mathbb{R}$, as in $\tilde{f}=e^{-2\pi x}$,
or $e^{2\pi x}$, it is understood that we use $1$-periodic versions
of these functions, see Figure \ref{fig:exp1} and \ref{fig:exp2}
below

\begin{figure}[H]
\includegraphics[scale=0.8]{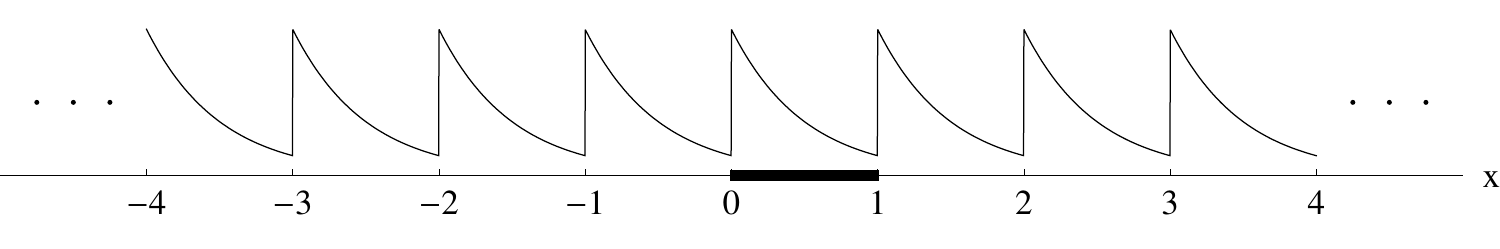}

\caption{\label{fig:exp1}Periodic version of $e^{-2\pi x}$}
\end{figure}

\begin{figure}[H]
\includegraphics[scale=0.8]{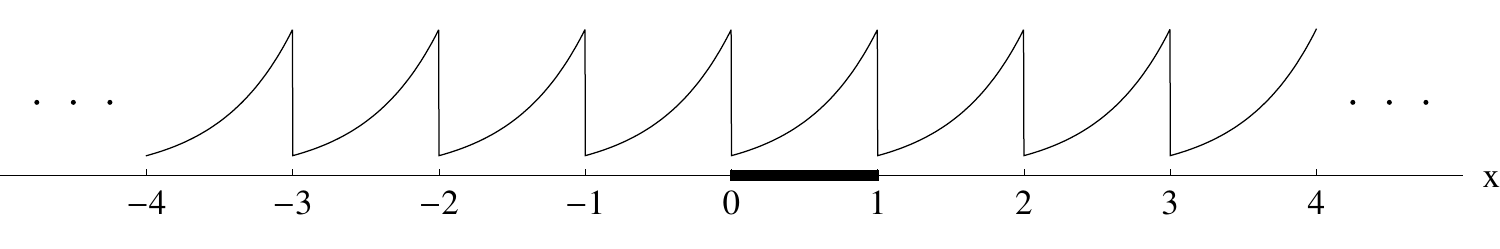}

\caption{\label{fig:exp2}Periodic version of $e^{2\pi x}$}
\end{figure}

\uline{WARNING}\textbf{ }The role of the periodization (illustrated
in Figures \ref{fig:exp1}-\ref{fig:exp2}, and in (\ref{eq:e}))
is important. Indeed, if the period-interval used in Figs 1-2 is changed
from $[0,1)$ into $[-\frac{1}{2},\frac{1}{2})$, we get the following
related function $\tilde{y}_{2}(x)=e^{-2\pi\left|x\right|}$; see
also Figure \ref{fig:exp3} below:

\begin{figure}[H]
\includegraphics[scale=0.8]{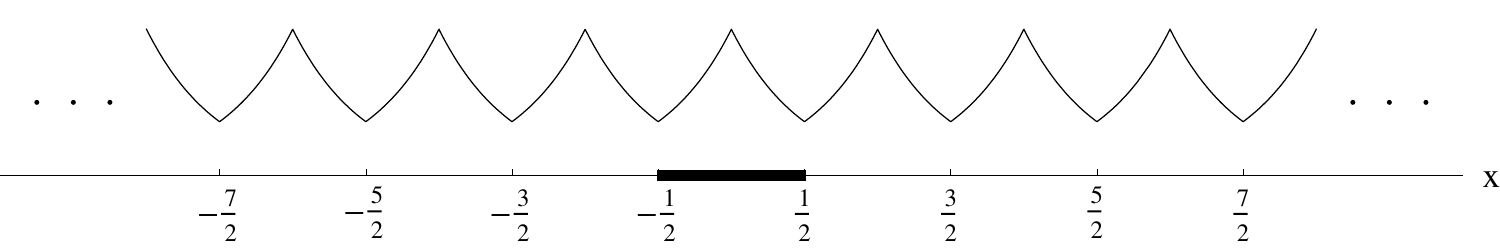}

\caption{\label{fig:exp3}Periodic version of $e^{-2\pi\left|x\right|}$, periodic-interval
$[-\frac{1}{2},\frac{1}{2})$}
\end{figure}

and
\begin{eqnarray}
y_{2}(z) & = & \int_{-\frac{1}{2}}^{\frac{1}{2}}\frac{e^{-2\pi\left|x\right|}}{1-\overline{e(x)}z}dx\nonumber \\
 & = & \frac{1}{\left(2\pi\right)^{2}}\sum_{n=0}^{\infty}\frac{z^{n}}{1+n^{2}};\label{eq:y2}
\end{eqnarray}
and we recovers the $l^{2}(\mathbb{N}_{0})$-sequence ${\displaystyle (y_{2})_{n}:=\frac{1}{1+n^{2}}}$
used in sections 3-4 above.
\end{rem}

\begin{rem}
In the isometric realization from Lemma \ref{lem:Hardy} of the Hardy
space $\mathscr{H}_{2}$ as a closed subspace inside $L^{2}(I)$ we
are selecting a specific a period interval $I$. Avoiding a choice,
an alternative to $L^{2}(I)$ is the Hilbert space $L^{2}(\mathbb{R}/\mathbb{Z})$
where the quotient group $\mathbb{R}/\mathbb{Z}$ is given its invariant
quotient measure on $\mathbb{R}/\mathbb{Z}$. With the identification
of $\mathbb{R}/\mathbb{Z}$ with $\mathbb{T}=\partial\mathbb{D}$,
this is Haar measure on the one-torus $\mathbb{T}$. Of these three
equivalent versions of Hilbert space, perhaps $L^{2}(\mathbb{R}/\mathbb{Z})$
is more natural as it doesn\textquoteright{}t presuppose a choice
of period interval.

The closed subspace $\mathscr{L}_{+}$ in $L^{2}(\mathbb{R}/\mathbb{Z})$
corresponding to $\mathscr{H}_{2}$ under (\ref{eq:bdf}) in Lemma
\ref{lem:Hardy} is of course the subspace of $L^{2}$ functions that
have their Fourier coefficients vanish on the negative part of $\mathbb{Z}$.
This subspace $\mathscr{L}_{+}$ in $L^{2}(\mathbb{R}/\mathbb{Z})$
is invariant under periodic translation $(U(t)f)(x)=f(x+t)$. Indeed
this unitary one-parameter group $U(t)$, acting in $\mathscr{L}_{+}$,
has as its infinitesimal generator our standard selfadjoint operator
$H$ from Lemmas \ref{lem:Hdom} and \ref{lem:Hardy}. Note that the
spectrum of $H$ is $\mathbb{N}_{0}$ when $H$ is realized as a selfadjoint
operator in $\mathscr{L}_{+}$; not in $L^{2}(\mathbb{R}/\mathbb{Z})$.
One can adapt this approach in order to get realizations of the other
selfadjoint extensions of $L$, and their associated unitary one-parameter
groups.
\end{rem}

\begin{rem}
\label{conclusion}\textbf{Conclusions} (Toeplitz operators). In our
study of selfadjoint extensions, we are making use of three (unitarily
equivalent) realizations of Toeplitz operators; taking here \textquotedblleft{}Toeplitz
operator\textquotedblright{} to mean \textquotedblleft{}matrix corner\textquotedblright{}
of an operator in an ambient $L^{2}$ space, i.e., restriction of
an operator $T$ in ambient $L^{2}$, followed by the projection $P$
onto the subspace; in short, $PTP$.

But it is helpful to realize these \textquotedblleft{}matrix corners\textquotedblright{}
in any one of three equivalent ways; hence three equivalent ways of
realizing operators $T$ in the ambient Hilbert space and $PTP$ in
its closed subspace:
\begin{enumerate}
\item realize the subspace as $\mathscr{L}_{+}$ inside $L^{2}(\mathbb{R}/\mathbb{Z})$,
where $\mathscr{L}_{+}$ is the subspace of $L^{2}$-functions with
vanishing negative Fourier coefficients; 
\item or we may take the subspace to be the Hardy space $\mathscr{H}_{2}$
inside $L^{2}(\mathbb{T})$. Or equivalently, 
\item we can work with the subspace of one-sided $l^{2}$ sequences inside
two-sided; so the subspace $l^{2}(\mathbb{N}_{0})$ inside $l^{2}(\mathbb{Z})$.
\end{enumerate}

So there are these three different but unitarily equivalent formulations;
each one brings to light useful properties of the operators under
consideration.

One detail which makes the analysis more difficult here as compared
with the more classical case of $L^{2}$ is the study of unitary one-parameter
groups: for example, periodic translation, $U(t):f\mapsto f(x+t)$,
leaves invariant the subspace; but the related quasi-periodic translation
does not.

\end{rem}

\begin{table}
\begin{tabular}{|>{\centering}m{0.15\textwidth}|>{\centering}m{0.2\textwidth}|>{\centering}m{0.15\textwidth}|>{\centering}m{0.15\textwidth}|>{\centering}m{0.2\textwidth}|}
\hline 
ambient Hilbert space & closed subspace & $\mathscr{D}(H)$ & $\mathscr{D}(L)$ & $\mathscr{D}(L^{*})$\tabularnewline
\hline 
case 1

\medskip{}

$l^{2}(\mathbb{Z})$ & $l^{2}(\mathbb{N}_{0})$ & $x\in l^{2}(\mathbb{N}_{0})$ s.t. $(k\, x_{k})\in l^{2}$ & \medskip{}
$x\in\mathscr{D}(H)$ s.t. $\sum_{k}x_{k}=0$\medskip{}
 & $\mathscr{D}(H)+\mathbb{C}y_{3}$\tabularnewline
\hline 
case 2

\medskip{}
 $L^{2}(\mathbb{T})$ & $\mathscr{H}_{2}(\mathbb{D})$

(Hardy space) & \medskip{}
$f\in\mathscr{H}_{2}$\\
s.t.\\
 $z\frac{d}{dz}f\in\mathscr{H}_{2}$\medskip{}
 & $f\in\mathscr{D}(H)$ s.t.\\
$\tilde{f}(1)=0$ & $\mathscr{D}(H)+\mathbb{C}\tilde{y}_{3}$\tabularnewline
\hline 
\medskip{}
case 3

\medskip{}
 $L^{2}(\mathbb{R}/\mathbb{Z})$\\
$\simeq$\\
$L^{2}(I)$\medskip{}
$I=[0,1)$ & $\mathscr{L}_{+}=\{f\in L^{2}\:\big|\:\hat{f}(n)=0,\: n\in\mathbb{Z}_{-}\}$ & $f\in\mathscr{L}_{+}$\\
s.t. $\frac{d}{dx}f\in L^{2}$ & $f\in\mathscr{D}(H)$ s.t. $f(0)=f(1)=0$  & \medskip{}
$\mathscr{D}(H)+\mathbb{C}P_{\mathscr{L}_{+}}\psi$,

\bigskip{}
$\psi(x)=$\\
$\begin{cases}
1 & 0\leq x<\frac{1}{2}\\
-1 & \frac{1}{2}\leq x<1
\end{cases}$\medskip{}
\tabularnewline
\hline 
\multicolumn{1}{>{\centering}m{0.15\textwidth}}{} & \multicolumn{1}{>{\centering}m{0.2\textwidth}}{} & \multicolumn{1}{>{\centering}m{0.15\textwidth}}{} & \multicolumn{1}{>{\centering}m{0.15\textwidth}}{} & \multicolumn{1}{>{\centering}m{0.2\textwidth}}{}\tabularnewline
\end{tabular}

\caption{\label{tab:H}Three models of Toeplitz operator analysis.}
\end{table}

\begin{rem}
In our model, we have $L\subset H\subset L^{*}$ (see (\ref{eq:LHL})).
As a result of (\ref{eq:Y2}), (\ref{eq:Y3}) and Remark \ref{rem:H},
we see that 
\begin{equation}
\dim\mathscr{D}(H)/\mathscr{D}(L)=\dim\mathscr{D}(L^{*})/\mathscr{D}(H)=1\label{eq:codim}
\end{equation}
(the two quotients are both one-dimensional); hence 
\begin{equation}
\mathscr{D}(L^{*})=\mathscr{D}(H)+\mathbb{C}P_{\mathscr{L}_{+}}\psi.\label{eq:de}
\end{equation}

The choice of the generating vector $\psi$ is not unique. For example,
let $T$ be the triangular wave, i.e.,
\begin{equation}
T(x)=\frac{1}{2}-\left|x\right|,\quad x\in[-\frac{1}{2},\frac{1}{2}];\label{eq:trif}
\end{equation}
which has Fourier series 
\begin{equation}
T(x)=\frac{1}{4}+\frac{2}{\pi^{2}}\sum_{n=1}^{\infty}\frac{1}{(2n-1)^{2}}\cos(2\pi(2n-1)x),\quad x\in[-\frac{1}{2},\frac{1}{2}].\label{eq:tri}
\end{equation}
Differentiating (\ref{eq:trif}), we get the Haar wavelet $\psi$
as in Table \ref{tab:H}; and where 
\begin{equation}
\psi(x)=-\frac{4}{\pi}\sum_{n=1}^{\infty}\frac{1}{2n-1}\sin(2\pi(2n-1)x),\quad x\in(-\frac{1}{2},\frac{1}{2}).\label{eq:haar}
\end{equation}
It follows that 
\begin{equation}
\left(P_{\mathscr{L}_{+}}\psi\right)(x)=\frac{2i}{\pi}\sum_{n=1}^{\infty}\frac{1}{2n-1}\, e_{2n-1}(x).
\end{equation}
Note the Fourier coefficients satisfy 
\[
\left(\left(P_{\mathscr{L}_{+}}\psi\right)^{\wedge}(n)\right)\in l^{2}(\mathbb{N}_{0})
\]
but
\[
\left(n\left(P_{\mathscr{L}_{+}}\psi\right)^{\wedge}(n)\right)\notin l^{2}(\mathbb{N}_{0}).
\]

\end{rem}

\subsection{$\mathscr{D}_{\pm}(L_{F})$ as RKHSs}

For a fixed closed subset $F$ of $\partial\mathbb{D}$ of zero angular
measure, we identified in Corollary \ref{cor:F} an associated Hermitian
operator $L_{F}$ with dense domain in the Hardy space $\mathscr{H}_{2}$,
see (\ref{eq:Df}) and (\ref{eq:LF}). Below we compute the corresponding
pair of deficiency subspaces (see Lemma \ref{lem:vNext}). While they
are closed subspaces in $\mathscr{H}_{2}$, it turns out that they
can be computed with reference to only the given closed set $F$.
To this end we show in Theorem \ref{thm:defLF} that each deficiency
subspace is a reproducing kernel Hilbert space (RKHS) with a positive
definite kernel function on $F\times F$. The kernel is not the Szegö
kernel, but rather the Hurwitz zeta function computed on differences
of points in $F$ (angular variables), see (\ref{eq:KernelH}) below.
From this we show that the partial isometries between the two deficiency
spaces \textquotedbl{}is\textquotedbl{} a compact group $G(F)$, a
Lie group if $F$ is finite. With this, we prove in Corollary \ref{cor:Liesp}
a formula for the spectrum of each of the selfadjoint extensions of
$L_{F}$. 
\begin{cor}
\label{cor:finite}Let $z_{1},z_{2},\ldots,z_{n}$ be a finite set
of distinct points $\in\mathbb{C}$ s.t. $|z_{i}|=1$, $1\leq i\leq n$;
and set
\begin{equation}
\mathscr{D}_{\{z_{i}\}}=\{f\in\mathscr{D}\widetilde{(H)}\:\big|\:\tilde{f}(z_{i})=0,\,1\leq i\leq n\};\label{eq:Dz}
\end{equation}
then
\begin{equation}
L_{n}:=z\frac{d}{dz}\big|_{\mathscr{D}_{\{z_{i}\}}}\label{eq:Ln}
\end{equation}
is Hermitian with dense domain in $\mathscr{H}_{2}$, and with deficiency
indices $(n,n)$. \end{cor}
\begin{proof}
This is an application of Lemma \ref{lem:semibdd}.\end{proof}
\begin{cor}
\label{cor:F}Let $F\subset\mathbb{T}=\{z\in\mathbb{C}\:\big|\:|z|=1\}=\partial\mathbb{D}$
be a closed subset of zero Haar measure, e.g., some fixed Cantor subset
of $\mathbb{T}$; and set
\begin{equation}
\mathscr{D}_{F}:=\{f\in\mathscr{D}_{\mathscr{H}_{2}}(z\frac{d}{dz})\:\big|\: f=0\:\mbox{ on }F\},\label{eq:Df}
\end{equation}
then 
\begin{equation}
L_{F}:=z\frac{d}{dz}\big|_{\mathscr{D}(F)}\label{eq:LF}
\end{equation}
is Hermitian with dense domain in the Hardy space $\mathscr{H}_{2}$,
and with deficiency indices $(\infty,\infty)$.\end{cor}
\begin{proof}
Follows again from the general result in Lemma \ref{lem:semibdd}.
Since $F$ has Haar measure $0$, one verifies that the closed subspace
$\mathfrak{M}$ in $\mathscr{H}_{2}$ defined for $\mathscr{D}_{F}$
in (\ref{eq:Df}) (via Lemma \ref{lem:semibdd}) is closed in $\mathscr{H}_{2}$
and satisfies 
\begin{equation}
\mathfrak{M}\cap\mathscr{D}(z\frac{d}{dz})=0.\label{eq:int-1}
\end{equation}
\end{proof}
\begin{rem}
Let $\mathscr{A}(\mathscr{H}_{2})$ be the Banach algebra introduced
in Lemma \ref{lem:H2}, and let $F\subset\mathbb{T}$ be a closed
subset specified as in Corollary \ref{cor:F}. Let $L_{F}$ be the
Hermitian restriction operator in (\ref{eq:LF}), and $L_{F}^{*}$
its adjoint operator. Then both domains $\mathscr{D}(L_{F})$ and
$\mathscr{D}(L_{F}^{*})$ are invariant under pointwise multiplication
by functions $a$ from $\mathscr{A}(\mathscr{H}_{2})$, i.e., 
\begin{equation}
\left(af\right)(z)=a(z)f(z),\quad a\in\mathscr{A}(\mathscr{H}_{2}),\: f\in\mathscr{H}_{2},\:\mbox{and }z\in\mathbb{D}.\label{eq:mop}
\end{equation}
Hence this action of $\mathscr{A}(\mathscr{H}_{2})$ passes to the
quotient
\begin{equation}
\mathscr{D}(L_{F}^{*})/\mathscr{D}(L_{F})\simeq\mathscr{D}_{+}(L_{F})+\mathscr{D}_{-}(L_{F}).\label{eq:quotient}
\end{equation}
\end{rem}
\begin{cor}
\label{cor:Fdef}Let $F\subset\mathbb{T}$ be a closed subset of zero
measure, see Corollary \ref{cor:F} and (\ref{eq:Df}). Let $x_{\pm}\in l^{2}(\mathbb{N}_{0})$
be the sequences from (\ref{eq:defv}), i.e., $x_{\pm}(k)=(k\mp i)^{-1}$,
$k\in\mathbb{N}_{0}$. Then the two deficiency subspaces in $\mathscr{H}_{2}$
derived from (\ref{eq:LF})
\begin{equation}
\mathscr{D}_{\pm}(L_{F})=\{f_{\pm}\in\mathscr{D}(L_{F}^{*})\:\big|\: L_{F}^{*}f_{\pm}=\pm if_{\pm}\}\label{eq:defLF}
\end{equation}
are as follows: For $\theta\in F$, set 
\begin{equation}
f_{\pm}^{(\theta)}(z):=\sum_{n=0}^{\infty}\frac{e(-n\,\theta)}{n\mp i}z^{n};\label{eq:defLF-1}
\end{equation}
i.e., the expansion coefficients for $f_{\pm}^{(\theta)}$ are 
\begin{equation}
\widehat{f_{\pm}^{(\theta)}}(n)=\overline{e(n\,\theta)}\, x_{\pm}(n),\quad n\in\mathbb{N}_{0}.\label{eq:LFcoeff}
\end{equation}
Then $\mathscr{D}_{\pm}(L_{F})$ is the closed span in $\mathscr{H}_{2}$
of the functions in (\ref{eq:defLF-1}), as $\theta$ ranges over
$F$. \end{cor}
\begin{proof}
We will do the detailed steps for any $\{f_{+}^{(\theta)}(\cdot)\:\big|\:\theta\in F\}\subset\mathscr{D}_{+}(L_{F})$,
as the other case for $\mathscr{D}_{-}(L_{F})$ is the same argument,
\emph{mutatis mutandis}.

Note that functions $\varphi\in\mathscr{D}(L_{F})$ are given by the
condition
\begin{equation}
\varphi(\theta)\simeq\varphi(e(\theta))=0,\quad\forall\theta\in F\label{eq:varphi}
\end{equation}
where we use the usual identification $\theta\longleftrightarrow e(\theta)=e^{i2\pi\theta}$
with points in a period interval $\simeq$ points in $\mathbb{T}$.
For $\varphi\in\mathscr{D}(L_{F})$, set $x_{n}=\hat{\varphi}(n)$,
i.e., ${\displaystyle \varphi(z)=\sum_{n\in\mathbb{N}_{0}}x_{n}z^{n}}$.
Then for $\theta\in F$, we have; $(n\, x_{n})\in l^{2}(\mathbb{N}_{0})$,
and:
\begin{alignat}{1}
0=\varphi(\theta) & =\sum_{n=0}^{\infty}x_{n}e(n\,\theta)\nonumber \\
 & =\left\langle \frac{\overline{e(n\,\theta)}}{n-i},(n+i)x_{n}\right\rangle _{l^{2}(\mathbb{N}_{0})}\nonumber \\
 & =\left\langle f_{+}^{(\theta)},(L_{F}+iI)\varphi\right\rangle _{\mathscr{H}_{2}}\label{eq:LFi}
\end{alignat}
where, in the last step, we used the isomorphism $l^{2}(\mathbb{N}_{0})\simeq\mathscr{H}_{2}$
of Lemma \ref{lem:H2}. The respective subscripts $\left\langle \cdot,\cdot\right\rangle _{l^{2}(\mathbb{N}_{0})}$
and $\left\langle \cdot,\cdot\right\rangle _{\mathscr{H}_{2}}$ indicates
the reference Hilbert space used. 

Since (\ref{eq:LFi}) holds for all $\theta\in F$, it follows that
each $f_{+}^{(\theta)}\in\mathscr{D}(L_{F})$, so the asserted ``$=$''
in the Corollary follows.
\end{proof}

\begin{rem}
If $\theta\mbox{ and }\rho\in I$ ($=$ the period interval), then
the $\mathscr{H}_{2}$-inner product of the two functions $f_{+}^{(\theta)}$
and $f_{-}^{(\theta)}$ is as follows:
\begin{equation}
\left\langle f_{+}^{(\theta)},f_{+}^{(\rho)}\right\rangle _{\mathscr{H}_{2}}=\sum_{n=0}^{\infty}\frac{e(n(\theta-\rho))}{1+n^{2}}=Z(\theta-\rho,1,2)\label{eq:dInner}
\end{equation}
where $Z$ is the Hurwitz-zeta function.

\end{rem}
\begin{thm}
\label{thm:defLF}Let the closed subset $F\subset\mathbb{T}$ be as
in Corollary \ref{cor:Fdef} and let $L_{F}$ be the corresponding
Hermitian unbounded and densely defined operator in the Hardy space
$\mathscr{H}_{2}$. Then each of the two defect spaces $\mathscr{D}_{\pm}(L_{F})$
in (\ref{eq:defLF}) is a reproducing kernel Hilbert space (RKHS)
with RK equal to the Hurwitz zeta function (\ref{eq:dInner}). \end{thm}
\begin{proof}
For the theory of RKHS, see for example \cite{Nel57,Alp92,ABK02}.
In summary, given a set $F$ and a positive definite kernel $\{K(\alpha,\beta)\}_{(\alpha,\beta)\in F\times F}$
then the RKHS, $\mathscr{H}(K)$ is the completion of finitely supported
functions $\varphi$ on $F$, i.e., 
\begin{equation}
\varphi:F\rightarrow\mathbb{C}\label{eq:ff}
\end{equation}
in the pre-Hilbert inner product:
\begin{equation}
\left\langle \varphi,\psi\right\rangle _{\mathscr{H}(K)}:=\sum_{\alpha}\sum_{\beta}\,\overline{\varphi(\alpha)}\psi(\beta)K(\alpha,\beta).\label{eq:KH}
\end{equation}
The positive definite property in (\ref{eq:KH}) is the assertion
that
\begin{equation}
\left\langle \varphi,\varphi\right\rangle _{\mathscr{H}(K)}=\underset{F\times F}{\sum\sum}\,\,\overline{\varphi(\alpha)}\varphi(\beta)K(\alpha,\beta)\geq0\label{eq:pdf}
\end{equation}
for all finitely supported functions $\varphi$, see (\ref{eq:ff}).

To establish the theorem, take 
\begin{equation}
K(\alpha,\beta)=\sum_{n=0}^{\infty}\frac{e(n(\alpha-\beta))}{1+n^{2}}=Z(\alpha-\beta,1,2)\label{eq:KernelH}
\end{equation}
where the expression in (\ref{eq:KernelH}) is the Hurwitz zeta function. 

We will denote the Hurwitz zeta-function simply 
\begin{equation}
Z(x):=\sum_{n=0}^{\infty}\frac{e(nx)}{1+n^{2}}.\label{eq:HZf}
\end{equation}

Continue the proof now for $\mathscr{D}_{+}(L_{F})$ (the other case
is by the same argument), recall from Corollary \ref{cor:Fdef} that
if $\varphi$ is a finitely supported function on $F$ (see (\ref{eq:ff})
and (\ref{eq:defLF-1})) then
\begin{equation}
\left(T\varphi\right)(z)=\sum_{\alpha\in F}\varphi(\alpha)f_{+}^{(\alpha)}(z).\label{eq:Tvarphi}
\end{equation}
By (\ref{eq:Tvarphi}) and (\ref{eq:KH}), we conclude that
\[
\left\Vert T\varphi\right\Vert _{\mathscr{H}_{2}}^{2}=\left\Vert \varphi\right\Vert _{\mathscr{H}(K_{Hurwitz})}^{2}.
\]
But this means that $T$ in (\ref{eq:Tvarphi}) extends by closure
and completion to become an isometric isomorphism of the RKHS $\mathscr{H}(K)$
onto $\mathscr{D}_{+}(L_{F})\subset\mathscr{H}_{2}$.\end{proof}
\begin{lem}
\label{lem:real}Let $Z$ be as in (\ref{eq:dInner}), and write 
\begin{eqnarray}
Z(x) & = & \sum_{n=0}^{\infty}\frac{e(nx)}{1+n^{2}}\nonumber \\
 & = & \sum_{n=0}^{\infty}\frac{\cos(2\pi nx)}{1+n^{2}}+i\sum_{n=1}^{\infty}\frac{\sin(2\pi nx)}{1+n^{2}}\nonumber \\
 & = & f(x)+ig(x)\label{eq:HL}
\end{eqnarray}
where $f:=\Re(Z)$, and $g:=\Im(Z)$. Then 
\begin{equation}
f(x)=\frac{\pi}{2}\sum_{n\in\mathbb{Z}}e^{-2\pi\left|x+n\right|}+\frac{1}{2}.\label{eq:re0}
\end{equation}
Moreover,
\begin{equation}
\left(f\big|_{[0,1]}\right)(x)=\frac{\pi}{2}\frac{\cosh(2\pi(x-\frac{1}{2}))}{\sinh(\pi)}+\frac{1}{2}.\label{eq:re}
\end{equation}
Therefore, $f$ is the $1$-periodic extension of the RHS in (\ref{eq:re}).\end{lem}
\begin{proof}
From Remark \ref{rem:K}, we see that
\begin{equation}
f(x)=\sum_{n=0}^{\infty}\frac{\cos(2\pi nx)}{1+n^{2}}=\frac{\pi}{2}\frac{\cosh(2\pi(x-\frac{1}{2}))}{\sinh(\pi)}+\frac{1}{2},\quad x\in[0,1].\label{eq:srep}
\end{equation}

It is well-known that for causal sequences in $l^{2}(\mathbb{N}_{0})\simeq\mathscr{H}_{2}$,
the real and imaginary parts of the corresponding Fourier transform
are related via the Hilbert transform. Thus, we have
\begin{equation}
g(\theta)=\mbox{p.v.}\int_{0}^{1}f(t)\cot\left(\pi(\theta-t)\right)dt
\end{equation}
where $\cot(\pi(\theta-x))$ is the Hilbert-kernel. See Figure \ref{fig:L}
below.

\begin{figure}[H]
\includegraphics{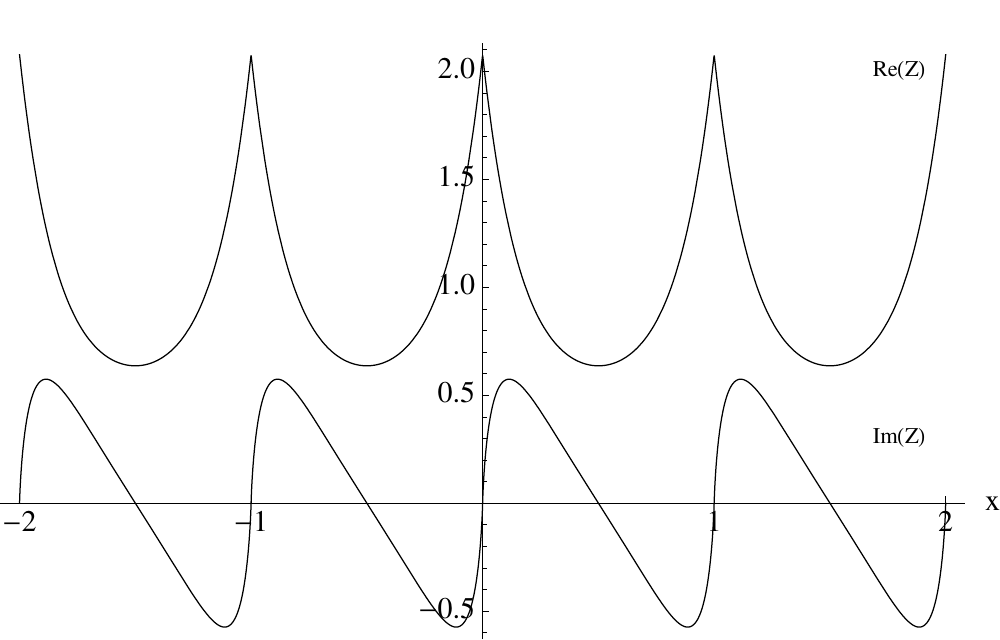}

\caption{\label{fig:L}The real and imaginary parts of the Hurwitz zeta-function
$Z(x)$. Note that $Z(x)$ is real-valued at $\mathbb{Z}/2$. }
\end{figure}

Let $\psi(x):=e^{-2\pi\left|x\right|}$, so that
\[
\hat{\psi}(\lambda)=\int_{-\infty}^{\infty}\psi(x)e^{-i2\pi\lambda x}dx=\frac{1}{\pi}\frac{1}{1+\lambda^{2}}.
\]
For any function $f$ on $\mathbb{R}$, define
\[
\left(\mbox{PER}f\right)(x):=\sum_{n\in\mathbb{Z}}f(x+n).
\]
It follows that (see \cite{BJ02})
\begin{eqnarray}
\left(\mbox{PER}\psi\right)(x) & = & \sum_{n\in\mathbb{Z}}\hat{\psi}(n)e(nx)\nonumber \\
 & = & \frac{1}{\pi}\sum_{n\in\mathbb{Z}}\frac{e(nx)}{1+n^{2}}=\frac{1}{\pi}+\frac{2}{\pi}\sum_{n=1}^{\infty}\frac{\cos(2\pi nx)}{1+n^{2}}.\label{eq:PERpsi}
\end{eqnarray}
Eq. (\ref{eq:re0}) follows from this. 

For $x\in[0,1]$, we also have
\begin{eqnarray}
\left(\mbox{PER}\psi\right)(x) & = & \sum_{n\in\mathbb{Z}}e^{-2\pi\left|x+n\right|}\nonumber \\
 & = & e^{-2\pi x}\sum_{n=0}^{\infty}e^{-2\pi n}+e^{2\pi x}\sum_{n=1}^{\infty}e^{-2\pi n}\nonumber \\
 & = & \frac{\cosh(2\pi(x-\frac{1}{2}))}{\sinh(\pi)}.\label{eq:temp}
\end{eqnarray}
Combine (\ref{eq:re0}) and (\ref{eq:temp}), we get the desired result
in (\ref{eq:re}). \end{proof}
\begin{lem}
The following Fourier integral identities hold: 
\begin{eqnarray}
\int_{0}^{\infty}\frac{\cos(2\pi\lambda x)}{1+\lambda^{2}} & d\lambda= & \pi e^{-2\pi\left|x\right|}\label{eq:rezetaa}\\
\int_{0}^{\infty}\frac{\sin(2\pi\lambda)}{1+\lambda^{2}}d\lambda & = & \frac{1}{2}\left(e^{-2\pi x}li\left(e^{2\pi x}\right)-e^{2\pi x}li\left(e^{-2\pi x}\right)\right)\label{eq:imzetaa}
\end{eqnarray}
where $li(x)$ is the logarithmic integral 
\begin{equation}
li(x)=\int_{0}^{x}\frac{dt}{\ln t},\quad x>0;\label{eq:logint}
\end{equation}
and for $0<x<1$, the RHS in (\ref{eq:logint}) denotes Cauchy principal
value. \end{lem}
\begin{proof}
Eq. (\ref{eq:rezetaa}) can be verified directly. For (\ref{eq:imzetaa}),
see \cite[page 67]{Boc59}.\end{proof}
\begin{cor}
Let $Z$ be the Hurwitz zeta-function in (\ref{eq:dInner}), and let
$g=\Im(Z)$. Set 
\begin{equation}
\varphi(x):=\frac{1}{2}\left(e^{-2\pi x}li\left(e^{2\pi x}\right)-e^{2\pi x}li\left(e^{-2\pi x}\right)\right);\label{eq:imzetaaa}
\end{equation}
then 
\begin{equation}
g(x)=\left(\mbox{\emph{PER }\ensuremath{\varphi}}\right)(x)=\sum_{n\in\mathbb{N}}\varphi(x+n).\label{eq:perimzeta}
\end{equation}
\end{cor}
\begin{proof}
See \cite{BJ02}. Figure \ref{fig:perzetai} below illustrates the
approximation 
\[
\lim_{N\rightarrow\infty}\sum_{n=-N}^{N}\varphi(x+n)=g(x)\left(=\Im Z\right).
\]

\end{proof}
\begin{figure}
\begin{tabular}{c}
\includegraphics[scale=0.9]{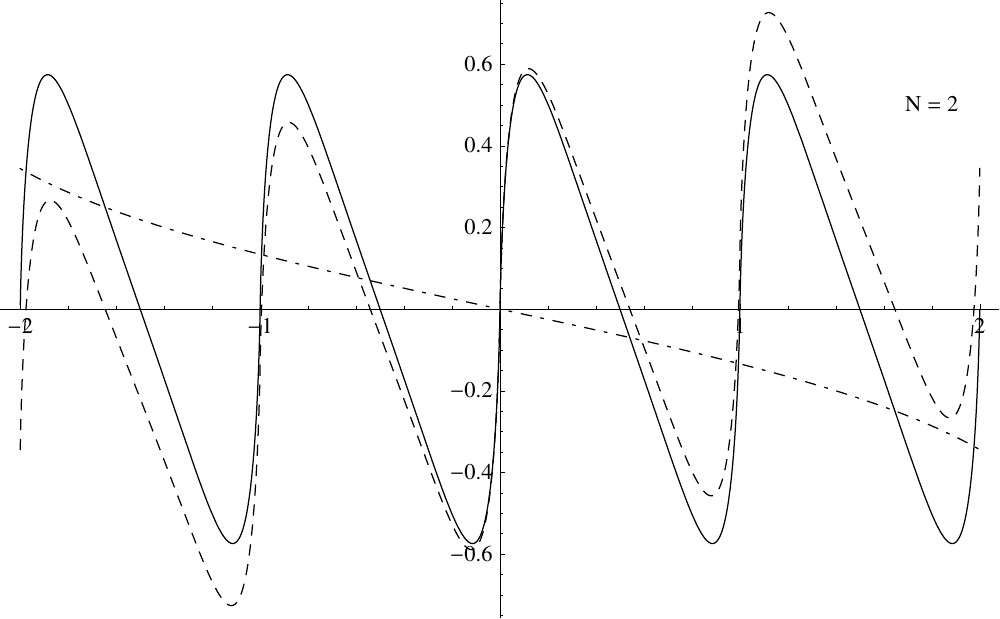}\tabularnewline
\includegraphics[scale=0.9]{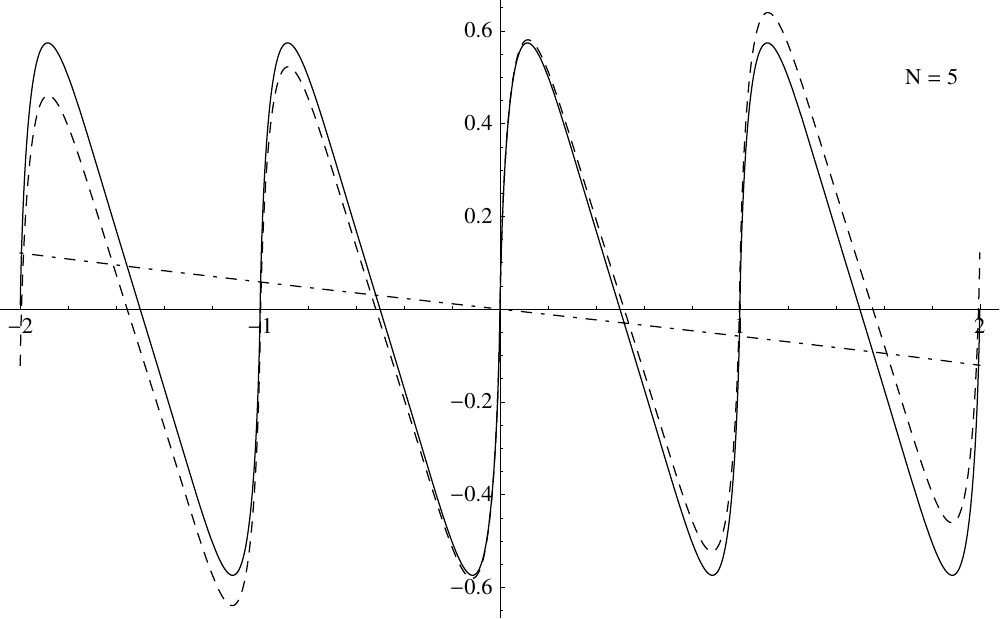}\tabularnewline
\includegraphics[scale=0.9]{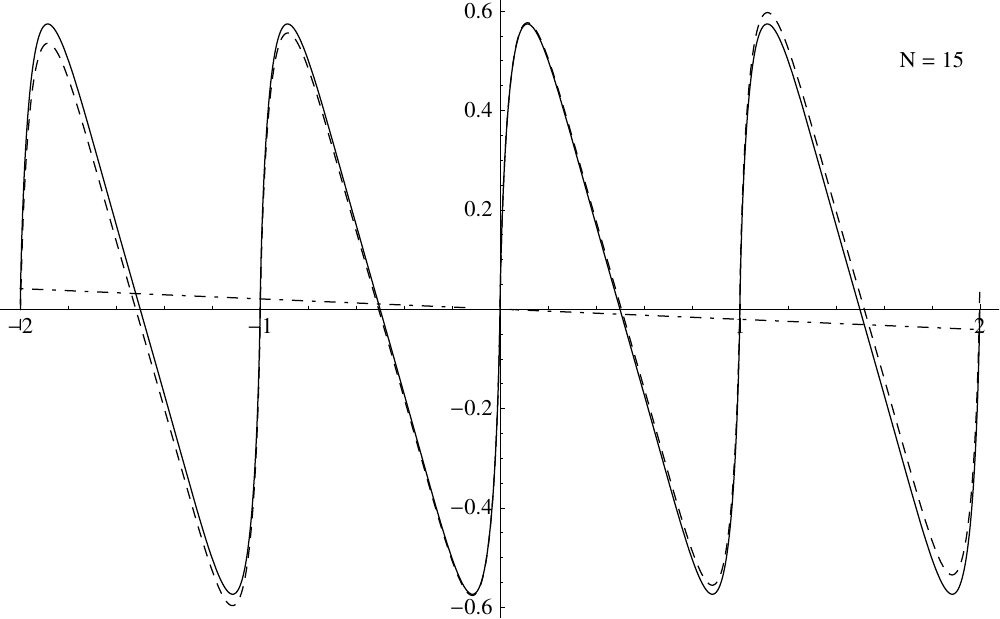}\tabularnewline
\tabularnewline
\end{tabular}

\caption{\label{fig:perzetai}$g(x)=\left(\mbox{PER}\varphi\right)(x)$. The
dashed line on the diagonal denotes the approximation error.}
\end{figure}

\begin{lem}
Hurwitz zeta-function is positive definite on $\mathbb{R}$, i.e.,
if $\varphi:\mathbb{R}\rightarrow\mathbb{C}$ is any finitely supported
function on $\mathbb{R}$, then
\begin{equation}
\sum_{x}\sum_{y}\overline{\varphi(x)}\varphi(y)Z(x-y)\geq0.\label{eq:Zpdf}
\end{equation}
\end{lem}
\begin{proof}
Computation of the double-sum in (\ref{eq:Zpdf}) yields
\begin{equation}
\sum_{n=0}^{\infty}\frac{1}{1+n^{2}}\left|\sum_{x}\varphi(x)\, e(nx)\right|^{2}\geq0.\label{eq:HZpos}
\end{equation}

\end{proof}
The next results yield a representation of all the partial isometries
between the two defect spaces determined by some chosen and fixed
finite subset $F$ of $\mathbb{T}$, as in Corollary \ref{cor:finite}.
But, by von Neumann\textquoteright{}s classification (Lemma \ref{lem:vN def-space}),
this will then also be a representation of all the selfadjoint extensions
of the basic Hermitian operator $L_{F}$ in (\ref{eq:Ln}) determined
by the set $F$. If the cardinality of $F$ is $m$, then the operator
$L_{F}$ has deficiency indices $(m,m)$, and the partial isometries
map between $m$-dimensional deficiency-spaces.
\begin{cor}
Let $Z$ be the Hurwitz zeta function from (\ref{eq:HL}), and let
$F\subset\mathbb{T}$ be a finite subset. Let
\begin{equation}
L_{F}:=H_{\{f\in\mathscr{H}_{2}\:\big|\: z\frac{d}{dz}f\in\mathscr{H}_{2},\, f=0\mbox{ on }F\}}.\label{eq:defLFF}
\end{equation}
Then the partial isometries $U_{F}$ between the two deficiency spaces
$\mathscr{D}_{\pm}(L_{F})$ from the von Neumann decomposition (\ref{eq:vN2})
in Lemma \ref{lem:vN def-space} are in bijective correspondence with
$\#F\times\#F$ complex matrices $\left(M_{\alpha,\beta}\right)_{(\alpha,\beta)\in F\times F}$
satisfying
\begin{equation}
\underset{(\gamma,\xi)\in F\times F}{\sum\sum}\overline{M}_{\gamma,\alpha}Z(\gamma-\xi)M_{\xi,\beta}=Z(\alpha-\beta)\label{eq:defmat}
\end{equation}
for all $(\alpha,\beta)\in F\times F$. \end{cor}
\begin{proof}
In Theorem \ref{thm:defLF}, we showed that each of the two deficiency
spaces $\mathscr{D}_{\pm}(L_{F})$ is an isomorphic image of the same
RKHS, the one from the kernel
\begin{equation}
K_{Z}(\alpha,\beta)=Z(\alpha-\beta)\label{eq:KLL}
\end{equation}
where $Z=Z_{Hurwitz}$ is the Hurwitz zeta-function. Hence a partial
isometry $U_{F}:\mathscr{D}_{+}(L_{F})\rightarrow\mathscr{D}_{-}(L_{F})$,
onto, will be acting on functions $\varphi$ on $F$ via the representation
(\ref{eq:Tvarphi})
\begin{equation}
U_{F}(\sum_{\alpha\in F}\varphi(\alpha)f_{+}^{(\alpha)})=\sum_{\alpha\in F}\left(M\varphi\right)(\alpha)f_{-}^{(\alpha)}\label{eq:isoF}
\end{equation}
where $M\varphi$ on the RHS in (\ref{eq:isoF}) has the following
matrix-representation:
\begin{equation}
\left(M\varphi\right)(\alpha)=\sum_{\beta\in F}M_{\alpha,\beta}\varphi(\beta).\label{eq:MF}
\end{equation}
But we are also viewing $M$ as an operator in $l^{2}(F)$ which is
finite-dimensional since $F$ is assumed finite. 

Substituting (\ref{eq:MF}) into (\ref{eq:isoF}), and unravelling
the isometric property of $U_{F}$, the desired conclusion (\ref{eq:defmat})
follows. To see this, notice (from Theorem \ref{thm:defLF}) that
\begin{equation}
\left\Vert \sum_{\alpha\in F}\varphi(\alpha)f_{+}^{(\alpha)}\right\Vert _{\mathscr{H}_{2}}^{2}=\underset{(\alpha,\beta)\in F\times F}{\sum\sum}\overline{\varphi(\alpha)}\varphi(\beta)Z(\alpha-\beta).\label{eq:norm}
\end{equation}

\end{proof}

\begin{cor}
\label{cor:lie}Let $F\subset\mathbb{T}$ be a finite subset (of distinct
points), $\#F=m$, and let 
\begin{equation}
K_{F}(\alpha,\beta):=Z(\alpha-\beta)\quad,(\alpha,\beta)\in F\times F\label{eq:KFab}
\end{equation}
be the corresponding kernel defined from restricting the Hurwitz zeta-funciton
$Z$. 
\begin{enumerate}
\item Then $K_{F}(\cdot,\cdot)$ is (strictly) positive definite on the
vector space $V_{F}=\mathbb{C}^{F}=$ all complex-valued functions
on $F$, i.e., $K_{F}(\cdot,\cdot)$ has rank $\#F$. 
\item The $(\#F)\times(\#F)$ complex matrices $M$ satisfying (\ref{eq:defmat})
form a compact Lie group $G(F)$ of transformations in $V_{F}=\mathbb{C}^{F}$. 
\end{enumerate}
\end{cor}
\begin{proof}
Set $m:=\#F$. The key step in the proof is the assertion that the
sesquilinear form $K_{F}(\cdot,\cdot)$ in (\ref{eq:KFab}) has full
rank, i.e., that its eigenvalues are all strictly positive. 

Then it follows from Lie theory (see e.g., \cite{Hel08}) that
\begin{equation}
G(F)=\{M\:\big|\: m\times m\mbox{ complex matrix s.t. }(\ref{eq:defmat})\mbox{ holds}\}\label{eq:GF}
\end{equation}
is a compact Lie group as stated.

It follows from (\ref{eq:Zpdf}) and (\ref{eq:norm}) that $K_{F}(\cdot,\cdot)$
is positive semi-definite. To show that it has full rank $=m$, we
must check that if $\varphi\in V_{F}=\mathbb{C}^{F}$ satisfying 
\begin{equation}
\sum_{\beta\in F}K_{F}(\alpha,\beta)\varphi(\beta)=0,\quad\forall\alpha\in F\label{eq:cpdf}
\end{equation}
then $\varphi=0$. 

Let $\varphi\in V_{F}$ satisfying (\ref{eq:cpdf}). Using (\ref{eq:norm}),
note that (\ref{eq:cpdf}) implies
\begin{equation}
\sum_{\beta\in F}Z(\alpha-\beta)\varphi(\beta)=0,\quad\forall\alpha\in F,\label{eq:cpdf1}
\end{equation}
and therefore, by (\ref{eq:HZpos})
\begin{equation}
\sum_{\beta\in F}\varphi(\beta)e(n\beta)=0,\quad\forall n\in\mathbb{N}_{0}.\label{eq:cpdf2}
\end{equation}
Now index the points $\{\beta\}$ in $F$ as follows $\beta_{1},\ldots,\beta_{m}$,
with corresponding $\zeta_{j}:=e(\beta_{j})=e^{i2\pi\beta_{j}},$$1\leq j\leq m$;
and set $\mathbb{N}_{m}:=\{0,1,2,\ldots,m-1\}$; then the matrix $\left(e(n\beta_{j})\right){}_{1\leq j\leq m,\: n\in\mathbb{N}_{m}}$
is a Vandermonde matrix
\begin{equation}
\left[\begin{array}{ccccc}
1 & 1 & \cdots & \cdots & 1\\
\zeta_{1} & \zeta_{2} & \cdots & \cdots & \zeta_{m}\\
\zeta_{1}^{2} & \zeta_{2}^{2} & \cdots & \cdots & \zeta_{m}^{2}\\
\vdots & \vdots & \vdots & \vdots & \vdots\\
\zeta_{1}^{m-1} & \zeta_{2}^{m-1} & \cdots & \cdots & \zeta_{m}^{m-1}
\end{array}\right]\label{eq:vmatrix}
\end{equation}
with determinant
\[
\prod_{1\leq j<k\leq m}\left(\zeta_{k}-\zeta_{j}\right)\neq0.
\]

Hence, translating back to the sesquilinear form $K_{F}$, we conclude
that $K_{F}$ is strictly positive definite, and that, therefore $G(F)$
is a compact Lie group of $m\times m$ complex matrices.
\end{proof}
We proved that whenever a finite subset $F\subset\mathbb{T}$ is chosen
as above, and if $M$ is an element in the corresponding Lie group
$G(F)$, then there is a unique selfadjoint extension $H_{M}$ corresponding
to the partial isometry induced by $M$, acting between the two deficiency
spaces for $L_{F}$. In the next result we compute the spectrum of
$H_{M}$. Each $H_{M}$ has pure point spectrum as $L_{F}$ has finite
deficiency indices $(m,m)$ where $m=(\#F)$. Since for finite index
all selfadjoint extensions have the same essential spectrum \cite{AG93};
and as a result we have pure point-spectrum.
\begin{example}
Let the closed subset $F$ of $\partial\mathbb{D}$ consist of the
two points $z_{\pm}=\pm1$. Then the compact group $G(F)$ from Corollary
\ref{cor:lie} is (up to conjugacy) the group of $2\times2$ complex
matrices preserving the quadratic form
\begin{equation}
\mathbb{C}^{2}\ni(z_{1},z_{2})\mapsto K_{ev}\left|z_{1}\right|^{2}+K_{odd}\left|z_{2}\right|^{2}\label{eq:GFF}
\end{equation}
where
\begin{equation}
\begin{cases}
K_{ev} & :={\displaystyle \sum_{n=0}^{\infty}\frac{1}{1+(2n)^{2}}},\quad\mbox{and}\\
\\
K_{odd} & :={\displaystyle \sum_{n=0}^{\infty}\frac{1}{1+(1+2n)^{2}}};
\end{cases}\label{eq:Ksplit}
\end{equation}
i.e., the splitting of the summation ${\displaystyle \sum_{k\in\mathbb{N}_{0}}=\frac{1}{2}\left(1+\pi\coth(\pi)\right)}$
into even and odd parts.\end{example}
\begin{proof}
Computation of the Hurwitz zeta-function at the two points $F=\{\pm1\}$
yields the two numbers $K_{ev}$ and $K_{odd}$ in (\ref{eq:Ksplit}).

Note $0<K_{odd}<K_{ev}$. Hence when the matrix $K_{Z_{F}}$ in (\ref{eq:KernelH})
is computed for $F=\{\pm1\}$, we get for eigenvalues the two numbers
in (\ref{eq:Ksplit}). The corresponding system of normalized eigenvectors
in $\mathbb{C}^{2}$ is 
\[
\left\{ \frac{1}{\sqrt{2}}\left(\begin{array}{c}
1\\
1
\end{array}\right),\frac{1}{\sqrt{2}}\left(\begin{array}{c}
1\\
-1
\end{array}\right)\right\} .
\]
The assertion in (\ref{eq:Ksplit}) follows from this.\end{proof}
\begin{cor}
\label{cor:Liesp}Let $F\subset\mathbb{T}$ be finite, and let $M\in G(F)$
where $G(F)$ is the Lie group from Corollary \ref{cor:lie}. Then
$\lambda\in\mathbb{R}$ is in the spectrum of the selfadjoint extension
$H_{M}$ if and only if there is some $\psi\in\mathbb{C}^{F}$(a complex
valued function on $F$) such that $\lambda$ is a root in the following
function
\begin{equation}
F_{M}(\lambda):=\sum_{k\in\mathbb{N}_{0}}\sum_{\alpha\in F}e(k\alpha)\frac{\psi(\alpha)(\lambda-i)(k+i)+\left(M\psi\right)(\alpha)(\lambda+i)(k-i)}{\left(k-\lambda\right)\left(k^{2}+1\right)}.\label{eq:liesp}
\end{equation}
Moreover, we have
\begin{equation}
\frac{dF_{M}}{d\lambda}(\lambda)=\sum_{k\in\mathbb{N}_{0}}\sum_{\alpha\in F}e(k\alpha)\frac{\psi(\alpha)+\left(M\psi\right)(\alpha)}{\left(k-\lambda\right)^{2}}.\label{eq:dFM}
\end{equation}

It follows that the selfadjoint extension $H_{M}$ has the same qualitative
spectral configuration as we described in our results from section
\ref{sec:spH}, which deal only with the special case of deficiency
indices $(1,1)$. From our spectral generating function $F_{M}$ and
its derivative, given above, it follows that the spectral picture
in the $(m,m)$ case is qualitatively the same, now for $m>1$, as
we found in section \ref{sec:spH} in the special case of $m=1:$
Only point-spectrum\emph{:} and when one of the selfadjoint extensions
$H_{M}$ is fixed, we get eigenvalues distributed in each of the intervals
$(-\infty,0)$, and $[n,n+1)$ for $n\in\mathbb{N}_{0}$. But excluding
$(-\infty,0)$ for the case of the Friedrichs extension.\end{cor}
\begin{proof}
Now the partial isometries $U{}_{M}:\mathscr{D}_{+}(L_{F})\rightarrow\mathscr{D}_{-}(L_{F})$
from Lemma \ref{lem:vN def-space} are given by $\mathbb{C}^{F}\ni\varphi\mapsto M\varphi\in\mathbb{C}^{F}$
via the formula (\ref{eq:defmat}) from Corollary \ref{cor:lie},
where $M\in G(F)$. For $\psi\in\mathbb{C}^{F}$, set
\begin{equation}
f_{+}(\psi)=\sum_{\alpha\in F}\psi(\alpha)f_{+}^{(\alpha)}=\sum_{\alpha\in F}\left(\sum_{k\in\mathbb{N}_{0}}\frac{\psi(\alpha)e(k\alpha)}{k-i}z^{k}\right)\in\mathscr{D}_{+}(L_{F});\label{eq:fp}
\end{equation}
and 
\begin{equation}
f_{-}(\psi)=\sum_{\alpha\in F}\psi(\alpha)f_{-}^{(\alpha)}=\sum_{\alpha\in F}\left(\sum_{k\in\mathbb{N}_{0}}\frac{\psi(\alpha)e(k\alpha)}{k+i}z^{k}\right)\in\mathscr{D}_{-}(L_{F}).\label{eq:fm}
\end{equation}
Using now the characterization of the selfadjoint extensions $H_{M}$
($M\in G(F)$) of the initial operator $L_{F}$, we get:
\begin{equation}
H_{M}(g+f_{+}(\psi)+f_{-}(M\psi))=Lg+i\left(f_{+}(\psi)-f_{-}(M\psi)\right)\label{eq:ffext}
\end{equation}
valid for all $g\in\mathscr{D}(L_{F})$, and all $\psi\in\mathbb{C}^{F}$.
Indeed by Lemma \ref{lem:vNext} the vectors $f$ in $\mathscr{D}(H_{M})$
must have the form 
\begin{equation}
f=g+f_{+}(\psi)+f_{-}(M\psi)\label{eq:extdomain}
\end{equation}
where $g\in\mathscr{D}(L_{F})$, i.e., $g(\alpha)=0$, $\forall\alpha\in F$,
and where $\psi\in\mathbb{C}^{F}$. But (\ref{eq:extdomain}) has
an $l^{2}(\mathbb{N}_{0})$-representation as follows ($k\in\mathbb{N}_{0}$):
\begin{equation}
f_{k}=g_{k}+\sum_{\alpha\in F}\psi(\alpha)\frac{e(k\alpha)}{k-i}+\sum_{\alpha\in F}\left(M\psi\right)(\alpha)\frac{e(k\alpha)}{k+i}.\label{eq:extfk}
\end{equation}
For details on $(M\psi)(\alpha)$, see (\ref{eq:isoF}). 

Hence, the eigenvalue problem (for $\lambda\in\mathbb{R}$)
\begin{equation}
H_{M}f=\lambda f,\quad f\in\mathscr{D}(H_{M})\label{eq:Feig}
\end{equation}
takes the following form:
\begin{eqnarray}
 &  & kg_{k}+i\sum_{\alpha\in F}e(k\alpha)\left(\frac{\psi(\alpha)}{k-i}-\frac{\left(M\psi\right)(\alpha)}{k+i}\right)\nonumber \\
 & = & \lambda g_{k}+\lambda\sum_{\alpha\in F}e(k\alpha)\left(\frac{\psi(\alpha)}{k-i}+\frac{\left(M\psi\right)(\alpha)}{k+i}\right);\label{eq:Feigen2}
\end{eqnarray}
which in turn simplifies as follows: The function
\begin{eqnarray*}
F_{M}(\lambda) & = & \sum_{k\in\mathbb{N}_{0}}\sum_{\alpha\in F}\frac{e(k\alpha)}{\left(k-\lambda\right)}\left[\lambda\left(\frac{\psi(\alpha)}{k-i}+\frac{\left(M\psi\right)(\alpha)}{k+i}\right)-i\left(\frac{\psi(\alpha)}{k-i}-\frac{\left(M\psi\right)(\alpha)}{k+i}\right)\right]\\
 & = & \sum_{k\in\mathbb{N}_{0}}\sum_{\alpha\in F}e(k\alpha)\frac{\psi(\alpha)(\lambda-i)(k+i)+\left(M\psi\right)(\alpha)(\lambda+i)(k-i)}{\left(k-\lambda\right)\left(k^{2}+1\right)}.
\end{eqnarray*}
So $\psi\in\mathbb{C}^{F}$ must be such that the function 
\begin{equation}
F_{M}^{(\psi)}(\lambda)=\sum_{k\in\mathbb{N}_{0}}\sum_{\alpha\in F}e(k\alpha)\frac{\psi(\alpha)(\lambda-i)(k+i)+\left(M\psi\right)(\alpha)(\lambda+i)(k-i)}{\left(k-\lambda\right)\left(k^{2}+1\right)}\label{eq:FFla}
\end{equation}
has $\lambda$ as a root, i.e., $F_{M}(\lambda)=0$ must hold for
points $\lambda\in spect(H_{M})$; and conversely if $F_{M}(\lambda)=0$,
then the vector $f$ in (\ref{eq:extfk}) will be an eigenvector,
note
\[
\sum_{k\in\mathbb{N}_{0}}g_{k}\, e(k\alpha)=0,\quad\forall\alpha\in F.
\]

We proceed to verify (\ref{eq:dFM}). Setting
\begin{eqnarray*}
A & := & \psi(\alpha)\left(k+i\right)\\
B & := & \left(M\psi\right)(\alpha)\left(k-i\right)
\end{eqnarray*}
and 
\[
g_{k,\alpha}(\lambda):=\frac{A(\lambda-i)+B(\lambda+i)}{\left(k-\lambda\right)\left(k^{2}+1\right)};
\]
then from (\ref{eq:liesp}), we have 
\[
F_{M}(\lambda)=\sum_{k\in\mathbb{N}_{0}}\sum_{\alpha\in F}e(k\alpha)g_{k,\alpha}(\lambda).
\]
Note that 
\begin{equation}
g_{k,\alpha}'(\lambda)=\frac{\left(A+B\right)k-i(A-B)}{\left(k-\lambda\right)^{2}\left(k^{2}+1\right)},\label{eq:tderiv}
\end{equation}
and the numerator in (\ref{eq:tderiv}) is given by 
\begin{eqnarray*}
 &  & \left(A+B\right)k-i(A-B)\\
 & = & k\left(\psi(\alpha)\left(k+i\right)+\left(M\psi\right)(\alpha)\left(k-i\right)\right)-i\left(\psi(\alpha)\left(k+i\right)-\left(M\psi\right)(\alpha)\left(k-i\right)\right)\\
 & = & \psi(\alpha)\left(k(k+i)-i(k+i)\right)+\left(M\psi\right)(\alpha)\left(k(k-i)+i(k-i)\right)\\
 & = & \left(\psi(\alpha)+\left(M\psi\right)(\alpha)\right)\left(k^{2}+1\right).
\end{eqnarray*}
Substitute the above equation into (\ref{eq:tderiv}), we get
\begin{eqnarray*}
g_{k,\alpha}'(\lambda) & = & \frac{\left(\psi(\alpha)+\left(M\psi\right)(\alpha)\right)\left(k^{2}+1\right)}{\left(k-\lambda\right)^{2}\left(k^{2}+1\right)}\\
 & = & \frac{\psi(\alpha)+\left(M\psi\right)(\alpha)}{\left(k-\lambda\right)^{2}}.
\end{eqnarray*}
It follows that
\begin{eqnarray*}
\frac{dF_{M}}{d\lambda} & = & \sum_{k\in\mathbb{N}_{0}}\sum_{\alpha\in F}g_{k,\alpha}'(\lambda)\\
 & = & \sum_{k\in\mathbb{N}_{0}}\sum_{\alpha\in F}e(k\alpha)\frac{\psi(\alpha)+\left(M\psi\right)(\alpha)}{\left(k-\lambda\right)^{2}}
\end{eqnarray*}
which is eq. (\ref{eq:dFM}).

Note in the computation of the derivative we get cancellation of the
factor $(k^{2}+1)$ in numerator and denominator.
\end{proof}

\subsection{A Comparison}

Below we offer a comparison of the extension theory in the subspace
$\mathscr{L}_{+}$ and in the ambient Hilbert space $L^{2}(\mathbb{R}/\mathbb{Z})\simeq L^{2}(I)$,
where we are using the usual identification between the quotient $\mathbb{R}/\mathbb{Z}$
and a choice of a period interval $I$. There is a slight notational
ambiguity, as $L$ may be understood as refer to a Hermitian operator
with dense domain, referring each of the two Hilbert spaces $\mathscr{L}_{+}$
and $L^{2}(\mathbb{R}/\mathbb{Z})\simeq L^{2}(I)$, where $I=[0,1)$;
see Table \ref{tab:H}. But the boundary conditions $f(0)=f(1)=0$
make sense in both cases; and in both cases, it is understood that
$f$ and $\frac{d}{dx}f$ are in $L^{2}$.
\begin{lem}
Let $L$ be the above mentioned Hermitian operator with dense domain
$\mathscr{D}(L)$ in $L^{2}(\mathbb{R}/\mathbb{Z})$. For $\zeta=e(\theta)\in\mathbb{T}$,
let $H_{\theta}(=H_{\zeta})$ be the corresponding selfadjoint extension;
see the von Neumann classification, Lemma \ref{lem:vN def-space},
and let $\mathscr{H}(\theta)$ be the Hilbert space:
\begin{equation}
f:\mathbb{R}\rightarrow\mathbb{C}\quad\mbox{measurable, and in }L_{loc}^{2};\label{eq:fR}
\end{equation}
\begin{equation}
f(x+n)=e(n\theta)f(x),\quad\forall x\in\mathbb{R},\forall n\in\mathbb{Z};\mbox{ and}\label{eq:fn}
\end{equation}
\begin{equation}
\left\Vert f\right\Vert _{\mathscr{H}(\theta)}^{2}=\int_{\mathbb{R}/\mathbb{Z}}\left|f(x)\right|^{2}dx<\infty\label{eq:fnorm}
\end{equation}
(Note that the integral in (\ref{eq:fnorm}) makes sense on account
of (\ref{eq:fn}), i.e., $\left|f\right|^{2}$ is a $1$-periodic
function on $\mathbb{R}$. )
\begin{enumerate}
\item Then for every $\theta$, the restriction mapping $\mathscr{H}(\theta)\overset{T_{\theta}}{\rightarrow}L^{2}(I)$
is a unitary isometric isomorphism, and
\begin{equation}
U_{\theta}(t)=T_{\theta}U(t)T_{\theta}^{*},\quad t\in\mathbb{R},\;\mbox{ on }L^{2}(I)\label{eq:U}
\end{equation}
yields all the unitary one-parameter groups $\{U_{\theta}(t)\}$ corresponding
to the selfadjoint extensions of $L$. In (\ref{eq:fn}) RHS, $\{U(t)\}_{t\in\mathbb{R}}$
is the periodic translation $f\mapsto f(x+t)$ acting in the Hilbert
space $\mathscr{H}(\theta)$. 
\item If $\theta\in[0,1)$ is the parameter of the von Neumann classification,
then $\{U_{\theta}(t)\}_{t\in\mathbb{R}}$ in (\ref{eq:fn}) leaves
invariant the subspace $\mathscr{L}_{+}$ if and only if $\theta=0$.
\end{enumerate}
\end{lem}
\begin{proof}
See the discussion above. The construction in (\ref{eq:fn}) is an
example of an induced representation; induction from $\mathbb{Z}$
up to $\mathbb{R}$; see \cite{Mac88}. 

It follows from (\ref{eq:U}) that for fixed $\theta\in[0,1)$ the
spectrum of the unitary one-parameter group $\{U_{\theta}(t)\}_{t\in\mathbb{R}}$
and its selfadjoint generator $H_{\theta}$ in $L^{2}(I)\simeq L^{2}(\mathbb{R}/\mathbb{Z})$
is $\{e_{\theta+n}\:\big|\: n\in\mathbb{Z}\}$, where $e_{\varphi}(x)=e(\varphi x)=e^{i2\pi\varphi x}$.
Now let $P_{\mathscr{L}_{+}}$ be the projection of $L^{2}(I)$ onto
$\mathscr{L}_{+}$, then a computation yields
\[
\left\Vert P_{\mathscr{L}_{+}}e_{\varphi}\right\Vert ^{2}=1-\frac{\sin^{2}(\pi\varphi)}{\pi^{2}}\zeta_{1}(\varphi).
\]
see Lemma \ref{lem:Hzeta} and Figure \ref{fig:sh} below.  Apply
this to $\varphi=\theta+n$, and the last conclusion in the lemma
follows from this. 
\end{proof}
\begin{figure}
\includegraphics[scale=0.9]{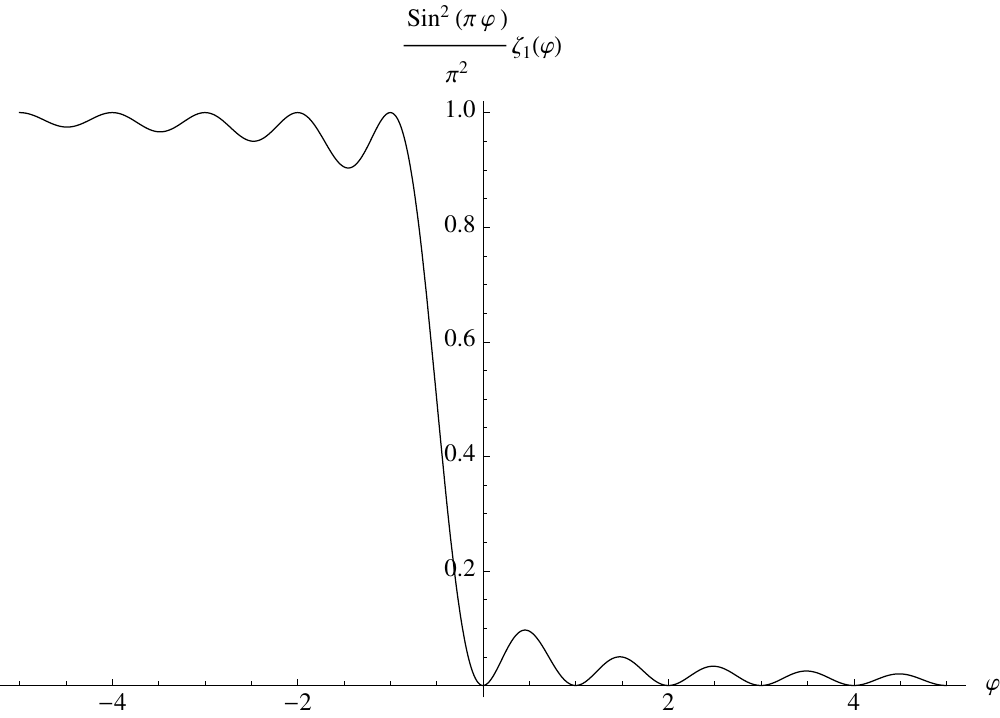}

\caption{\label{fig:sh}The function $\frac{\sin^{2}(\pi\varphi)}{\pi^{2}}\,\zeta_{1}(\varphi)$ }

\end{figure}

\begin{lem}
\label{lem:Hzeta}Let 
\begin{equation}
\zeta_{1}(\varphi)=\sum_{n=1}^{\infty}\frac{1}{(\varphi+n)^{2}}\label{eq:Hz}
\end{equation}
be the Hurwitz zeta function. (Note, the summation in (\ref{eq:Hz})
begins at $n=1$.) Then
\begin{equation}
\left\Vert P_{\mathscr{L}_{+}}e_{\varphi}\right\Vert ^{2}=1-\frac{\sin^{2}(\pi\varphi)}{\pi^{2}}\zeta_{1}(\varphi).\label{eq:PL}
\end{equation}
\end{lem}
\begin{proof}
In the verification of (\ref{eq:PL}), it is convenient to choose
$I=[-\frac{1}{2},\frac{1}{2})$ as period interval in the duality
of Lemma \ref{lem:H2}. For $P_{\mathscr{L}_{+}}e_{\varphi}$ we get
\begin{eqnarray*}
P_{\mathscr{L}_{+}}e_{\varphi} & = & \sum_{n=0}^{\infty}\left\langle e_{n},e_{\varphi}\right\rangle _{L^{2}(I)}e_{n}\\
 & = & \sum_{n=0}^{\infty}\frac{\sin\pi(\varphi-n)}{\pi(\varphi-n)}e_{n};\;\mbox{ and}
\end{eqnarray*}
therefore
\begin{eqnarray*}
\left\Vert P_{\mathscr{L}_{+}}e_{\varphi}\right\Vert ^{2} & = & \sum_{n=0}^{\infty}\frac{\sin^{2}(\pi(\varphi-n))}{\left(\pi(\varphi-n)\right)^{2}}\,\,\,\,\,\,\,\,\,\,\,\,\,\,\,(\mbox{by Parseval})\\
 & = & \frac{\sin^{2}(\pi\varphi)}{\pi^{2}}\sum_{n=0}^{\infty}\frac{1}{\left(\varphi-n\right)^{2}}.
\end{eqnarray*}
Now compare this to
\[
1=\left\Vert e_{\varphi}\right\Vert _{L^{2}(I)}^{2}=\frac{\sin^{2}(\pi\varphi)}{\pi^{2}}\sum_{n\in\mathbb{Z}}\frac{1}{\left(\varphi-n\right)^{2}};
\]
also by Parseval. Subtraction yields
\begin{equation}
1-\left\Vert P_{\mathscr{L}_{+}}e_{\varphi}\right\Vert ^{2}=\frac{\sin^{2}(\pi\varphi)}{\pi^{2}}\zeta_{1}(\varphi).\label{eq:LZ}
\end{equation}
To get this, change variable in the summation range $-\infty<n\leq-1$.
The desired conclusion (\ref{eq:PL}) is now immediate from (\ref{eq:LZ}).
\end{proof}

\subsection{\label{sub:LadH}The Operator $L^{*}$}

We now turn to von Neumann\textquoteright{}s boundary theory for the
Hermitian operators $L_{F}$ in the complex case (see Corollary \ref{cor:F}),
so the case when the Hilbert space is $\mathscr{H}_{2}$ (= the Hardy
space). We already proved that the deficiency subspaces (in the sense
of von Neumann), are different from their analogues for boundary value
problems in the real case, see e.g., Theorem \ref{thm:defLF}. The
two deficiency spaces are RKHSs, and they depend on conditions assigned
on the prescribed closed subset $F\subset\partial\mathbb{D}$, of
(angular) measure $0$.

When such a subset $F$ is fixed, let $L_{F}$ be the corresponding
Hermitian operator. But now the adjoint operators $L_{F}^{*}$, defined
relative to the inner product in $\mathscr{H}_{2}$, turns out no
longer to be differential operators. The nature of these operators
$L_{F}^{*}$ depends on use of analytic function theory in an essential
way, and it is studied below. By analogy to the real case, one might
guess that $L_{F}^{*}$ is again a differential operator acting on
a suitable domain in $\mathscr{H}_{2}$, but it is not; -- not even
in the simplest case when the set $F$ is a singleton. 
\begin{lem}
\begin{flushleft}
\label{lem:Lad}Let $L$ be the Hermitian operator as before, i.e.,
$L$ is the restriction of ${\displaystyle H=z\frac{d}{dz}}$ on the
dense domain $\mathscr{D}(L)$ in $\mathscr{H}_{2}$, consisting of
functions $f\in\mathscr{D}(H)$ s.t. $\tilde{f}(1)=0$. On meromorphic
functions $f$, set 
\begin{eqnarray}
\left(CP_{1}f\right)(z) & := & \left(\lim_{w\rightarrow1}(w-1)f(w)\right)\frac{1}{z-1}\label{eq:cp1}\\
 & = & \frac{1}{z-1}\frac{1}{2\pi i}\oint_{\gamma\: at\: z=1}f(w)dw\label{eq:cp1i}
\end{eqnarray}
where the contour is chosen as a circle centered at $z=1$. Then on
its domain, as an operator, $L^{*}$ acts as follows: 
\begin{equation}
L^{*}=\left(1-CP_{1}\right)z\frac{d}{dz}.\label{eq:LadjH}
\end{equation}

\par\end{flushleft}\end{lem}
\begin{proof}
Recall that in our model in the Hardy space, we have $L\subset H\subset L^{*}$
(see (\ref{eq:LHL})), and 
\[
\dim\mathscr{D}(H)/\mathscr{D}(L)=\dim\mathscr{D}(L^{*})/\mathscr{D}(H)=1.
\]
In particular, the two defect vectors (\ref{eq:Y2}) and (\ref{eq:Y3})
have the following representation:
\begin{eqnarray}
y_{2}=\left(\frac{1}{1+n^{2}}\right)_{n\in\mathbb{N}_{0}} & \mapsto & \rho_{2}(z):=\sum_{n=0}^{\infty}\frac{z^{n}}{1+n^{2}}\label{eq:rhoy2}\\
y_{3}=\left(\frac{n}{1+n^{2}}\right)_{n\in\mathbb{N}_{0}} & \mapsto & \rho_{3}(z):=\sum_{n=0}^{\infty}\frac{n\, z^{n}}{1+n^{2}}.\label{eq:rhoy3}
\end{eqnarray}
By von Neumann's theory, the domain of $L^{*}$ is characterized by
\begin{equation}
\mathscr{D}(L^{*})=\left\{ \varphi(z)+a\rho_{2}(z)+b\rho_{3}(z)\;\big|\;\varphi\in\mathscr{D}(L),\;\mbox{and }a,b\in\mathbb{C}\right\} ;\label{eq:domLadH}
\end{equation}
and 
\begin{eqnarray}
L^{*}\left(\varphi(z)+a\rho_{2}(z)+b\rho_{3}(z)\right) & = & L\varphi(z)+a\rho_{2}(z)+b\rho_{3}(z)\nonumber \\
 & = & z\frac{d}{dz}\varphi(z)+a\rho_{3}(z)-b\rho_{2}(z).\label{eq:LadH}
\end{eqnarray}

Note for all $z\in\mathbb{D}$, we have 
\[
z\frac{d}{dz}\rho_{2}(z)=\sum_{n=0}^{\infty}\frac{n\, z^{n}}{1+n^{2}}=\rho_{3}(z);
\]
and
\begin{eqnarray*}
z\frac{d}{dz}\rho_{3}(z) & = & \sum_{n=0}^{\infty}\frac{n^{2}\, z^{n}}{1+n^{2}}\\
 & = & -\sum_{n=0}^{\infty}\frac{z^{n}}{1+n^{2}}+\sum_{n=0}^{\infty}z^{n}\\
 & = & -\rho_{2}(z)+\frac{1}{1-z}.
\end{eqnarray*}
As a result, for all $f(z)=\varphi(z)+a\rho_{2}(z)+b\rho_{3}(z)\in\mathscr{D}(L^{*})$,
we see that
\begin{eqnarray}
\left(z\frac{d}{dz}f\right)(z) & = & z\frac{d}{dz}\varphi(z)+a\rho_{3}(z)-b\rho_{2}(z)+\frac{b}{1-z}\nonumber \\
 & = & \left(L^{*}f\right)(z)+\frac{b}{1-z}.\label{eq:Ladtemp}
\end{eqnarray}

Now, the last term in the above equation has a simple pole at $z=1$,
and it is extracted by the map $CP_{1}$ in (\ref{eq:cp1i}) as follows:
\begin{equation}
\frac{b}{1-z}=CP_{1}z\frac{d}{dz}f.\label{eq:Ladcp}
\end{equation}
Hence (\ref{eq:Ladtemp}) and (\ref{eq:Ladcp}) together yield 
\begin{eqnarray*}
\left(L^{*}f\right)(z) & = & \left(z\frac{d}{dz}f\right)(z)-CP_{1}z\frac{d}{dz}f\\
 & = & \left(1-CP_{1}\right)z\frac{d}{dz}f(z)
\end{eqnarray*}
which is (\ref{eq:LadjH}).\end{proof}
\begin{cor}
If $g\in\mathscr{D}(L^{*})$ is such that ${\displaystyle z\frac{d}{dz}g\in\mathscr{H}_{2}}$,
then 
\begin{equation}
L^{*}g=z\frac{d}{dz}g.\label{eq:LadHa}
\end{equation}

\end{cor}
Lemma \ref{lem:Lad} generalizes naturally to the operator $L_{F}$,
where the boundary conditions are specified by a finite subset $F\subset\mathbb{T}$;
see Corollaries \ref{cor:finite} and \ref{cor:F}.
\begin{thm}
\label{thm:LFad}Let $F=\{\zeta_{j}\}_{j=1}^{m}$ be a finite subset
of $\mathbb{T}=\partial\mathbb{D}$. Let $L_{F}$ with domain $\mathscr{D}(L_{F})$
be as Corollary \ref{cor:F}. Specifically, 
\begin{equation}
\mathscr{D}(L_{F})=\left\{ f\in\mathscr{H}_{2}\:\big|\: z\frac{d}{dz}f\in\mathscr{H}_{2},\:\tilde{f}=0\mbox{ on }F\right\} .\label{eq:DLFH}
\end{equation}
Recall $L_{F}$ has deficiency-indices $(m,m)$. Set 
\begin{equation}
\left(CP_{F}(f)\right)(z):=\sum_{j=1}^{m}\left(\lim_{w\rightarrow\zeta_{j}}\left(w-\zeta_{j}\right)f(w)\right)\frac{1}{z-\zeta_{j}}.\label{eq:CPF}
\end{equation}
Then $L^{*}$ acts on $\mathscr{D}(L^{*})$ as follows: 
\begin{equation}
L^{*}=\left(1-CP_{F}\right)z\frac{d}{dz}.\label{eq:LFadCPF}
\end{equation}

\end{thm}
In justifying formula (\ref{eq:LFadCPF}) we need the following
\begin{lem}
Let $F=\{\zeta_{j}\}_{j=1}^{m}\subset\partial\mathbb{D}$ be as above,
and let $\mathscr{D}(L_{F}^{*})$ be the domain of the adjoint operator
$\left(L_{F}\right)^{*}$ acting in $\mathscr{H}_{2}$. 

Then the functions in the subspace
\begin{equation}
\mathscr{G}_{F}:=\left\{ z\frac{d}{dz}f\:\big|\: f\in\mathscr{D}(L_{F}^{*})\right\} \label{eq:GFH}
\end{equation}
have the following properties:
\begin{enumerate}
\item Every $g\in\mathscr{G}_{F}$ is analytic in $\mathbb{D}$;
\item But not every $g\in\mathscr{G}_{F}$ is in $\mathscr{H}_{2}$;
\item The functions $g\in\mathscr{G}_{F}$ are meromorphic with possible
poles of order at most one, and the poles are contained in $F$.
\end{enumerate}
\end{lem}
\begin{proof}
The justification of conclusions (1)-(3) in the lemma will follow
from the computations below, but we already saw that 
\[
g(z):=z\frac{d}{dz}\rho_{3}(z)=-\rho_{2}(z)+\frac{1}{1-z}
\]
holds, where $\rho_{2}$ and $\rho_{3}$ are the functions in (\ref{eq:rhoy2})-(\ref{eq:rhoy3}).
Note that ${\displaystyle g=z\frac{d}{dz}\rho_{3}\in\mathscr{G}_{F}}$
where $F=\{1\}$ is the singleton.
\end{proof}

\begin{proof}[Proof of Theorem \ref{thm:LFad}]
See the proof of Lemma \ref{lem:Lad}. In detail:

The system of vectors 
\begin{equation}
f_{\pm}^{(j)}(z)=\sum_{n=0}^{\infty}\frac{\overline{\zeta}_{j}^{n}}{n\mp i}z^{n}=\sum_{n=0}^{\infty}\frac{1}{n\mp i}\left(\overline{\zeta}_{j}z\right)^{n}\label{eq:sdef}
\end{equation}
is a generating system of vectors for the two deficiency spaces $\mathscr{D}_{\pm}^{(F)}$
for $L_{F}$. 

Introducing, $R_{2}(w,z)=\rho_{2}(\overline{w}z)$, and $R_{3}(w,z)=\rho_{3}(\overline{w}z)$
for $z\in\mathbb{D}$, and $w\in\partial\mathbb{D}$, we conclude
that $\mathscr{D}_{+}+\mathscr{D}_{-}$ is spanned by the $2m$ functions
\begin{equation}
z\mapsto R_{2}(\zeta_{j},z)\quad\mbox{and}\quad z\mapsto R_{3}(\zeta_{j},z).\label{eq:gdef}
\end{equation}
Moreover, 
\begin{equation}
\begin{cases}
z\frac{d}{dz}R_{3}(\zeta_{j},z)=-R_{2}(\zeta_{j},z)+\frac{1}{1-\overline{\zeta}_{j}z}\\
z\frac{d}{dz}R_{2}(\zeta_{j},z)=R_{3}(\zeta_{j},z), & z\in\mathbb{D},\: i\leq j\leq m,
\end{cases}\label{eq:sdefH}
\end{equation}
where we recall that the second term of the RHS in (\ref{eq:sdefH})
is the \emph{Szegö-kernel},
\begin{equation}
K_{S_{z}}(\zeta_{j},z)=\left(1-\overline{\zeta}_{j}z\right)^{-1}=-\frac{\zeta_{j}}{z-\zeta_{j}},\quad z\in\mathbb{D},\:1\leq j\le m.\label{eq:Ksz}
\end{equation}
Hence, application of $CP_{F}$ (in (\ref{eq:CPF})) to $K_{S_{z}}(\zeta_{j},\cdot)$
yields
\begin{equation}
CP_{F}K_{S_{z}}(\zeta_{j},z)=K_{S_{z}}(\zeta_{j},z),\quad z\in\mathbb{D}.\label{eq:CPFKS}
\end{equation}

Combining this with the arguments from the proof of Lemma \ref{lem:Lad},
we now conclude that the desired formula (\ref{eq:LFadCPF}) holds
on $\mathscr{D}(L_{F}^{*})=\mathscr{D}(L_{F})+\mathscr{D}_{+}+\mathscr{D}_{-}$,
where we used (\ref{eq:sdef}) and Corollary \ref{cor:Fdef} in the
last step.\end{proof}
\begin{cor}
\label{cor:BadHF}Let $F=\{\zeta_{j}\}_{j=1}^{m}\subset\mathbb{T}$
and $L_{F}$ be as in Theorem \ref{thm:LFad}. Then the boundary form
in (\ref{eq:BoundaryForm}) is given by
\begin{equation}
\boldsymbol{B}(f,f)=\Im\left\{ \sum_{j=1}^{m}C_{j}(f)\overline{\tilde{f}(\zeta_{j})}\right\} \label{eq:BLFH}
\end{equation}
for all $f\in\mathscr{D}(L^{*})$, where 
\begin{equation}
C_{j}(f):=\lim_{w\rightarrow\zeta_{j}}(w-\zeta_{j})\left(w\frac{d}{dw}f(w)\right),\quad j=1,2,\ldots,m;\label{eq:CLFH}
\end{equation}
and $\tilde{f}$ is the continuous extension of $f$ onto the boundary
$\partial\mathbb{D}$. Note that $C_{j}(f)$ is the residue of ${\displaystyle z\frac{d}{dz}f}$
at the simple pole $z=\zeta_{j}$, $j=1,2,\ldots,m$.\end{cor}
\begin{proof}
By eq. (\ref{eq:LFadCPF}), we have 
\begin{eqnarray*}
\left\langle L^{*}f,f\right\rangle  & = & \left\langle z\frac{d}{dz}f,f\right\rangle -\left\langle CP_{F}z\frac{d}{dz}f,f\right\rangle \\
\left\langle f,L^{*}f\right\rangle  & = & \left\langle f,z\frac{d}{dz}f\right\rangle -\left\langle f,CP_{F}z\frac{d}{dz}f\right\rangle .
\end{eqnarray*}
Note by (\ref{eq:CPF}) and (\ref{eq:CLFH}), we see that
\[
CP_{F}z\frac{d}{dz}f(z)=\sum_{j=1}^{m}C_{j}(f)K_{S_{z}}(\zeta_{j},z)
\]
where $K_{S_{z}}(\zeta_{j},z)$ is the \emph{Szegö-kernel} (\ref{eq:Ksz}). 

Hence 
\begin{eqnarray*}
2i\boldsymbol{B}(f,f) & = & \left\langle L^{*}f,f\right\rangle -\left\langle f,L^{*}f\right\rangle \\
 & = & 2i\,\Im\left\{ \left\langle f,CP_{F}z\frac{d}{dz}f\right\rangle \right\} \\
 & = & 2i\,\Im\left\{ \sum_{j=1}^{m}\left\langle f,C_{j}(f)K_{S_{z}}(\zeta_{j},z)\right\rangle \right\} \\
 & = & 2i\,\Im\left\{ \sum_{j=1}^{m}C_{j}(f)\overline{\tilde{f}(\zeta_{j})}\right\} 
\end{eqnarray*}
and (\ref{eq:BLFH}) follows from this.
\end{proof}

\section{\label{sec:Fried}The Friedrichs Extension}

In sections \ref{sec:sbdd} and \ref{sec:spH} we introduced a particular
semibounded operator $L$, and we proved that its deficiency indices
are $(1,1)$. In section \ref{sec:Hardy} we explored its relevance
for the study of an harmonic analysis in the Hardy space $\mathscr{H}_{2}$
of the complex disk $\mathbb{D}$. We further proved that $0$ is
an effective lower bound for $L$ (Lemma \ref{lem:lbdd}.) We further
proved that every selfadjoint extension of $L$ has pure point-spectrum,
uniform multiplicity one. Moreover (Corollary \ref{cor:bottom}),
for every $b<0$, we showed that there is a unique selfadjoint extension
$H_{b}$ of $L$ such that $b$ is the smallest eigenvalue of $H_{b}$.

Now, in general, for a semibounded Hermitian operator $L$, there
are two distinguished selfadjoint extensions with the same lower bound,
the Friedrichs extension, and the Krein extension; and in general
they are quite different. For example, in boundary-value problems,
these two selfadjoint extensions correspond to Dirichlet vs Neumann
boundary conditions, respectively.

But for our particular operator $L$ from sections \ref{sec:sbdd}
and \ref{sec:spH}, we show that the two the Friedrichs extension,
and the Krein extension, must coincide. While this may be obtained
from abstract arguments, nonetheless, it is of interest to compute
explicitly this unique extension. Indeed, the abstract characterizations
in the literature (\cite{DS88b,Kr55,AG93,Gru09}) of the two, the
Friedrichs extension and the Krein extension, are given only in very
abstract terms.

Moreover, in general, it is not true that when a semibounded selfadjoint
operator $H$ is restricted, that its Friedrichs extension will coincide
with $H$. But it is true for our particular model operator $H$.
We now turn to the details of the study of the selfadjoint extensions
of $L$.

Let $\mathscr{H}=l^{2}(\mathbb{N}_{0})$, and define the selfadjoint
operator $H$ as in section \ref{sec:sbdd}, 
\begin{equation}
\left(Hx\right)_{k}=k\, x_{k},\: k\in\mathbb{N}_{0}\label{eq:H}
\end{equation}
and
\begin{equation}
\mathscr{D}(H)=\{x\in l^{2}\:\big|\: k\, x_{k}\in l^{2}\}.\label{eq:domH}
\end{equation}
On the dense domain
\begin{equation}
\mathbb{D}_{0}:=\mathscr{D}(L)=\{x\in\mathscr{D}(H)\:\big|\:\sum_{k\in\mathbb{N}_{0}}x_{k}=0\},\label{eq:D0}
\end{equation}
set
\begin{equation}
L:=H\big|_{\mathbb{D}_{0}},\label{eq:L}
\end{equation}
i.e., $L$ is the restriction of $H$ to the dense domain $\mathbb{D}_{0}$
specified in (\ref{eq:D0}). 
\begin{lem}
\label{lem:DLadjoint}For the domain of $L^{*}$ (the adjoint operator),
we have
\begin{equation}
\mathscr{D}(L^{*})=\mathbb{D}_{0}+\mathbb{C}y_{2}+\mathbb{C}y_{3}\label{eq:DLad}
\end{equation}
as a direct sum, where 
\begin{equation}
{\displaystyle \left(y_{2}\right)_{k}=\frac{1}{1+k^{2}}},\quad k\in\mathbb{N}_{0}\label{eq:YY2}
\end{equation}
and 
\begin{equation}
{\displaystyle \left(y_{3}\right)_{k}=\frac{k}{1+k^{2}}},\quad k\in\mathbb{N}_{0},\label{eq:YY3}
\end{equation}
see sect. \ref{sec:sbdd}. \end{lem}
\begin{proof}
The details for this formula (\ref{eq:DLad}) are contained in section
\ref{sec:sbdd}. \end{proof}
\begin{thm}
\label{thm:Friedrichs}The operator $H$ from (\ref{eq:H})-(\ref{eq:domH})
is the Friedrichs extension of $L$. \end{thm}
\begin{proof}
We denote the Friedrichs extension by $H_{Friedrichs}$. On $\mathbb{D}_{0}$
from (\ref{eq:D0}), we define the quadratic form
\begin{equation}
Q(x):=Q_{L}(x)=\left\langle x,Lx\right\rangle _{l^{2}}+\left\Vert x\right\Vert _{2}^{2}=\sum_{k\in\mathbb{N}_{0}}k\left|x_{k}\right|^{2}+\left\Vert x\right\Vert _{2}^{2},\; x\in\mathbb{D}_{0}.\label{eq:Q}
\end{equation}
Hence 
\begin{equation}
Q(x)\geq\left\Vert x\right\Vert _{2}^{2},\:\mbox{for all }x\in\mathbb{D}_{0}.\label{eq:Q-1}
\end{equation}
Let $\mathscr{H}_{Q}$ be the Hilbert completion of the pre-Hilbert
space $(\mathbb{D}_{0},Q)$. Then from (\ref{eq:Q-1}), we see that
$\mathscr{H}_{Q}$ is naturally contained in $l^{2}$, i.e., containment
with a contractive embedding mapping $\mathscr{H}_{Q}\hookrightarrow l^{2}$,
and bounded by $1$; and moreover that $(\sqrt{k}\, x_{k})\in l^{2}$
holds for $x\in\mathscr{H}_{Q}$. From \cite[p. 1240]{DS88a}, we
infer that
\begin{equation}
\mathscr{D}(H_{Friedrichs})=\mathscr{D}(L^{*})\cap\mathscr{H}_{Q}.\label{eq:DFriedrichs}
\end{equation}

In addition to (\ref{eq:DFriedrichs}), we shall also need the following
lemma. 
\begin{lem}
\label{lem:H}The selfadjoint operator $H$ in (\ref{eq:H}) has as
domain
\begin{equation}
\mathscr{D}(H)=\mathbb{D}_{0}+\mathbb{C}y_{2}.\label{eq:domH-1}
\end{equation}
\end{lem}
\begin{proof}
( $\subseteq$ ) Let $x\in l^{2}$ satisfy (\ref{eq:domH}), and set
\begin{equation}
t:=\sum_{k\in\mathbb{N}_{0}}\frac{1}{1+k^{2}}=\frac{1}{2}\left(1+\pi\coth(\pi)\right),\;\mbox{and }\label{t}
\end{equation}
\begin{equation}
s:=\sum_{k\in\mathbb{N}_{0}}x_{k};\;\mbox{then}\label{eq:s}
\end{equation}
\begin{equation}
z:=x-\left(\frac{s}{t}\right)y_{2}\label{eq:z}
\end{equation}
satisfies ${\displaystyle \sum_{k\in\mathbb{N}_{0}}z_{k}=0}$. 

Consequently, $z\in\mathbb{D}_{0}$, and therefore $x\in\mbox{RHS(\ref{eq:domH-1})}$.
Since $y_{2}\in\mathscr{D}(H)$, the other inclusion ``$\supseteq$''
in (\ref{eq:domH-1}) is clear.
\end{proof}

We now continue with the proof of Theorem \ref{thm:Friedrichs}. We
claim that the intersection in (\ref{eq:DFriedrichs}) coincides with
(\ref{eq:domH-1}). Indeed, by (\ref{eq:DLad}), every $x\in\mathscr{D}(L^{*})$
has the form
\begin{equation}
x=z+a\, y_{2}+b\, y_{3}\label{eq:x}
\end{equation}
where $z\in\mathbb{D}_{0}$, and $a,b\in\mathbb{C}$. This decomposition
is unique. If $x$ is also in $\mathscr{H}_{Q}$, it follows from
(\ref{eq:Q})-(\ref{eq:Q-1}) that $x\in l^{2}$, and $\left(\sqrt{k}\, x_{k}\right)\in l^{2}$.
Since both terms $z$ and $a\, y_{2}$ satisfy the last condition,
we conclude that the last term $b\, y_{3}$ in (\ref{eq:x}) will
as well. But since ${\displaystyle \sqrt{k}\left(y_{3}\right)_{k}=\frac{k^{3/2}}{1+k^{2}}\notin l^{2}}$,
we conclude that $b=0$. Hence (\ref{eq:x}) reduces to $x=z+a\, y_{2}$;
and this is in $\mathscr{D}(H)$ by Lemma \ref{lem:Hdom}. We proved
that 
\[
\mathscr{D}(H)=\mathscr{D}(H_{Friedrichs})
\]
and therefore ${\displaystyle H=H_{Friedrichs}}$. 

\end{proof}
\begin{cor}
\label{cor:FKext}Let the Hermitian operator $L$ and its domain $\mathscr{D}(L)=\mathbb{D}_{0}$
be as in Theorem \ref{thm:Friedrichs}, see also eq (\ref{eq:domL});
then the Friedrichs and the Krein extensions of $L$ coincide as selfadjoint
operators in $l^{2}(\mathbb{N}_{0})$.\end{cor}
\begin{proof}
For the definition of the Krein extension, see e.g., \cite[sect. 107, p 367-8]{AG93}.
To understand it, it is useful to introduce the following inner product
on $\mathscr{D}(L^{*})$ in the general case of the von Neumann decomposition
(\ref{eq:vNdep}). For $f,g\in\mathscr{D}(L^{*})$ set
\begin{equation}
\left\langle f,g\right\rangle _{*}=\left\langle f,g\right\rangle +\left\langle L^{*}f,L^{*}g\right\rangle .\label{eq:Ginner}
\end{equation}
In this inner product $\left\langle \cdot,\cdot\right\rangle _{*}$,
the three subspaces $\mathscr{D}(L)$, and $\mathscr{D}_{\pm}$ in
(\ref{eq:vNdep}) are mutually orthogonal, and, for $f_{\pm}\in\mathscr{D}_{\pm}$,
we have
\begin{equation}
\left\Vert f_{\pm}\right\Vert _{*}^{2}=\left\langle f_{\pm},f_{\pm}\right\rangle _{*}=2\left\Vert f_{\pm}\right\Vert ^{2}.\label{eq:ndefv}
\end{equation}

Return now to the particular example with $L$ in $l^{2}(\mathbb{N}_{0})$
as described in the corollary, and in Theorem \ref{thm:Friedrichs},
we introduce the standard ONB in $l^{2}(\mathbb{N}_{0})$, $\{e_{n}\:\big|\: n\in\mathbb{N}_{0}\}$
where $e_{n}(k)=\delta_{n,k}$ for all $n,k\in\mathbb{N}_{0}$, and
$\delta_{n,k}$ denoting the Kronecker-delta.

A computation shows that $e_{0}\in\mathscr{D}(L^{*})$, and $L^{*}e_{0}=0$. 

Clearly, $e_{0}\notin\mathscr{D}(L)$. Moreover,
\begin{equation}
\left\langle e_{0},x_{\pm}\right\rangle _{*}=\pm i\label{eq:idefv}
\end{equation}
where ${\displaystyle (x_{\pm})_{n}=\frac{1}{n\mp i}}$ are the defect-vectors
from (\ref{eq:defv}). 

We now define a selfadjoint extension $K$ of $L$ as follows
\begin{equation}
\mathscr{D}(K)=\mathscr{D}(L)+\mathbb{C}\, e_{0}\label{eq:domK}
\end{equation}
and
\begin{equation}
K(\varphi+c\, e_{0})=L\varphi\label{eq:defK}
\end{equation}
for all $\varphi\in\mathscr{D}(L)$ and all $c\in\mathbb{C}$. 

We see from (\ref{eq:gInner}), (\ref{eq:ndefv}), and (\ref{eq:idefv})
that the co-dimension of $\mathscr{D}(L)$ in $\mathscr{D}(K)$ is
one. From (\ref{eq:defK}), we get
\begin{eqnarray*}
\left\langle \varphi+c\, e_{0},K(\varphi+c\, e_{0})\right\rangle  & = & \left\langle \varphi+c\, e_{0},L\varphi\right\rangle \\
 & = & \left\langle \varphi,L\varphi\right\rangle +\overline{c}\left\langle e_{0},L\varphi\right\rangle \\
 & = & \left\langle \varphi,L\varphi\right\rangle +\overline{c}\left\langle L^{*}e_{0},\varphi\right\rangle \\
 & = & \left\langle \varphi,L\varphi\right\rangle \geq0.
\end{eqnarray*}
Since $L$ has indices $(1,1)$, we conclude that the operator $K$defined
in (\ref{eq:domK})-(\ref{eq:defK}) is selfadjoint and semibounded,
$K\geq0$; moreover $Ke_{0}=0$, from (\ref{eq:defK}). 

We claim that
\begin{equation}
Ke_{n}=n\, e_{n},\quad\forall n\in\mathbb{N}_{0}.\label{eq:eigK}
\end{equation}
This holds for $n=0$. If $n>0$, then $e_{n}-e_{0}\in\mathscr{D}(L)$,
and 
\begin{equation}
L(e_{n}-e_{0})=n\, e_{n}.\label{eq:Len}
\end{equation}
Hence 
\begin{eqnarray*}
Ke_{n} & = & K(e_{n}-e_{0}+e_{0})\\
 & = & L(e_{n}-e_{0})\,\,\,\,\,\,\,\,\,\,\,\,\,\,\,\,\,\,\mbox{by }(\ref{eq:defK})\\
 & = & n\, e_{n}\,\,\,\,\,\,\,\,\,\,\,\,\,\,\,\,\,\,\,\,\,\,\,\,\,\,\,\,\,\,\,\,\,\,\,\mbox{by }(\ref{eq:Len})
\end{eqnarray*}
which is the desired conclusion (\ref{eq:eigK}).\end{proof}
\begin{cor}
Among all the selfadjoint extensions of $L$ in (\ref{eq:D0}) and
(\ref{eq:L}), the Friedrichs extension is the only one having its
spectrum contained in $[0,\infty)$. \end{cor}
\begin{proof}
First note from Theorem \ref{thm:Friedrichs} that $\mbox{spectrum}(H_{Fridrichs})=\mathbb{N}_{0}$.
But if $H_{\theta}$ is one of the other (different) selfadjoint extension
of $L$ in the von Neumann classification, we found in Theorem \ref{thm:Spectrum-L_t}
that the spectrum has the form
\begin{equation}
\lambda_{0}<0,\quad n-1<\lambda_{n}<n,\quad n=1,2,\ldots\label{eq:lambda}
\end{equation}
with $\lambda_{n}=\lambda_{n}(\theta)$ depending on $\theta$ in
the von Neumann classification. The conclusion follows.\end{proof}
\begin{cor}
Let $L$ be the semibounded operator in (\ref{eq:D0}) and (\ref{eq:L})
with deficiency indices $(1,1)$. Then for every selfadjoint extension
$H_{\theta}$ in the von Neumann classification we have resolvent
operator
\begin{equation}
R(\xi,H_{\theta})=\left(\xi I-H_{\theta}\right)^{-1},\quad\Im\xi\neq0\label{eq:resH}
\end{equation}
in the Hilbert-Schmidt class.\end{cor}
\begin{proof}
From the eigenvalue list in (\ref{eq:lambda}) we note that
\[
\sum_{n}\left|\left(\xi-\lambda_{n}\right)^{-1}\right|^{2}<\infty.
\]

\end{proof}

\subsection{The Domain of the Adjoint Operator}

In Lemma \ref{lem:DLadjoint} and Theorem \ref{thm:Friedrichs}, we
considered $\mathscr{H}=l^{2}(\mathbb{N}_{0})$, and a Hermitian symmetric
operator $L$ with dense domain
\begin{equation}
\mathscr{D}(L)=\left\{ x=\left(x_{k}\right)_{k\in\mathbb{N}_{0}}\in l^{2}\:\Big|\:\left(kx_{k}\right)\in l^{2},\mbox{ and }\sum_{k\in\mathbb{N}_{0}}x_{k}=0\right\} .\label{eq:domLL}
\end{equation}
In consideration of the selfadjoint extensions of $L$, we used the
domain $\mathscr{D}(L^{*})$ and the two vectors $x_{\pm}$ (see (\ref{eq:defv})),
and the pair $y_{2},y_{3}$ (see (\ref{eq:YY2}) \& (\ref{eq:YY3})),
and 
\begin{equation}
z=\left(z_{k}\right)_{k\in\mathbb{N}_{0}},\quad z_{k}=\frac{1}{1+k}.\label{eq:defz}
\end{equation}
These vectors all lie in $\mathscr{D}(L^{*})\backslash\mathscr{D}(L)$;
i.e., they are in the domain of the larger of the two operators, $L\subset L^{*}$.
Recall 
\begin{equation}
\mathscr{D}(L^{*})/\mathscr{D}(L)=2.\label{eq:Dadjointmod}
\end{equation}

The fact that special choices are needed results from the following:
\begin{prop}
Let $n\in\mathbb{N}_{0}$, then the basis vectors $e_{n}$ ($e_{n}=\left(\delta_{n,k}\right)$,
$k\in\mathbb{N}_{0}$) are in $\mathscr{D}(L^{*})$; and
\[
L^{*}e_{n}=n\, e_{n}.
\]
Moreover, the Friedrichs extension $H$ of $L$ is the unique selfadjoint
extension s.t. 
\begin{equation}
\left\{ e_{n}\:\big|\: n\in\mathbb{N}_{0}\right\} \subseteq\mathscr{D}(H).\label{eq:enD}
\end{equation}
\end{prop}
\begin{proof}
Let $y_{2},y_{3}$, and $t:=\sum_{k\in\mathbb{N}_{0}}\frac{1}{1+k^{2}}$
be as in (\ref{eq:YY2}), (\ref{eq:YY3}), and (\ref{t}). Then set
\begin{equation}
e_{n}=\underset{\in\mathscr{D}(L)}{\underbrace{\left(e_{n}-\frac{1}{t}y_{2}\right)}}+\underset{\in\mathscr{D}(L^{*})}{\underbrace{\frac{1}{t}y_{2}}}\label{eq:endecomp}
\end{equation}
Since $\mathscr{D}(L)\subset\mathscr{D}(L^{*})$, it follows that
$e_{n}\in\mathscr{D}(L^{*})$.

Fix $n$. Since ${\displaystyle \varphi=e_{n}-\frac{1}{t}y_{2}}$
satisfies
\begin{equation}
\varphi_{k}=\begin{cases}
{\displaystyle -\frac{1}{t\left(1+k^{2}\right)}} & \mbox{if }k\neq n\\
\\
{\displaystyle 1-\frac{1}{t\left(1+n^{2}\right)}} & \mbox{if }k=n,
\end{cases}\label{eq:varphik}
\end{equation}
it follows that ${\displaystyle \sum_{k\in\mathbb{N}_{0}}\varphi_{k}=0}$,
and so $\varphi\in\mathscr{D}(L)$. As a result, using $L\subset L^{*}$,
we obtain the following:
\begin{equation}
L^{*}e_{n}=L\varphi+\frac{1}{t}L^{*}y_{2}.\label{eq:Ladj}
\end{equation}
Now use $L^{*}y_{2}=y_{3}$ (see section \ref{sec:sbdd}) in eq. (\ref{eq:Ladj})
to compute $L^{*}e_{n}$ as follows:

If $n,k\in\mathbb{N}_{0}$, from (\ref{eq:varphik}) we obtain:
\[
\left(L^{*}e_{n}\right)(k)=\begin{cases}
{\displaystyle -\frac{k}{t(1+k^{2})}+\frac{k}{t(1+k^{2})}}=0 & \quad\mbox{if }k\neq n\\
\\
{\displaystyle n\left(1-\frac{1}{t(1+n^{2})}\right)+\frac{n}{t(1+n^{2})}=n} & \quad\mbox{if }k=n.
\end{cases}
\]
Hence, $\left(L^{*}e_{n}\right)(k)=n\delta_{n,k}=ne_{n}(k)$; or $L^{*}e_{n}=n\, e_{n}$
as asserted.
\end{proof}
In Lemma \ref{lem:vNext} and (\ref{eq:Ginner}), we introduced the
graph-norm on $\mathscr{D}(L^{*})$, i.e., 
\begin{equation}
\left\Vert f\right\Vert _{*}^{2}=\left\Vert L^{*}f\right\Vert ^{2}+\left\Vert f\right\Vert ^{2},\quad f\in\mathscr{D}(L^{*}).\label{eq:GLnorm}
\end{equation}
The selfadjoint extensions $H$ of $L$ satisfy
\begin{equation}
L\subseteq H\subseteq L^{*}.\label{eq:LHLad}
\end{equation}
Set 
\begin{equation}
\left\Vert f\right\Vert _{H}^{2}:=\left\Vert Hf\right\Vert ^{2}+\left\Vert f\right\Vert ^{2},\quad f\in\mathscr{D}(H).\label{eq:GHnorm}
\end{equation}

We now have the following result:
\begin{prop}
Let $L$ and $H$ be as in (\ref{eq:H}) and (\ref{eq:L}). Then $span\{e_{n}\}_{n\in\mathbb{N}}$
is dense in $\mathscr{D}(H)$ w.r.t. the graph-norm $\left\Vert \cdot\right\Vert _{H}$,
but not dense in $\mathscr{D}(L^{*})$ w.r.t. the other graph-norm
$\left\Vert \cdot\right\Vert _{*}$.\end{prop}
\begin{proof}
The first assertion in the proposition is clear (we say that $span\{e_{n}\}$
is a \textbf{\emph{core}} for $H$.).

We now show that $\rho_{3}\sim y_{3}$ and $\rho_{2}\sim y_{2}$ are
mutually orthogonal in $\mathscr{D}(L^{*})$ relative to the inner
product $\left\langle \cdot,\cdot\right\rangle _{*}$ from (\ref{eq:Ginner})
and (\ref{eq:GLnorm}).

Indeed, $\rho_{3}\in\mathscr{D}(L^{*})$, and $L^{*}\rho_{3}=-\rho_{2}$
(see (\ref{eq:rhoy2}) \& (\ref{eq:rhoy3})). Hence 
\begin{eqnarray*}
\left\langle \rho_{2},\rho_{3}\right\rangle _{*} & = & \left\langle \rho_{2},\rho_{3}\right\rangle +\left\langle L^{*}\rho_{2},L^{*}\rho_{3}\right\rangle \\
 & = & \left\langle \rho_{2},\rho_{3}\right\rangle +\left\langle \rho_{3},-\rho_{2}\right\rangle =0.
\end{eqnarray*}

Using the details from the proof of Corollary \ref{cor:FKext}, we
conclude that 
\begin{equation}
\rho_{3}\in\mathscr{D}(L^{*})\ominus\mathscr{D}(H)\label{eq:rho3do}
\end{equation}
where $\ominus$ in (\ref{eq:rho3do}) refers to $\left\langle \cdot,\cdot\right\rangle _{*}$.\end{proof}
\begin{cor}
In the computation of $L_{F}^{*}$ in sect. \ref{sec:Hardy}, we must
have a term not like ${\displaystyle z\frac{d}{dz}}$; see Theorem
\ref{thm:LFad}.
\end{cor}

\section{\label{sec:hd}Higher Dimensions}

Now there is a higher dimensional version of our analysis in section
\ref{sec:Hardy} above (for the Hardy-Hilbert space $\mathscr{H}_{2}$).
This is the Arveson-Drury space $\mathscr{H}_{d}^{(AD)}$, $d>1$.
While the case $d>1$ does have a number of striking parallels with
$d=1$ from section \ref{sec:Hardy}, there are some key differences
as well.

The reason for the parallels is that the reproducing kernel, the Szegö
kernel (\ref{eq:ZKK}), extends from one complex dimension to $d>1$
almost verbatim. This is a key point of the Arveson-Drury analysis
\cite{Arv98,Dru78}.

In section \ref{sec:Hardy}, for $d=1$, we showed that the study
of von Neumann boundary theory for Hermitian operators translates
into a geometric analysis on the boundary of the disk $\mathbb{D}$
in one complex dimension, so on the circle $\partial\mathbb{D}$. 

It is the purpose of this section to show that multivariable operator
theory is more subtle. A main reason for this is a negative result
by Arveson \cite[Coroll 2]{Arv98} stating that the Hilbert norm in
$\mathscr{H}_{d}^{(AD)}$, $d>1$, cannot be represented by a Borel
measure on $\mathbb{C}^{d}$. So, in higher dimension, the question
of \textquotedblleft{}geometric boundary\textquotedblright{} is much
more subtle. Contrast this with eq (\ref{eq:H2norm}) and Lemma \ref{lem:H2}
above. 

The Szegö kernel $K_{w}(z)=\left(1-\overline{w}z\right)^{-1}$ (see
(\ref{eq:ZKK})) in higher dimensions, i.e., $\mathbb{C}^{d}$ is
called the Arveson-Drury kernel, see \cite{Arv00,Arv98,Dru78}. 

Let $z=\left(z_{1},.\ldots,z_{d}\right)\in\mathbb{C}^{d}$, $\alpha=\left(\alpha_{1},\ldots,\alpha_{d}\right)\in\mathbb{N}_{0}^{d}$,
i.e., $\alpha_{i}\in\mathbb{N}_{0}$; and set
\begin{equation}
\begin{cases}
\left\langle w,z\right\rangle =\sum_{j=1}^{d}\overline{w}_{j}z_{j},\\
\\
\left\Vert z\right\Vert ^{2}=\sum_{j=1}^{d}\left|z_{j}\right|^{2},
\end{cases}\label{eq:dd}
\end{equation}
and 
\begin{equation}
K_{w}^{(AD)}(z):=\frac{1}{1-\left\langle w,z\right\rangle }.\label{eq:dker}
\end{equation}
Then the corresponding reproducing kernel Hilbert space (RKHS) is
called the Arveson-Drury Hilbert space. It is a Hilbert space of analytic
functions in 
\begin{equation}
z=(z_{1},.\ldots,z_{d})\in B_{d}=\left\{ z\in\mathbb{C}^{d}\:\big|\:\left\Vert z\right\Vert <1\right\} ,
\end{equation}
(see (\ref{eq:dd}).)

Since, for $d=1$, 
\begin{eqnarray*}
z\frac{d}{dz}K_{w}(z) & = & \frac{\overline{w}z}{\left(1-\overline{w}z\right)^{2}}=\overline{w}zK_{w}(z)^{2},\\
z_{j}\frac{\partial}{\partial z_{j}}K_{w}(z) & = & \overline{w_{j}}z_{j}K_{w}(z)^{2};\quad\mbox{and}
\end{eqnarray*}
in higher dimensions, 
\begin{eqnarray}
\sum_{j=1}^{d}z_{j}\frac{\partial}{\partial z_{j}}K_{w}(z) & = & \left\langle w,z\right\rangle K_{w}(z)^{2}\nonumber \\
 & = & =-K_{w}(z)+K_{w}(z)^{2};
\end{eqnarray}
it is natural to view $\mathscr{H}_{d}^{(AD)}$ as a direct extension
of the Hardy space $\mathscr{H}_{2}$ from section \ref{sec:Hardy}
above.

But $\mathscr{H}_{d}^{(AD)}$ is also a symmetric Fock space over
the Hilbert space $\mathbb{C}^{d}$, i.e., 
\[
\mathscr{H}_{d}^{(AD)}=\mbox{Fock}_{symm}(\mathbb{C}^{d}),\quad\mbox{see }(\ref{eq:dd})\:\mbox{and }[\mbox{Arv98}].
\]

Set ${\displaystyle H_{j}:=z_{j}\frac{\partial}{\partial z_{j}}}$,
and ${\displaystyle H:=\sum_{j=1}^{d}H_{j}=\sum_{j=1}^{d}z_{j}\frac{\partial}{\partial z_{j}}}$
= the Arveson-Dirac operator, and 
\begin{equation}
\mathscr{D}(H)=\left\{ f\in\mathscr{H}_{d}^{(AD)}\:\big|\: Hf\in\mathscr{H}_{d}^{(AD)}\right\} .\label{eq:domHad}
\end{equation}
We know \cite{Arv00} that $\left\{ H_{j}\right\} $ is a commuting
family of selfadjoint operators in $\mathscr{H}_{d}^{(AD)}$.

For $\alpha=\left(\alpha_{1},\ldots,\alpha_{d}\right)\in\mathbb{N}_{0}^{d}$,
set ${\displaystyle n=\left|\alpha\right|=\sum_{j=1}^{d}\alpha_{j}}$;
then ${\displaystyle \binom{n}{\alpha}=\frac{n!}{\alpha_{1}!\alpha_{2}!\cdots\alpha_{d}!}}$
are the multinomial coefficients. It follows that 
\begin{equation}
\left\Vert z_{1}^{\alpha_{1}}\cdots z_{d}^{\alpha_{d}}\right\Vert _{\mathscr{H}_{d}^{(AD)}}^{2}=\binom{n}{\alpha}^{-1}.\label{eq:mul1}
\end{equation}
Hence, if 
\begin{equation}
f(z)=\sum_{\alpha\in\mathbb{N}_{0}^{d}}c(\alpha)z^{\alpha},\label{eq:mul2}
\end{equation}
$z^{\alpha}:=z_{1}^{\alpha_{1}}z_{2}^{\alpha_{2}}\cdots z_{d}^{\alpha_{d}}$,
is in $\mathscr{H}_{d}^{(AD)}$, then 
\begin{equation}
\left\Vert f\right\Vert _{\mathscr{H}_{d}^{(AD)}}^{2}=\sum_{n=0}^{\infty}\sum_{\left|\alpha\right|=n}\binom{n}{\alpha}^{-1}\left|c(\alpha)\right|^{2}.\label{eq:mul3}
\end{equation}

In $\mathscr{H}_{d}^{(AD)}$, $d>1$, the analogue of $\mathscr{D}(L)=\mathbb{D}_{0}$,
in Lemma \ref{lem:Hardy}, and Lemma \ref{lem:densedom}, is 
\begin{equation}
\mathbb{D}_{0}=\left\{ f=f_{c}\in\mathscr{D}(H)\:\big|\:\sum_{\alpha\in\mathbb{N}_{0}^{d}}c(\alpha)=0\right\} .\label{eq:ADdom}
\end{equation}
It is a dense linear subspace in $\mathscr{H}_{d}^{(AD)}$. 
\begin{thm}
\label{thm:zdzmd}The operator family $\left\{ H_{j}\:\big|\:1\leq j\leq d\right\} $
is a commuting family and each $H_{j}$ is essentially selfadjoint
on $\mathbb{D}_{0}$. \end{thm}
\begin{rem}
Comparing the theorem with Lemma \ref{lem:Hardy}, and Corollary \ref{cor:finite},
we note that the unitary one-parameter groups acting on $\mathscr{H}_{d}^{(AD)}$,
$d>1$, are more stable than is the case for $d=1$, (where the unitary
one-parameter groups are acting in $\mathscr{H}_{2}$.) In short,
\textquotedblleft{}unitary motion\textquotedblright{} in $\mathbb{C}^{d}$
for $d>1$ does not get trapped at \textquotedblleft{}points.\textquotedblright{} \end{rem}
\begin{proof}[Proof of Theorem \ref{thm:zdzmd}]
To prove the theorem, we may use a result of Nelson \cite{Nel59}
showing now instead that 
\begin{equation}
\sum_{j=1}^{d}H_{j}^{2}
\end{equation}
is essentially selfadjoint on $\mathbb{D}_{0}$. 

We must therefore show that, if $g\in\mathscr{H}_{d}^{(AD)}$, and
\begin{equation}
\left\langle g,\left(I+\sum_{j=1}^{d}H_{j}^{2}\right)f_{c}\right\rangle _{\mathscr{H}_{d}^{(AD)}}=0\label{eq:Nel}
\end{equation}
for all $f_{c}\in\mathbb{D}_{0}$, then it follows that $g=0$. 

Hence, suppose (\ref{eq:Nel}) holds for some $g\in\mathscr{H}_{d}^{(AD)}$,
${\displaystyle g(z)=\sum_{\alpha\in\mathbb{N}_{0}}b(\alpha)z^{\alpha}}$,
then
\begin{equation}
\sum_{\alpha\in\mathbb{N}_{0}^{d}}\binom{\left|\alpha\right|}{\alpha}^{-1}\overline{b(\alpha)}\left(1+\sum_{j=1}^{d}\alpha_{j}^{2}\right)c(\alpha)=0\label{eq:Nel1}
\end{equation}
for all $f_{c}\in\mathbb{D}_{0}$.

Consider $\alpha\in\mathbb{N}_{0}^{d}\backslash(0)$ fixed, and set:
\begin{equation}
c(\beta)=\begin{cases}
-1 & \beta=\left(0,0,\ldots,0\right)\\
1 & \beta=\alpha\\
0 & \beta\in\mathbb{N}_{0}^{d}\backslash\{0,\alpha\},
\end{cases}
\end{equation}
and set $n=\left|\alpha\right|$. Then, by (\ref{eq:Nel1}), we have
\begin{equation}
b(\alpha)=b(0)\frac{{\displaystyle \binom{n}{\alpha}}}{{\displaystyle 1+\sum_{j=1}^{d}\alpha_{j}^{2}}},\label{eq:Nel3}
\end{equation}
and therefore (by (\ref{eq:mul1}) \& (\ref{eq:mul3})), that 
\begin{equation}
\left\Vert g\right\Vert _{\mathscr{H}_{d}^{(AD)}}^{2}=\left|b(0)\right|^{2}\sum_{\alpha\in\mathbb{N}_{0}^{d}}\frac{{\displaystyle \binom{\left|\alpha\right|}{\alpha}}}{\left({\displaystyle 1+\sum_{j=1}^{d}\alpha_{j}^{2}}\right)^{2}}.\label{eq:Nel4}
\end{equation}
Since ${\displaystyle \sum_{\left|\alpha\right|=n}\binom{n}{\alpha}=d^{n}}$,
we conclude from (\ref{eq:Nel3}) \& (\ref{eq:Nel4}) that $b(0)=0$,
and therefore that $g=0$ in $\mathscr{H}_{d}^{(AD)}$.

Nelson's theorem implies that $\sum_{j=1}^{d}H_{j}^{2}$ is essentially
selfadjoint on $\mathbb{D}_{0}$ (in (\ref{eq:ADdom})), and the desired
conclusion follows.
\end{proof}

\section*{Acknowledgments}

The co-authors, some or all, had helpful conversations with many colleagues,
and wish to thank especially Professors Daniel Alpay, Ilwoo Cho, Dorin
Dutkay, Alex Iosevich, Paul Muhly, Yang Wang, and Qingbo Huang. And
going back in time, Bent Fuglede (PJ, SP), and Robert T. Powers, Ralph
S. Phillips, Derek Robinson (PJ). The first named author was supported
in part by the National Science Foundation, via a VIGRE grant.

\bibliographystyle{amsalpha}
\bibliography{number4}

\end{document}